\documentclass[12pt,a4paper]{uiucthesis}
\usepackage[utf8]{inputenc}
\usepackage[T1]{fontenc}
\usepackage[french]{babel}
\usepackage{arabtex}
\usepackage{amsthm,amssymb,amsmath,amsopn,amsfonts,mathrsfs,amsbsy,amscd}
\usepackage[notext,frenchstyle,light]{kpfonts}
\usepackage{lipsum}
\usepackage{ntimes}
\usepackage{type1ec}
\usepackage{epigraph}
\usepackage{graphics}
\usepackage{url}
\usepackage{fancyhdr}
\usepackage{fancybox}
\usepackage{shadow}
\usepackage{lettrine}
\usepackage{makeidx}
\makeindex
\usepackage{ifpdf}
\ifpdf
\usepackage[pdftex=true,hyperindex=true,colorlinks=false,pdfborder=0]{hyperref}
\else
\fi

\DeclareMathAlphabet{\mathsfsl}{OT1}{cmss}{m}{sl}
\DeclareMathOperator*{\tr}{tr} 
\DeclareMathOperator*{\Ad}{Ad} \DeclareMathOperator*{\ad}{ad}
 
\DeclareMathOperator*{\im}{Im}\DeclareMathOperator*{\gl}{g\ell}
\DeclareMathOperator*{\Div}{div}
\DeclareMathOperator*{\id}{id}
\DeclareMathOperator*{\rang}{rang}

\newcommand{\esp}{\quad\mbox{et}\quad}
\newcommand{\prs}{\langle\;,\;\rangle}
\newcommand{\too}{\longrightarrow}
\newcommand{\om}{\omega}

\newcommand{\G}{{\cal G}}

\newcommand{\h}{{\cal H}}
\newcommand{\Li}{{\cal L}}
\newcommand{\B}{{\cal B}}

\newcommand{\Om}{\Omega}

\newcommand{\wi}{\widetilde}
\newcommand{\al}{\alpha}
\newcommand{\be}{\beta}
\newcommand{\ga}{\gamma}

\font\bb=msbm10

\def\B{\hbox{\bb B}}
\def\R{\hbox{\bb R}}

\newtheorem{theorem}{Théorème}
\newtheorem{corollary}{Corollaire}
\newtheorem{definition}{Définition}
\newtheorem*{example}{Exemples}

\newtheorem{lemma}{Lemme}
\newtheorem*{problem}{Problème}

\newtheorem{proposition}{Proposition}
\newtheorem*{remark}{Remarque}
\newtheorem*{remarks}{Remarques}

\makeatletter
\def\thickhrulefill{\leavevmode \leaders \hrule height 0.5ex \hfill \kern \z@}
\def\@makechapterhead#1{%
  \vspace*{0\p@}%
  {\parindent \z@ \centering \reset@font
        \thickhrulefill\quad
        \scshape \@chapapp{} \thechapter
        \quad \thickhrulefill
        \par\nobreak
        \vspace*{10\p@}%
        \interlinepenalty\@M
        \hrule
        \vspace*{10\p@}%
        \Huge #1\par\nobreak
        \par
        \vspace*{10\p@}%
        \hrule
    \vskip 100\p@
  }}
\begin{document}
\pagenumbering{arabic}
\begin{center}
\large{Remerciements}
\end{center}
\qquad Ma thèse s'est déroulée à moitié à l'Université Kasdi Marbah de Ouargla et le Laboratoire Émile Picard de l’Université Paul Sabatier de Toulouse. Tout d'abord, je remercie vivement Monsieur Philippe Monnier de m'avoir accueilli dans son laboratoire afin que je puisse y effectuer mon travail de thèse. Mon séjour y a été des plus agréables et j'ai pu y bénéficier d'une ambiance stimulante. Ce double statut est le résultat d'un programme national de bourses algéro-françaises et du fait que Mohamed Belkhelfa et Mohamed Boucetta ont tous deux accepté de diriger ma thèse, durant ces cinq années, avec patience et dévouement. Je dois dire que cette situation m'a permis de bénéficier d'un encadrement de grande qualité et je tiens à leur exprimer ma plus sincère reconnaissance et toute ma sympathie. J'ai beaucoup apprécié Monsieur M. Belkhelfa pour son enthousiasme scientifique et Monsieur M. Boucetta pour son attention et son souci de rigueur permanent. Je les remercie pour l'expérience et les savoirs qu'ils ont su partager avec moi et pour m'avoir initié au vaste domaine qu'est la géométrie de Poisson.

     Je tiens à remercier Monsieur M. Benchohra d'avoir accepté de présider le jury et pour l'attention qui a porté sur mon travail et pour ce qu'il a fait, pour le bon déroulement de la soutenance. Je tiens également à remercier Monsieur M. Djaa et Monsieur M. Lakmeche d'avoir accepté la lourde tache d'examinateur, de m'avoir accordé leur temps et leur attention, ainsi que pour les remarques qu'ils ont pu me faire et qui m'ont permis de clarifier certains points de ce manuscrit.
     
     Je voudrais exprimer ma gratitude et ma vive sympathie à toutes les personnes qui m'ont permis d'avancer dans ma perception de certains problèmes : André Diatta, Nicolas Ciccoli, Paolo Bertozzini, Charles-Michel Marle, Pierre Lecomte, Dominique Manchon, Christian Ohn, Mohamed Sbai,... et bien d'autres encore, enseignants et chercheurs, dont la liste serait trop longue à énumérer ici.
     
	Enfin c'est avec joie que je remercie ma famille et mes amis pour leur soutien et encouragements.
Un grand merci à ma femme, à qui je dédie cette thèse.
\newpage
\newpage
\begin{center}
\large{\bf Connexions contravariantes sur les groupes de Lie-Poisson}
\end{center}
\begin{center}
\large{Résumé}
\end{center}
Ce travail est consacré à l'étude d'une classe de groupes de Lie-Poisson à métriques invariantes à gauche. Plus précisément, les triplets $(G,\pi,<\,,\,>)$, où $G$ est un groupe de Lie simplement connexe, $\pi$ est un tenseur de Poisson multiplicatif et $<\,,\,>$ est une métrique riemannienne invariante à gauche telles que les conditions de Hawkins sont satisfaites. Les conditions de Hawkins sont des conditions nécessaires, pour la déformation de l'algèbre graduée des formes différentielles d'une variété riemannienne. Ces conditions proviennent de la déformation non commutative du triplet spectral qui décrit la variété.

Le résultat principal de la thèse est l'équivalence entre le problème géométrique de classification des groupes de Lie-Poisson riemanniens qui vérifient les conditions de Hawkins et le problème algébrique de classification des structures de bigèbres de Lie sur les algèbres de Milnor qui vérifient certaines conditions.

Exploitant le fait que les structures de bigèbres de Lie sur les algèbres de Milnor, dans certaines situations, peuvent être calculées, On a déterminé les groupes de Lie-Poisson riemanniens satisfaisant les conditions de Hawkins dans le cas linéaire, dans le cas de Heisenberg, dans le cas triangulaire et en petites dimensions (jusqu'à la dimension $5$). Le cas général reste un problème ouvert.
\paragraph{Mots clefs :} Groupes de Lie-Poisson, connexions contravariantes, métacourbure.
\bigskip
\begin{center}
\rule{8cm}{.6pt}
\end{center}
\bigskip
\begin{center}
\large{\bf Contravariant connections on Poisson-Lie groups}
\end{center}
\begin{center}
\large{Abstract}
\end{center}
 This work is devoted to the study of a class of Poisson-Lie groups endowed with left invariant metrics. The triples $(G,\pi,<\,,\,>)$ are considered, where $G$ is a simply connected Lie group, $\pi$ is a multiplicative Poisson tensor and $<\,,\,>$ is a left invariant riemannian metric such that Hawkins conditions are satisfied. Hawkins conditions are necessary conditions for the deformation of the graded algebra of differential forms of a riemannian manifold. These conditions come from the deformation of the noncommutative spectral triple describing the manifold. 
The main result of this thesis is the equivalence between, on one hand, the geometric problem of classifying riemannian Poisson-Lie groups that satisfy the conditions of Hawkins and, secondly, the problem of classifying algebraic structures of Lie bialgebras on Milnor algebras that satisfy certain conditions. 
Exploiting the fact that the structures of Lie bialgebras on Milnor algebras, in certain situations, can be calculated, we determine riemannian Poisson-Lie groups that satisfy Hawkins in the linear case, in the case of Heisenberg in the triangular case and in low dimensions (up to dimension 5). The general case remains an open problem.
\paragraph{Keywords :} Poisson-Lie groups, contravariant connections, metacurvature.
\bigskip
\begin{center}
\rule{8cm}{.6pt}
\end{center}
\bigskip
\transfalse
\begin{arabtext}
\begin{center}
\large{\bf AlmtrAb.tAt kntrAfrywnt `l_A zmr ly pwAswn}
\end{center}
\begin{center}
\large{m--_ht.sr}
\end{center}
nt`r.d fy h_dh al'a.trw.hT 'il_A zmr ly pwAswn al--mzwdT b--mtryT ry--mAn al--lAmt.gyrT mn alysAr. ndrs al_tlA_tyAt $(G,\pi,<\,,\,>)$ Alty t--.hq--q ^srw.t t^swyh hAwkynz .hy_t $G$ zmrT mtrAb.tT b--bsA.tT, $\pi$ mwtr pwAswn ^gdA'iy w $<\,,\,>$ mtryT ry--mAn .gyr mt.gyrT mn alysAr.\\
Alnty^gT Alr'iysyT lh_dh al'a.trw.hT hw AltkAfw' byn Alms'alT AlhndsyT lhAwkynz w Alms'alT al^gbryT lm.dA`fT ^gbr ly `l_A bn_A ^gbr mlnwr. qmnA bdrAsT kAmlT lzmr pwAswn Al_h.tyT, zmr hAyznbr.g w kl zmr ly pwAswn _dAt Al'ab`Ad Al.s.gyrT, 'il_A .gAyT Alb`d Al_hAms. Al.hAlT Al`AmT lA tzAl t--m_tl ms'alT mftw.hT.\\
{\bf AlklmAt AldAlT $:$} zmr ly pwAswn, AlmtrAb.tAt kntrAfrywnt, mytA-'in--.hnA'.
\end{arabtext}
\newpage
\tableofcontents
\newpage
\thispagestyle{empty}
\begin{center}
\Large{Introduction}
\end{center}
\rule{.6cm}{0cm}
{\fontfamily{cmss}\selectfont
Le problème mathématique traité dans cette thèse trouve son origine dans un travail récent de Hawkins (\cite{haw1} et \cite{haw2}). Animé par des motivations physiques, Hawkins a étudié les déformations non commutatives de l'algèbre différentielle d'une variété différentielle. Plus précisément, étant donnée une variété différentielle $M$, l'espace des formes différentielles $(\Om^*(M),\wedge,d)$ muni du produit extérieur $\wedge$ et de la différentielle $d$ est une algèbre différentielle graduée, associative et commutative (au sens gradué). Il est connu que cette algèbre encode d'une certaine manière la géométrie de $M$\footnote{En fait, c'est l'idée de base et l'origine de la géométrie non commutative à la Connes.}. Une manière de chercher d'autres géométries (non commutatives) est de déformer cette algèbre.
Une déformation non commutative de $(\Om^*(M),\wedge,d)$ est la donnée d'une algèbre différentielle graduée $(\mathcal{A},\cdot,\delta)$, non nécessairement commutative, qui soit une extension de $(\Om^*(M),\wedge,d)$, c'est-à-dire, qu'il existe un homomorphisme (surjectif) d'algèbres différentielles graduées\footnote{Cet homomorphisme vérifie d'autres hypothèses que nous omettons ici, par souci de clarté.}
$$\mathcal{P} : \mathcal{A}\too \Om^*(M).$$
Si $\alpha,\beta$ sont deux éléments de $\Om^*(M)$, la formule
\begin{equation}\label{crochetpoisson}
\left\{\alpha,\beta\right\}=\mathcal{P}\left([\wi\alpha,\wi\beta]\right),\end{equation}
où $\wi\al,\wi\be$ sont deux antécédents de $\alpha$ et $\beta$, définit sur
$\Om^*(M)$ un crochet, appelé par Hawkins {\it crochet de Poisson généralisé}. Hawkins montre alors que $(\Om^*(M),\wedge,d,\{\;,\;\})$ est une algèbre de Poisson différentielle graduée (voir chapitre \ref{chapter2}). Il montre aussi que le crochet de Poisson généralisé sur $\Om^*(M)$ est entièrement déterminé par les crochets
\begin{equation}\label{crochetpoisson1}
\left\{f,g\right\}\esp \left\{f,\alpha\right\}\quad f,g\in C^\infty(M),\ \alpha\in\Om^1(M).\end{equation}
Le premier crochet donne naissance à une {\it structure de Poisson}\footnote{C'est la première notion fondamentale de ce travail.} sur $M$ et donc un tenseur de Poisson $\pi\in\Gamma(\wedge^2TM)$ et, en posant
\begin{equation}\label{connexion}\mathcal{D}_{df}\alpha:=\left\{f,\al\right\},\end{equation}
on définit $$\mathcal{D}:\Om^1(M)\times\Om^1(M)\too\Om^1(M),$$ qui est en fait une {\it connexion contravariante}\footnote{C'est la deuxième notion fondamentale de ce travail.} associée à $\pi$.
Cette connexion est sans courbure ni torsion.\\
Inversement, étant donnée une structure de Poisson sur $M$ et une connexion contravariante sans courbure ni torsion $\mathcal{D}$, les formules \eqref{crochetpoisson1} et \eqref{connexion} se généralisent pour définir sur $\Om^*(M)$ un crochet $\{\;,\;\}$ compatible avec $\wedge$ et $d$ (voir chapitre \ref{chapter2}). En général, ce crochet ne vérifie pas l'identité de Jacobi graduée. Hawkins a mis en évidence un tenseur ${\mathcal M}$  de type $(2,3)$ appelé {\it métacourbure}\footnote{C'est la troisième notion fondamentale de ce travail.} et a montré que ce crochet $\{\;,\;\}$ vérifie l'identité de Jacobi graduée, si et seulement si, ${\mathcal M}$ est identiquement nul.\\
En conclusion, toute déformation non commutative de $(\Om^*(M),\wedge,d)$ définit sur $\Om^*(M)$ un crochet de Poisson généralisé $\{\;,\;\}$. Ce crochet est entièrement déterminé par la donnée d'un tenseur de Poisson $\pi$ sur $M$, d'une connexion contravariante $\mathcal{D}$ associée à $\pi$ qui est sans courbure ni torsion et dont le tenseur de métacourbure est identiquement nul.
\paragraph{Problème mathématique étudié}~\\
Dans \cite{haw1} et \cite{haw2}, Hawkins a montré que si une déformation de l'algèbre des formes différentielles d'une variété riemannienne $(M,g)$ provient d'une déformation du triplet spectrale\footnote{La notion de triplet spectrale est une notion clef de la géométrie non commutative, nous invitons le lecteur curieux à consulter l'abondante littérature sur le sujet. (Voir \cite{var}).} décrivant la structure riemannienne, alors le tenseur de Poisson $\pi$ (associé à la déformation) et la métrique riemannienne vérifient les conditions suivantes :
\begin{enumerate}
\item la connexion de Levi-Civita contravariante $\mathcal{D}$ associée au couple $(\pi,g)$ est plate,
\item la métacourbure de $\mathcal{D}$ est nulle,
\item le tenseur de Poisson est {\it unimodulaire} par rapport au volume riemannien $\mu$, c'est-à-dire, $d(i_{\pi}\mu)=0$.
\end{enumerate}
La connexion de Levi-Civita contravariante $\mathcal{D}$ associée au couple $(\pi,g)$ est l'analogue de la connexion de Levi-Civita classique ; elle a été introduite dans \cite{bou:compatibility}. Un triplet $(M,\pi,g)$ satisfaisant  les conditions 1. et 2. (respectivement, 1., 2. et .3) sera dit {\it compatible (respectivement, fortement compatible) au sens de Hawkins.} \\
Dans \cite{haw2}, Hawkins a étudié la géométrie des triplets $(M,\pi,g)$ fortement compatibles, lorsque la variété $M$ est compacte. L'étude  des triplets  $(M,\pi,g)$ compatibles ou fortement compatibles au sens de Hawkins dans le cas général reste un problème ouvert et dans \cite{bou:yang-baxter} une large classe d'exemples a été donnée. On est maintenant en mesure d'énoncer le problème mathématique, objet de cette thèse.
\begin{problem}Caractériser les triplets $(G,\pi,\prs)$ compatibles ou fortement compatibles au sens de Hawkins, où $G$ est un groupe de Lie connexe, $\pi$ un tenseur de Lie-Poisson sur $G$ et $\prs$ une métrique riemannienne invariante à gauche sur $G$.\end{problem}
\paragraph{Solution du problème}~\\
Nous allons maintenant décrire la solution, du problème ci-dessus, telle qu'elle a été élaborée dans ce travail. Cette solution, formulée d'une manière précise dans les théorèmes \eqref{main1}-\eqref{main3}, constitue le cœur de cette thèse. Vu l'importance des groupes de Heisenberg, nous les avons traités séparément dans le Théorème \ref{main4} qui donne la solution du problème dans le cas particulier de ces groupes.\\ Avant d'énoncer ces théorèmes nous allons faire quelques rappels.
Plus précisément, on va introduire la notion d'{\it algèbre de Milnor}, notion centrale de la thèse, et rappeler brièvement la définition des  groupes de Lie-Poisson et leurs propriétés essentielles. La notion de groupe de Lie-Poisson a été introduite par Drinfel'd \cite{dr:bialgebra} et étudiée par Semenov Tian Shansky \cite{sts:dressing} (voir aussi \cite{lu-we:poi}).
\begin{enumerate}
\item Une algèbre de Milnor $\G$ est une algèbre de Lie réelle de dimension finie, munie d'un produit scalaire $\prs$ (défini-positif) telle que :
\begin{enumerate}
\item la sous-algèbre $S=\{u\in\mathcal{G}\mid \ad_u+\ad_u^t=0\}$ est abélienne (où $\ad_u^t$ dénote l'adjoint de $\ad_u$ par rapport à $\prs$),
\item l'idéal dérivé $[\G,\G]$ est abélien et $S^\perp=[\G,\G]$ (où $S^\perp$ est l'orthogonal de $S$ par rapport au produit scalaire $\prs$).\\ Cette terminologie est justifiée par un résultat classique de Milnor. En effet, dans \cite{mil}, Milnor a montré qu'une métrique riemannienne invariante à gauche sur un groupe de Lie est plate, si et seulement si, son algèbre de Lie est une somme semi-directe d'une algèbre abélienne $\mathfrak{b}$ avec un idéal abélien $\mathfrak{u}$ et, pour tout $u\in\mathfrak{b}$, $\ad_u$ est antisymétrique. Ce résultat peut être reformulé de façon plus précise et, en Proposition \ref{prmilnor}, on va montrer qu'une métrique riemannienne invariante à gauche sur un groupe de Lie est plate, si et seulement si, son algèbre de Lie est de Milnor.
\end{enumerate}
\item Soit $G$ un groupe de Lie et soit $\G$ son algèbre de Lie. Un tenseur de Poisson $\pi$ sur $G$ est dit multiplicatif si, pour tout $a,b\in G$, $$\pi(ab)=(L_{a})_*\pi(b)+(R_{b})_*\pi(a),$$ où $(L_{a})_*$ (resp. $(R_{b})_*$) dénote l'application tangente de la translation à gauche de $G$ par $a$ (resp. translation à droite de $G$ par $b$). En ramenant $\pi$ à l'élément neutre $e$ de $G$ par translation à droite, on obtient l'application $\pi_r:G\too{\cal G}\wedge{\cal G}$, définie par $\pi_r(g)=(R_{g^{-1}})_*\pi(g)$. Soit
$$\xi:=d_e\pi_r:{\cal G}\too{\cal G}\wedge{\cal G},$$ la dérivée intrinsèque de
$\pi_r$ en $e$. Le fait que $\pi$ est multiplicatif et de Poisson entraîne que $(\G,[\;,\;],\xi)$ est une bigèbre de Lie, c'est-à-dire, $\xi$ est un 1-cocycle par rapport à la représentation adjointe de $\G$ sur $\G\wedge\G$, et l'application duale de $\xi$, $$[\;,\;]^*:{\cal G}^*\times{\cal G}^*\too{\cal G}^*,$$
est un crochet de Lie sur ${\cal G}^*$. La bigèbre $(\G,[\;,\;],\xi)$ est appelé {\it bigèbre de Lie associée} à $(G,\pi)$. La correspondance
$$(G,\pi)\too (\G,[\;,\;],\xi)$$ entre l'ensemble des groupes de Lie-Poisson connexes et simplement connexes et l'ensemble des bigèbres de Lie est bi-univoque (Voir Théorème \ref{theomultiplicatif}, chapitre \ref{chapter1}).\\
D'un autre côté, $(\G^*,[\;,\;]^*,\rho)$ est aussi une bigèbre de Lie, où $\rho:\G^*\too\G^*\wedge\G^*$ est l'application duale du crochet de Lie de $\G$. La bigèbre de Lie $(\G^*,[\;,\;]^*,\rho)$ est appelée {\it bigèbre de Lie duale} de $(G,\pi)$ ou $(\G,[\;,\;],\xi)$. En vertu de ce qui précède, il existe un groupe de Lie-Poisson connexe et simplement connexe $(G^*,\pi^*)$ associé canoniquement à $(\G^*,[\;,\;]^*,\rho)$. C'est le groupe de Lie-Poisson dual de $(G,\pi)$. À noter que si $G$ est connexe et simplement connexe alors $(G,\pi)$ est le dual de $(G^*,\pi^*)$. À remarquer aussi, que si $\G^*$ est identifiée à l'espace des 1-formes différentielles invariantes à gauche et si $d:\G^*\too\G^*\wedge\G^*$ est la différentielle extérieure alors \begin{equation}\label{rho}\rho=-d.\end{equation} \item
 Un groupe de Lie-Poisson muni d'une métrique riemannienne invariante à gauche sera appelé {\it groupe de Lie-Poisson riemannien}.\\ La donnée d'une métrique invariante à gauche sur $G$ est équivalente à
 la donnée d'un produit scalaire $\prs$ sur son algèbre de Lie $\G$. On notera  $\prs^*$ le produit scalaire associé sur $\G^*$, défini par
 $$\langle\al,\be\rangle^*=\langle\#(\al),\#(\be)\rangle,$$où $\#:\G^*\too\G$ est l'isomorphisme défini par $\prs$.
\end{enumerate}
Exposons maintenant, nos principaux résultats :
\begin{theorem}\label{main1}
Soit $(G,\pi,\langle\;,\;\rangle)$ un groupe de Lie-Poisson riemannien et $(\mathcal{G}^*,[\,,\,]^*,\rho)$ sa bigèbre de Lie duale. Alors $(G,\pi,\langle\;,\;\rangle)$ est compatible au sens de Hawkins, si et seulement si :
\begin{enumerate}
\item $(\mathcal{G}^*,[\,,\,]^*,\langle\;,\;\rangle^*_e)$ est une algèbre de Milnor,
\item pour tout $\alpha,\beta,\gamma\in S=\{\alpha\in\mathcal{G}^*\mid \ad_\alpha+\ad_\alpha^t=0\}$,
\begin{equation}\label{flat}
\ad\nolimits_\alpha\ad\nolimits_\beta\rho(\gamma)=0.
\end{equation}
\end{enumerate}
\end{theorem}
\begin{theorem}\label{main2} Soit $(G,\pi)$ un groupe de Lie-Poisson connexe et unimodulaire et soit $\mu$ une forme volume invariante à gauche sur $G$. Alors $d\left(i_{\pi}\mu\right)=0$, si et seulement si :
\begin{enumerate}
\item $(\mathcal{G}^*,[\,,\,]^*)$ est une algèbre de Lie unimodulaire,
\item pour tout $u\in\G$,
\begin{equation}\label{unimodular}
\rho(i_{\xi(u)}\mu_e)=0,
\end{equation}où $\xi$ est le $1$-cocycle associé à $\pi$ et $\rho=-d$ est le $1$-cocycle dual prolongé comme différentiel à $\wedge^{\dim\G-2}\G^*$.\end{enumerate}
\end{theorem}
On va voir ({\it cf.} Proposition \ref{unimod.néc}) que pour un groupe de Lie-Poisson connexe quelconque, la condition $d\left(i_{\pi}\mu\right)=0$ implique \eqref{unimodular}.

Si $G$ est abélien alors $\rho=0$ et on peut déduire des Théorèmes \eqref{main1} et \eqref{main2} le résultat suivant.
\begin{corollary}\label{corollary1}Soit $(\G,\prs)$ une algèbre de Lie munie d'un produit scalaire et soit $\pi_\ell$ la structure de Poisson linéaire canonique sur $\G^*$. Alors $(\G^*,\pi_\ell,\prs^*)$ est fortement compatible au sens de Hawkins si et seulement si $(\G,\prs)$ est une algèbre de Milnor.
\end{corollary}
Le Théorème suivant est une conséquence intéressante des Théorèmes \eqref{main1} et \eqref{main2}.
\begin{theorem}\label{main3}Soit $(G,\pi,\prs)$ un groupe de Lie-Poisson riemannien. Si $G$ est compact semi-simple, $\prs$ est bi-invariant et $\pi=r^--r^+$ où $r^+$ (resp. $r^-$) est le champ de bivecteurs invariant à gauche (resp. invariant à droite) associé à $r\in\wedge^2\G$. Alors $(G,\pi,\prs)$ est fortement compatible au sens de Hawkins, si et seulement si, $[r,r]=0$.\end{theorem}
Soit $H_n$ le groupe de Heisenberg de dimension $2n+1$, $\h_n$ son algèbre de Lie et $z$ un élément non nul du centre de $\h_n$. Il existe une 2-forme $\om$ sur $\h_n$, telle que $i_{z}\om=0$, la projection de $\om$ sur $\h_n/(\R z)$ est non-dégénérée et, pour tous $u,v\in\h_n$, 
\[[u,v]=\om(u,v)z.\]
\begin{theorem}\label{main4}Soient $\pi$ et $\prs$, respectivement, un tenseur de Poisson multiplicatif et une métrique riemannienne invariante à gauche sur $H_n$. Alors $(H_n,\pi,\prs)$ est fortement  compatible au sens de Hawkins, si et seulement si :
\begin{enumerate}
\item il existe un endomorphisme $J:\mathcal{H}_n\too\mathcal{H}_n$ antisymétrique par rapport à $\prs_e$ tel que $J(z)=0$ et, pour tout $u\in\mathcal{H}_n$, $\xi(u)=z\wedge Ju,$ où $\xi:\h_n\too \h_n\wedge \h_n$ est le $1$-cocycle associé à $\pi$,
\item pour tous $u,v\in\mathcal{H}_n$, $\omega(J^2u,v)+\omega(u,J^2v)+2\omega(Ju,Jv)=0.$
\end{enumerate}
\end{theorem}
Au vu des Théorèmes \ref{main1} et \ref{main2}, on peut dire que :
\begin{enumerate}
\item La classification des groupes de Lie-Poisson riemanniens connexes et simplement connexes qui sont compatibles au sens de Hawkins, est équivalent à la classification des structures de bigèbres de Lie sur les algèbres de Milnor pour lesquelles \eqref{flat} est vérifiée.
\item La classification des groupes de Lie-Poisson riemanniens unimodulaires, connexes et simplement connexes qui sont fortement compatibles au sens de Hawkins, est équivalent à la classification des structures de bigèbres de Lie sur les algèbres de Milnor pour lesquelles \eqref{flat} et \eqref{unimodular} sont vérifiées.
\end{enumerate}
En conclusion, nous avons donc réduit le problème géométrique à un problème algébrique largement plus simple, comme le montre les calculs faits en petites dimensions. En effet, les structures de bigèbres de Lie sur les algèbres de Milnor de dimension $\leq5$ peuvent être calculées, et donc les groupes de Lie-Poisson riemanniens de dimension $\leq 5$ fortment compatibles au sens  de Hawkins peuvent être déduits. (Voir chapitre \ref{chapter4}).
\paragraph{Plan de la thèse}~\\
Cette thèse se présente de la façon suivante :
\begin{enumerate}
\item Premier chapitre de généralités, où on rappelle brièvement les notions et résultats classiques sur les variétés de Poisson, les groupes de Lie-Poisson et la classe modulaire des structures de Poisson.
\item Deuxième chapitre, consacré à la notion de connexion contravariante, le crochet de Poisson généralisé et la notion de métacourbure. Ce sont des notions moins connues et nous avons tenu à les introduire d'une manière détaillée.
\item Le troisième chapitre est consacré aux démonstration des Théorèmes \ref{main1} et \ref{main2}.
\item Le quatrième chapitre est consacré aux exemples:
 \begin{enumerate}\item on détermine tous les triplets $(V,\pi,\prs)$ compatibles au sens de Hawkins, où $V$ est un $\R$-espace vectoriel, $\pi$ un tenseur de Poisson linéaire et $\prs$ un produit scalaire sur $V$,
 \item on démontre le Théorème \ref{main4},
 \item on démontre le Théorème \ref{main3},
 \item on détermine tous les triplets $(G,\pi,\prs)$ fortement compatibles au sens de Hawkins, quand $\dim G\leq5$.
 \end{enumerate}
\item Dans l'annexe on indique quelques méthodes pour intégrer les bigèbres de Lie, on donne quelques définitions utiles, attachées à différentes notions de cohomologie.
\end{enumerate}
On termine avec une conclusion, où on indique quelques problèmes ouverts qui constituent la suite logique de ce travail. À la fin, on donne une bibliographie, forcément non exhaustive, de l'abondante littérature qui existe sur les sujets liés à cette thèse.}

\pagestyle{fancy}
\renewcommand{\chaptermark}[1]{%
\markboth{#1}{}}
\renewcommand{\chaptermark}[1]{\markboth{#1}{}}
\fancyhf{} 
\fancyhead[R]{\scriptsize{\textcircled{c}
\textsf{BAHAYOU}}\hspace{0.5cm}\bfseries\thepage}\fancyhead[L]{
\bfseries\leftmark}
\renewcommand{\headrulewidth}{0.5pt}
\renewcommand{\footrulewidth}{0pt}
\addtolength{\headheight}{0.5pt} 
\fancypagestyle{plain}{ 
\fancyhead{} 
\renewcommand{\headrulewidth}{0pt} 
}
\newpage
\chapter{Généralités sur les structures de Poisson}\label{chapter1}
\epigraph{La musique est une mathématique sonore, la mathématique une musique silencieuse.}{Edouard Herriot.}
\lettrine[lines=3, nindent=0em]{D}{ans} ce premier chapitre, on introduit les notions essentielles qui sont attachées à une variété de Poisson. Tout le matériel exposé ici est connu, mais on l'introduit pour avoir un texte "complet", avec un minimum de renvois à des références extérieures. Pour un traitement détaillé de ces thèmes, le lecteur pourra consulter les ouvrages de Vaisman \cite{vai}, Dufour \& Zung \cite{duf-zun}.
\section{Éléments de géométrie de Poisson}
Le crochet de Poisson de fonctions a été introduit par D. Poisson \cite{poi}, comme outil pour l'étude du mouvement des planètes. Par la suite, les structures de Poisson sont apparues comme généralisation des variétés symplectiques, qui formalisent la mécanique hamiltonienne.\\
Les variétés de Poisson, comme on les connait aujourd'hui, ont été introduites par Lichnerowicz                                                                \cite{lich1}, \cite{lich2}. Leur importance a été rapidement reconnue par Weinstein qui en a étudié les propriétés locales \cite{we:local}. Elles jouent aussi un rôle fondamental en physique dans la transition mécanique classique-mécanique quantique (théorie de quantification). Depuis, les articles de Weinstein, plusieurs aspects de la géométrie de Poisson ont fait l'objet d'intenses recherches, citons notamment : le problème de linéarisation, la cohomologie de Poisson-Lichnerowicz, la quantification par déformation, les groupes de Lie-Poisson...
\paragraph{Variété de Poisson.}
Un \emph{crochet de Poisson} sur une variété différentiable $M$ est une application $\mathbb{R}$-bilinéaire                                                                                                   $\{\,,\,\}$ sur l'algèbre $C^\infty(M)$ des fonctions lisses sur $M$ vérifiant
\begin{itemize}
\item[$(\mathrm{i})$] L'antisymétrie
$$\{f,g\}=-\{g,f\},\quad\text{pour tout}\ f,g\in C^\infty(M).$$
\item[$(\mathrm{ii})$] La règle de Leibniz
$$\{f,gh\}=\{f,g\}h+g\{f,h\},\quad\text{pour tout}\ f,g,h\in C^\infty(M).$$
\item[$(\mathrm{iii})$] L'identité de Jacobi
$$\{f,\{g,h\}\}+\{g,\{h,f\}\}+\{h,\{f,g\}\}=0,\quad\text{pour tout}\ f,g,h\in C^\infty(M).$$
\end{itemize}
Une variété munie d'un crochet de Poisson est appelée \emph{variété de Poisson}.\\
Soient $(M,\{\,,\,\})$ une variété de Poisson et $\mathcal{A}=C^\infty(M)$. Le crochet de Poisson fait de $\mathcal{A}$ une algèbre de Lie. La règle de Leibniz implique, pour toute fonction lisse $f$ sur $M$, que l'application linéaire $g\mapsto\{f,g\}$ est une dérivation de $\mathcal{A}$. À chaque fonction $f$ correspond un champ de vecteurs $X_f$, appelé l'hamiltonien de $f$, défini par $X_f(g)=\{f,g\}$. De l'identité de Jacobi on déduit également que
\begin{equation}
[X_f,X_g]=X_{\{f,g\}}.
\end{equation}
Autrement dit, l'ensemble des champs de vecteurs hamiltoniens est une sous-algèbre de l'algèbre de Lie des champs de vecteurs $\left(\mathfrak{X}(M),[\,,\,]\right)$ et l'application $f\mapsto X_f$ définit un homomorphisme d'algèbres de Lie de $\mathcal{A}$ dans $\mathfrak{X}(M)$.\\
Pour tout crochet $\{\,,\,\}$, sur l'algèbre des fonctions d'une variété différentielle $M$, bilinéaire, antisymétrique qui vérifie l'identité de Leibniz, est associé un unique tenseur $2$-fois contravariant antisymétrique, noté $\pi$, tel que
\begin{equation}\label{tenseur}
\{f,g\}=\pi(df,dg).
\end{equation}
Si de plus $\{\,,\,\}$ vérifie l'identité de Jacobi, $\pi$ est appelé \emph{tenseur de Poisson} de $M$.\\
En coordonnées locales $(U,x_1,...,x_n)$ un tel tenseur s'écrit :
\[\pi=\sum_{i<j}\pi_{ij}\,\partial_{x_i}\wedge\partial_{x_j}=\frac{1}{2}\sum_{i,j}\pi_{ij}\,\partial_{x_i}\wedge\partial_{x_j},\]
où les $\pi_{ij}$ sont des fonctions $C^\infty$ sur $U$, tel que $\pi_{ij}=-\pi_{ji}$.\\
L'identité de Jacobi pour $\{,\}$ est équivalente à la condition :
\begin{equation}\label{Poisson}
\oint_{jk\ell}\sum_{i=1}^n\frac{\partial \pi_{k\ell}}{\partial x_i}\pi_{ij}=0,\quad\text{pour tout}\ j,k,\ell=1,...,n
\end{equation}
où $\oint_{jk\ell}a_{jk\ell}$ dénote la somme circulaire $a_{jk\ell}+a_{k\ell j}+a_{\ell jk}$. En effet, compte tenu de \eqref{tenseur}, le jacobiateur
\[J(f,g,h)=\{f,\{g,h\}\}+\{g,\{h,f\}\}+\{h,\{f,g\}\},\]
est égal à 
\[J(f,g,h)=\sum_{j,k,\ell}\left[\sum_{i=1}^n\left(\frac{\partial \pi_{k\ell}}{\partial x_i}\pi_{ij}+\frac{\partial \pi_{\ell j}}{\partial x_i}\pi_{ik}+\frac{\partial \pi_{jk}}{\partial x_i}\pi_{i\ell}\right)\right]\left(\frac{\partial f}{\partial x_k}\,\frac{\partial g}{\partial x_\ell}\,\frac{\partial h}{\partial x_j}\right).\]
Notons $\mathfrak{X}^k(M)$ l'ensemble des champs de multivecteurs de degré $k$, avec \[\mathfrak{X}^0(M)=C^\infty(M),\ \mathfrak{X}^1(M)=\mathfrak{X}(M)\ \text{et}\ \mathfrak{X}^*(M)=\oplus_{k=0}^n\mathfrak{X}^k(M).\]
On peut définir les structures de Poisson grâce au crochet de Schouten \cite{duf-zun}, qui est l'unique extension du crochet de Lie à l'espace des champs de multivecteurs $\mathfrak{X}^*(M)$, caractérisé par
\begin{enumerate}
\item $[P,Q]\in\mathfrak{X}^{p+q-1}(M)$,
\item $[f,g]=0$ et $[X,P]=\mathscr{L}_XP$,
\item $[P,Q]=-\left(-1\right)^{(p-1)(q-1)}[Q,P]$,
\item $[P,Q\wedge R]=[P,Q]\wedge R+\left(-1\right)^{(p-1)q}Q\wedge[P,R]$,
\item $\left(-1\right)^{(p-1)(r-1)}\left[P,[Q,R]\right]+\left(-1\right)^{(q-1)(p-1)}\left[Q,[R,P]\right]+\left(-1\right)^{(r-1)(q-1)}\left[R,[P,Q]\right]=~0$,
\end{enumerate}
pour $P\in\mathfrak{X}^p(M)$, $Q\in\mathfrak{X}^q(M)$, $R\in\mathfrak{X}^r(M)$, $X\in\mathfrak{X}(M)$ et $f,g\in C^\infty(M)$.\\
Une autre définition équivalente du crochet de Schouten, due à Lichnerowicz, est donnée par la proposition suivante \footnote{Cette formule a l'avantage d'être adaptée aux calculs de nature globals.}
\begin{proposition} Pour tout $P\in\mathfrak{X}^p(M)$, $Q\in\mathfrak{X}^q(M)$ et $\omega\in\Omega^{p+q-1}(M)$.
\begin{equation}\label{lichn}
<\omega,[P,Q]>=(-1)^{(p-1)(q-1)}<d(i_Q\omega),P>-<d(i_P\omega),Q>+(-1)^p<d\omega,P\wedge Q>.
\end{equation}
\end{proposition}
\begin{proof}
Voir l'article de Lichnerowicz \cite{lich1}.
\end{proof}
Il en résulte, en particulier, que le crochet de Schouten de deux champs de multivecteurs nuls en un point, est un champs de multivecteurs nul aussi en ce point.\\
Soit $\pi$ un bivecteur sur une variété différentielle $M$ et soit $\{\,,\,\}$ le crochet associé sur $C^\infty(M)$, c'est-à-dire
\[\{f,g\}=<df\wedge dg,\pi>.\]
Le Lemme suivant permet d'exprimer l'identité de Jacobi en terme du tenseur $\pi$.
\begin{lemma}
Le jacobiateur est égale à la moitié du crochet de Schouten de $\pi$ par lui même, c'est-à-dire pour tout $f,g,h\in C^\infty(M)$
\begin{equation}
\{f,\{g,h\}\}+\{g,\{h,f\}\}+\{h,\{f,g\}\}=\frac{1}{2}[\pi,\pi]\left(df,dg,dh\right).
\end{equation}
\end{lemma}
Ceci permet de considérer le jacobiateur comme un tenseur $3$-fois contravariant ; donc nul sur toute variété de dimension $\leq2$.\\
\begin{proof} 
On a d'après la formule de Lichnerowicz \eqref{lichn}
\begin{align*}
<df\wedge dg\wedge dh,[\pi,\pi]>=&-2<d\left(i_\pi(df\wedge dg\wedge dh)\right),\pi>\\
=&-2<d\left(\{g,h\}df-\{f,h\}dg+\{f,g\}dh\right),\pi>\\
=&-2<d\{g,h\}\wedge df-d\{f,h\}\wedge dg+d\{f,g\}\wedge dh,\pi>\\
=&\phantom{-}2\left(\{f,\{g,h\}\}+\{g,\{h,f\}\}+\{h,\{f,g\}\}\right)
\end{align*}
\end{proof}
On peut définir donc une variété de Poisson $M$ par la donnée d'un tenseur $2$-fois contravariant antisymétrique $\pi$ sur $M$, qui vérifie $[\pi,\pi]=0$.
\paragraph{\underline{Exemple $1$ (Variétés symplectiques)}.}
Une variété différentielle $M$ est dite \emph{symplectique} si elle est munie d'une $2$-forme différentielle $\omega$ fermée et non dégénérée, c'est-à-dire le morphisme de fibrés $\omega^\flat : TM\rightarrow T^*M$ défini par \(\omega^\flat(X)=i_X\omega\), pour tout \(X\in\mathfrak{X}(M)\), est un isomorphisme. Pour tout $f\in C^\infty(M)$, il existe un unique champ de vecteurs $X_f$ tel que $i_{X_f}\omega=-df$. On montre alors que le crochet 
\[\{f,g\}=\omega\left(X_f,X_g\right)=-\langle df,X_g\rangle=-X_g(f)=X_f(g),\]
est un crochet de Poisson, de sorte que toute variété symplectique est de Poisson. En effet, 
\begin{align*}
0=&\phantom{-}d\omega\left(X_f,X_g,X_h\right)=X_f\cdot\omega\left(X_g,X_h\right)-X_g\cdot\omega\left(X_f,X_h\right)+X_h\cdot\omega\left(X_f,X_g\right)\\
&\hspace*{3.5cm}-\omega\left([X_f,X_g],X_h\right)+\omega\left([X_f,X_h],X_g\right)-\omega\left([X_g,X_h],X_f\right)\\
=&\phantom{-}X_f\cdot X_g(h)+X_g\cdot X_h(f)+X_h\cdot X_f(g)-[X_f,X_g](h)+[X_f,X_h](g)-[X_g,X_h](f)\\
=&-\Big(\{f,\{g,h\}\}+\{g,\{h,f\}\}+\{h,\{f,g\}\}\Big).
\end{align*}
\paragraph{\underline{Exemple $2$ (Variétés de Poisson Linéaires)}.}
Une structure de variété de Poisson, qui n'est pas symplectique, est donnée canoniquement sur le dual $\G^*$ d'une algèbre de Lie $\G$ par le tenseur, appelé \emph{tenseur de Poisson linéaire} :
\[\pi_x(X,Y)=\langle x,[X,Y]\rangle,\quad\text{pour tout}\ X,Y\in\G,\ x\in\G^*.\]
En d'autres termes, pour tout $f,g\in C^\infty(\G^*)$
\[\{f,g\}(x)=\langle x,[d_xf,d_xg]\rangle,\]
avec l'identification de $\G$ avec son bidual $\G^{**}$, ce qui permet de considérer la différentielle d'une fonction sur $\G^*$ comme un élément de $\G$.\\
Si $X_1,...,X_n$ est une base de $\G$ et $x_1,...,x_n$ le système de coordonnées globales de $\G^*$, associé à cette base, alors
\begin{equation}\label{lineaire1}
\{x_i,x_j\}=\sum_{k=1}^n c_{ij}^kx_k,
\end{equation}
où les $c_{ij}^k$ sont les constantes de structure de $\G$.
\paragraph{\underline{Exemple $3$ (Groupes de Lie-Poisson)}.}
Les groupes de Lie munis d'un tenseur de Poisson \emph{mutliplicatif} constituent une classe importante de variétés de Poisson. Ils jouent un rôle central dans ce travail, c'est pour cela que nous leur consacrons toute une section (voir Section \ref{liepoisson})
\paragraph{Morphisme de Poisson.}~\\
Soient $(M,\{\,,\,\}_M)$ et $(N,\{\,,\,\}_N)$ deux variétés de Poisson. Une application $\varphi : M\rightarrow N$ de classe $C^\infty$ est un \emph{morphisme de Poisson} si $\varphi^* : C^\infty(M)\rightarrow C^\infty(N)$ est un morphisme d'algèbres de Lie, c'est-à-dire, pour tout $f,g\in C^\infty(N)$
\begin{equation}
\{\varphi^*f,\varphi^*g\}_M=\varphi^*\{f,g\}_N,
\end{equation}
où $\varphi^*f$ désigne l'image réciproque de $f$ par $\varphi$, c'est-à-dire l'application $f\circ\varphi$.
\paragraph{Algébroïde de Lie associé à une variété de Poisson.}~\\
Soit $(M,\pi)$ une variété de Poisson. Il est associé au champ de tenseurs $\pi$ un morphisme de fibrés vectoriels $\pi_\sharp : T^*M\rightarrow TM$, appelé \emph{ancrage, ou application ancre}, défini, pour tout $\alpha,\beta\in\Omega^1(M)$, par
\[\beta\left(\pi_\sharp(\alpha)\right)=\pi(\alpha,\beta).\]
Noter que pour tout $f\in C^\infty(M)$, le champ hamiltonien associé est donné par 
\[X_f=\pi_\sharp(df).\]
En coordonnées locales, si $\pi=\frac{1}{2}\sum_{i,j}^n\pi_{ij}\,\partial_{x_i}\wedge\partial_{x_j}$ alors
\[\pi_\sharp(dx_i)=\sum_{j=1}^n\pi_{ij}\,\partial_{x_j}.\]
On appelle \emph{rang du tenseur de Poisson} $\pi$ au point $x$, le rang de l'application linéaire $\pi_\sharp(x) : T_x^*M\rightarrow T_xM$. On le note
\[\rho(x)=\rang\pi_\sharp(x)=\dim\im\pi_\sharp(x).\]
En coordonnées locales, c'est le rang de la matrice $(\pi_{ij}(x))_{1\leq i,j\leq n}$.\\
Un point $x\in M$ est dit \emph{régulier} s'il existe un voisinage de $x$ sur lequel le rang est constant (égal à $\rho(x)$). Une variété de Poisson est dite \emph{régulière} si son tenseur est de rang constant.
La proposition suivante résume les propriétés du rang.
\begin{proposition}
\begin{enumerate}
\item Pour tout $x\in M$, $\rho(x)$ est paire.
\item L'application $x\mapsto\rho(x)$ est semi-continue inférieurement.
\item L'ensemble des points réguliers est un ouvert dense.\footnote{L'ensemble des points régulier est appelé \emph{l'ouvert régulier} de $(M,\pi)$.}
\end{enumerate}
\end{proposition}
\begin{proof}
\begin{enumerate}
\item La matrice $(\pi_{ij}(x))_{1\leq i,j\leq n}$ est antisymétrique, donc de rang paire. 
\item Si $\rho(x)=k$ alors le plus grand mineur de la matrice $(\pi_{ij}(x))_{1\leq i,j\leq n}$, de déterminant non nul, est de type $k\times k$. Par continuité, ce mineur est de déterminant non nul sur un voisinage $V$ de $x$ et donc le rang est au moins égal à $k$ sur ce voisinage. 
\item Par définition même, l'ensemble des points réguliers est un ouvert. Notons $U_1$ l'ouvert de $M$ sur lequel le rang est maximal. La restriction de $\pi$ à l'ouvert $M\setminus U_1$ est un tenseur de Poisson ; soit $U_2$ l'ouvert de $M\setminus U_1$ où le rang est maximal. En réitérant ce processus, on construit une famille fini d'ouverts $U_1$, $U_2$,...,$U_k$ tels que sur chaque $U_i$ le rang
de $\pi$ est constant et tel que l'ensemble des points réguliers est exactement $M^{\text{rég}}=\cup_{i=1}^kU_i$ qui est dense dans $M$.
\end{enumerate}
\end{proof}
Soit $(M,\pi)$ une variété de Poisson. On définit un crochet sur l'ensemble des $1$-formes, appelé \emph{crochet de Koszul} des $1$-formes, pour tout $\alpha,\beta\in\Omega^1(M)$, par
\begin{equation}\label{crochet.forme}
[\alpha,\beta]_\pi=\mathscr{L}_{\pi_\sharp(\alpha)}\beta-\mathscr{L}_{\pi_\sharp(\beta)}\alpha-d\left(\pi(\alpha,\beta)\right).
\end{equation}
Pour simplifier, on notera ce crochet simplement par $[\,,\,]$ (sans l'indice $\pi$).\\
La proposition suivante résume les propriétés essentielles de ce crochet.
\begin{proposition}
Soit $(M,\pi)$ une variété de Poisson. Alors le crochet de Koszul vérifie les propriétés suivantes :
\begin{enumerate}
\item $[\alpha,f\beta]=\pi_\sharp(\alpha)(f)\beta+f[\alpha,\beta],\ \alpha,\beta\in\Omega^1(M)$ et $f\in C^\infty(M)$.
\item $[\,,\,]$ est un crochet de Lie sur $\Omega^1(M)$ et, pour tout $f,g\in C^\infty(M)$
\[[df,dg]=d\{f,g\}.\]
\item $\pi_\sharp : \Omega^1(M)\rightarrow\mathfrak{X}(M)$ est un morphisme d'algèbres de Lie, c'est-à-dire, pour tout $\alpha,\beta\in\Omega^1(M)$,
\[\pi_\sharp\left([\alpha,\beta]\right)=[\pi_\sharp(\alpha),\pi_\sharp(\beta)].\]
\end{enumerate}
\end{proposition}
\begin{proof}
Voir le livre de Vaisman \cite{vai}, pages 41-42.
\end{proof}
Cette proposition signifie que le triplet $(\pi_\sharp,T^*M,M)$ est un algébroïde de Lie au sens de Pradines \cite{pr:68}.
\paragraph{Structure locale d'une variété de Poisson.}
Le Théorème de décomposition de Weinstein \cite{we:local}, affirme que toute variété de Poisson est localement le produit d'une variété symplectique par une variété de Poisson dont le rang s'annule en un point. Plus précisément, tout point $x_0$ d'une variété de Poisson $(M,\pi)$, où le rang de $\pi$ est égal à $2k$, admet un voisinage $U$ isomorphe, par un difféomorphisme de Poisson, au produit d'une variété symplectique de dimension $2k$ et d'une variété de Poisson de rang nul en $x_0$. De plus, cette
décomposition est unique à un isomorphisme local près. Le Théorème s'énonce comme suit :
\begin{theorem}\label{weinstein.decomp}
Soit $(M,\pi)$ une variété de Poisson de dimension $n=2k+r$, et $a$ un point de $M$ où $\pi$ est de rang $2k$. Il existe alors un système de coordonnées $(U,x_1,...x_{2k},y_1,...,y_r)$ centré en $a$ tel que
\begin{equation}\label{weinstein}
\pi=\sum_{i=1}^k\partial_{x_i}\wedge\partial_{x_{i+k}}+\frac{1}{2}\sum_{i<j}\phi_{ij}(y)\,\partial_{y_i}\wedge\partial_{y_j},
\end{equation}
où chaque $\phi_{ij}$ ne dépend que des coordonnées $y_1,...,y_r$ et $\phi_{ij}(a)=0$.
\end{theorem}
\begin{proof}Voir l'article de Weinstein \cite{we:local}, on encore \cite{duf-zun}, \cite{vai}.
\end{proof}
On a les corollaires importants suivants :
\begin{corollary}[Théorème de Lie]
Si $\pi$ est un tenseur de Poisson sur $M$, de rang localement constant, alors tout point de
$M$ possède un voisinage dans lequel la matrice $(\pi_{\sharp_{ij}})$ est constante.
\end{corollary}
\begin{corollary}
Soit $\G^*$ le dual d'une algèbre de Lie $\G$, muni de sa structure de
Lie-Poisson canonique. Le tenseur de Poisson de $\G^*$ à l'origine est de la forme :
\[\pi=\frac{1}{2}\sum_{i<j}\varphi_{ij}(y)\,\partial_{y_i}\wedge\partial_{y_j},\]
c'est-à-dire, sans la partie symplectique et les fonctions $\phi_{ij}$ sont linéaires.
\end{corollary}
\begin{corollary}
Une variété de Poisson $(M,\pi)$ est symplectique, si et seulement si, $\pi$ est de
rang constant égal à la dimension, nécessairement paire, de $M$. De plus, si $\dim M=2n$,
alors $M$ est localement isomorphe à un ouvert de $\mathbb{R}^{2n}$ muni de sa structure symplectique
canonique.
\end{corollary}
On retrouve, en particulier le Théorème de Darboux \cite{ber} (connu pour les variétés symplectiques). On déduit aussi que la décomposition, dans le Théorème \eqref{weinstein.decomp} ci-dessus, est unique à un isomorphisme local près.
\paragraph{Feuilletage symplectique}
Étant donnée une variété de Poisson $(M,\pi)$, on lui associe la distribution $D=\bigcup_{x\in M}D_x$ où chaque $D_x$ est le sous-espace vectoriel de $T_xM$ :
\[D_x=\{X_f(x)\mid f\in C^\infty(M)\},\]
où $X_f$ est le champ de vecteurs hamiltonien associé à $f$. Cette distribution est involutive, puisque
\[[X_f,X_g]=X_{\{f,g\}}\quad\text{pour tout}\ f,g\in C^\infty(M).\]
Si $M$ est régulière, c'est-à-dire, si les $D_x$ sont tous de même dimension, (ou, ce qui revient au même, que l'application $x\mapsto\rang\pi_\sharp(x)$ est constante), alors la distribution est intégrable d'après le Théorème de Frobenius (où le fait d'être de rang constant est essentiel). Dans le cas général, cette distribution est singulière (dite aussi généralisée), dans le sens que les $D_x$ ne sont pas tous de même dimension. Il existe cependant d'autres arguments, Théorème de Stefan-Sussmann entre autres, pour montrer l'intégrabilité.
En effet, dans ce cas l'involutivité est remplacée par l'invariance de la distribution par l'action des champs hamiltoniens
\[\left(\exp tX\right)_*D_x=D_{\exp tX(x)},\]
pour tout $x\in M$, pour tout champ hamiltonien $X$ et pour tout $t$ pour lequel le flot local de $X$, $\exp tX$ est défini. Si $D_x$ est de dimension $k$ et si $X_1,...,X_k$ sont des champs hamiltoniens tels que $X_1(x),...,X_k(x)$ engendrent $D_x$ alors l'application
\[(t_1,...,t_k)\mapsto\exp t_1X_1\circ...\circ\exp t_kX_k\,(x)\]
définit une immersion d'une boule ouverte $B(0,r)$ de $\mathbb{R}^k$ dans une sous-variété de $M$ de dimension $k$, contenant $x$. Ce qui donne l'intégrabilité de la distribution.\\
On peut montrer l'intégrabilité de cette distribution, directement à partir du Théorème Weinstein. En effet, on muni $M$ de la relation d'équivalence $\sim$ définie par : $x\sim y$, si et seulement si, il existe une
courbe joignant $x$ et $y$, dont chaque segment est un morceau de courbe intégrale d'un
champ de vecteurs hamiltonien sur $M$. Alors chaque classe d'équivalence $S_{x_0}$ est une sous-variété de dimension $k=\rang\pi_\sharp(x_0)$. De plus, chaque sous-variété $S_{x_0}$ possède une structure symplectique. (Voir \cite{duf-zun}, page 20).
\paragraph{Cohomologie de Poisson-Lichnerowicz}
C'est une cohomologie relative à toute structure de Poisson ; elle a été introduite par A. Lichnerowicz et c'est une notion centrale dans la théorie de déformation des structures de Poisson ; (voir \cite{gam}).

Soit $(M,\pi)$ une variété de Poisson et soit $\delta_\pi : \mathfrak{X}^*(M)\rightarrow\mathfrak{X}^{*+1}(M)$ l'opérateur défini par
\[\delta_\pi:=[\pi,\cdot].\]
L'identité de Jacobi graduée du crochet de Schouten permet de montrer que $\delta_\pi$ est un
opérateur de cohomologie, c'est-à-dire $\delta_\pi\circ\delta_\pi=0$. En effet, pour tout $X\in\mathfrak{X}^k(M)$, $\delta_\pi(P)\in\mathfrak{X}^{k+1}(M)$ et, d'après l'identité de Jacobi généralisée pour le crochet de Schouten,
\[(-1)^{k-1}[\pi,[\pi,X]]-[\pi,[X,\pi]]+(-1)^{k-1}[X,[\pi,\pi]]=0.\]
Comme $[\pi,\pi]=0$ et $[X,\pi]=-(-1)^{k-1}[\pi,X]$, alors $[\pi,[\pi,X]]=0.$\\
Le complexe $(\mathfrak{X}^*(M),\delta_\pi)$ est le complexe de Poisson de $M$ et la cohomologie associée, que l'on notera par $H_\pi^*(M)$, s'appelle la cohomologie de Poisson de $M$. 
\begin{equation}
H_\pi^k(M)=\frac{\ker\left(\delta_\pi : \mathfrak{X}^k(M)\rightarrow\mathfrak{X}^{k+1}(M)\right)}{\im\left(\delta_\pi : \mathfrak{X}^{k-1}(M)\rightarrow\mathfrak{X}^k(M)\right)}.
\end{equation}
Pour un tenseur de Poisson donné, on notera l'opérateur de cobord simplement par $\delta$ et le $k^{\text{ième}}$ espace de cohomologie par $H^k(M)$.\\
L'opérateur de cobord $\delta$ peut être défini par :
\begin{align}
\delta X(\alpha_1,...,\alpha_{k+1})=&\sum_{i=1}^{k+1}(-1)^{i-1}\pi_\sharp(\alpha_i)\cdot X\left(\alpha_1,...,\widehat{\alpha_i},...,\alpha_{k+1}\right)\label{differential}\\
&+\sum_{1\leq i<j\leq k+1}(-1)^{i+j}X\left([\alpha_i,\alpha_j],\alpha_1,...,\widehat{\alpha_i},...,\widehat{\alpha_j},...,\alpha_{k+1}\right)\notag,
\end{align}
où le symbole $``\ \widehat{}\ ''$ signifie que le terme correspondant est omis.\\
En effet, pour toute $f\in C^\infty(M)$ et toute $\alpha\in\Omega^1(M)$
\[\delta f(\alpha)=-X_f(\alpha)=\pi_\sharp(\alpha)\cdot f,\]
et pour tout champs de multivecteurs $X,Y\in\mathfrak{X}^*(M)$
\[\delta\left(X\wedge Y\right)=\delta X\wedge Y+(-1)^{\deg X}X\wedge\delta Y,\]
ce qui montre, de la même manière que dans le cas de la différentielle usuelle sur les formes, que $\delta$ coïncide avec l'opérateur \eqref{differential}.
L'application d'ancrage $\pi_\sharp : \omega^1(M)\rightarrow\mathfrak{X}(M)$, qui est un morphisme d'algèbres de Lie, s'étend aux formes de tout degrés par
\[\sharp\left(\alpha_1\wedge...\wedge\alpha_k\right)=\pi_\sharp(\alpha_1)\wedge...\wedge\pi_\sharp(\alpha_k),\]
et induit un homomorphisme entre la cohomologie de de Rham $H_d^*(M)$ et la cohomologie de Poisson $H_\pi^*(M)$. En effet, pour toute $k$-forme $\alpha$ on a :
\[\sharp(d\alpha)=-\delta_\pi\left(\sharp(\alpha)\right),\]
ce qui entrelace $\delta_\pi$ et la différentielle extérieure $d$ de de Rham. Cet homomorphisme est un isomorphisme dans le cas symplectique.
Comme en général pour toute cohomologie, les espaces de bas degrés admettent des interprétations simples,
$H^0(M)$ est l'ensemble des fonctions Casimir (éléments du centre de l'algèbre $C^\infty(M)$),
$H^1(M)$ est le quotient des champs de vecteurs de Poisson par les champs hamiltoniens, $H^2(M)$ est le quotient des bivecteurs compatibles avec $\pi$ par les dérivées de Lie de ce crochet. On voit donc l'intérêt de calculer la cohomologie de Poisson. Une question ouverte et difficile est de déterminer les structures de cohomologie de Poisson finie (i.e à groupes de cohomologie $H_\pi^k(M)$ de dimension finie), en dehors du cas symplectique.
\section{Classe modulaire d'une variété de Poisson}
Soit $(M,\pi)$ une variété de Poisson. On dit que $M$ est \emph{unimodulaire}\footnote{Si la variété n'est pas orientable, on peut définir la classe modulaire en remplaçant forme volume par densité.} si elle possède une forme volume $\mu$ invariante par tous les champs hamiltoniens, c'est-à-dire
\begin{equation*}
\mathscr{L}_{X_f}\mu=0,\quad\text{pour tout}\ f\in C^\infty(M).
\end{equation*}
On montre que ceci équivaut à l'existence d'un champ modulaire nul.
\paragraph{Champ de vecteurs modulaire}
Soit $(M,\pi)$ une variété de Poisson de dimension $n$, munie d'une forme volume $\mu$. Comme $\Omega^n(M)$ est de dimension $1$, en tant que $C^\infty(M)$-module, il est engendré par $\mu$ et on peut associer à tout champ de vecteurs $X\in\mathfrak{X}(M)$ la fonction, notée par $\Div_\mu X$, définie par
\[\mathscr{L}_X\mu=\left(\Div\nolimits_\mu X\right)\,\mu,\]
appelée \emph{divergence de $X$ par rapport à $\mu$}.\\La divergence possède les propriétés suivantes, pour tout $X,Y\in\mathfrak{X}(M)$ :
\begin{equation}\label{unimodular1}
\Div\nolimits_\mu[X,Y]=X\left(\Div\nolimits_\mu Y\right)-Y\left(\Div\nolimits_\mu X\right),
\end{equation}
\begin{equation}\label{unimodular2}
\Div\nolimits_\mu(fX)=f\Div\nolimits_\mu X+X(f),\ \text{pour tout}\ f\in C^\infty(M),
\end{equation}
et pour tout $f>0$
\begin{equation}\label{unimodular3}
\Div\nolimits_{f\mu}X=\Div\nolimits_\mu X+X(\log f).
\end{equation}
Pour la preuve de ces propriétés, on procède comme suit : On déduit \eqref{unimodular1} de l'égalité \[\mathscr{L}_{[X,Y]}\mu=\mathscr{L}_X\left(\mathscr{L}_Y\mu\right)-\mathscr{L}_Y\left(\mathscr{L}_X\mu\right).\]
Pour \eqref{unimodular2}, on a
\[\mathscr{L}_{fX}\mu=f\mathscr{L}_X\mu+df\wedge i_X\mu=f\mathscr{L}_X\mu+X(f)\mu.\]
Pour \eqref{unimodular3} on a, pour tout $f>0$
\begin{align*}
\mathscr{L}_X(f\mu)=&X(f)\mu+f\mathscr{L}_X\mu=\frac{X(f)}{f}f\mu+f\Div\nolimits_\mu X\mu\\
=&\left(X(\log f)+\Div\nolimits_\mu X\right)f\mu.
\end{align*}
On définit maintenant le champ modulaire, on montre que c'est un champ de vecteurs de Poisson, et que sa classe modulo les champs hamiltoniens, ne dépend pas du choix de la forme volume $\mu$.
\begin{definition}[Champ modulaire]
Soit $(M,\pi,\mu)$ une variété de Poisson orientable. On définit le champ modulaire \index{champ modulaire}, pour toute $f\in C^\infty(M)$, par
\[X_\mu(f)=\Div\nolimits_\mu X_f.\]
\end{definition}
Le champ modulaire possède les propriétés suivantes :
\begin{proposition}~
\begin{enumerate}
\item $X_\mu$ est un champ de vecteurs de Poisson.
\item Pour tout $f>0$, on a
\begin{equation}\label{modular.class}
X_{f\mu}=-X_{\log f}+X_\mu.
\end{equation}
\end{enumerate}
\end{proposition}
\begin{proof}~
\begin{enumerate}
\item Montrons tout d'abord que $X_\mu$ est un champ de vecteurs, c'est-à-dire une dérivation. Pour cela soient $f,g\in C^\infty(M)$
\begin{align*}
X_\mu(fg)&=\Div\nolimits_\mu(X_{fg})=\Div\nolimits_\mu(fX_g+gX_f)\\
&\stackrel{\eqref{unimodular2}}{=}fX_\mu(g)+X_g(f)+gX_\mu(f)+X_f(g)\\
&=fX_\mu(g)+\{g,f\}+gX_\mu(f)+\{f,g\}\\
&=fX_\mu(g)+gX_\mu(f).
\end{align*}
Rappelons qu'un champ de vecteurs $X$ est de Poisson si son flot préserve le tenseur de Poisson, c'est-à-dire $\mathscr{L}_X\pi=0$. On montre que ceci équivaut à
\begin{equation}\label{Poisson.field}
X\{f,g\}=\{X(f),g\}+\{f,X(g)\}.
\end{equation}
En effet, on a
\begin{align*}
X\{f,g\}=&\mathscr{L}_X\langle\pi,df\wedge dg\rangle\\
=&\langle\mathscr{L}_X\pi,df\wedge dg\rangle+
\langle\pi,\mathscr{L}_X\left(df\wedge dg\right)\rangle\\
=&\langle\mathscr{L}_X\pi,df\wedge dg\rangle+
\langle\pi,dX(f)\wedge dg\rangle+\langle\pi,df\wedge dX(g)\rangle\\
=&\langle\mathscr{L}_X\pi,df\wedge dg\rangle+\{X(f),g\}+\{f,X(g)\}.
\end{align*}
Montrons maintenant \eqref{Poisson.field} pour $X_\mu$.
\begin{align*}
X_\mu\left(\{f,g\}\right)=&\Div\nolimits_\mu\left(X_{\{f,g\}}\right)=\Div\nolimits_\mu\left([X_f,X_g]\right)\\
\stackrel{\eqref{unimodular1}}{=}&X_f\left(X_\mu(g)\right)-X_g\left(X_\mu(f)\right)=\{f,X_\mu(g)\}+\{X_\mu(f),g\}.
\end{align*}
\item Pour tout $f>0$, on a
\[X_{f\mu}(g)=\Div\nolimits_{f\mu}(X_g)\stackrel{\eqref{unimodular3}}{=}X_{\mu}(g)+X_g\left(\log f\right)=X_{\mu}(g)-X_{\log f}(g).\]
\end{enumerate}
\end{proof}
En vertu de la Proposition ci-dessus on déduit que, dans le premier espace de cohomologie de Poisson $H^1_\pi(M)$, le champ modulaire $X_\mu$ représente une classe qui ne dépend pas de la forme volume choisie, puisque $[X_{f\mu}]=[X_\mu]$ pour tout $f>0$. On notera cette classe par $\mathcal{M}$ et on l'appellera \emph{la classe modulaire} de $\pi$. La structure de Poisson est dite \emph{unimodulaire} si $\mathcal{M}=0$. On va comprendre l'origine de cette appellation lorsqu'on étudie le cas d'une variété de Poisson linéaire et son lien avec l'unimodularité de l'algèbre de Lie.
\begin{proposition}
On a l'équivalence des propriétés suivantes :
\begin{enumerate}
\item $(M,\pi)$ est unimodulaire.
\item Il existe une forme volume $\mu$ sur $M$ invariante par tout champ hamiltonien, c'est-à-dire
\[\mathscr{L}_{X_f}\mu=0,\ \text{pour toute}\ f\in C^\infty(M).\]
\item Il existe une forme volume $\mu$, sur $M$, telle que $d\left(i_\pi\mu\right)=0$.
\end{enumerate}
\end{proposition}
\begin{proof}
Il est clair que les propositions $1$ et $2$ sont équivalentes puisque
\begin{align*}
\mathscr{L}_{X_f}\mu=0,\ \forall f\in C^\infty(M)&\Longleftrightarrow\Div\nolimits_\mu X_f=0,\ \forall f\in C^\infty(M)\\
&\Longleftrightarrow X_\mu(f)=0,\ \forall f\in C^\infty(M)\\
&\Longleftrightarrow X_\mu=0.
\end{align*}
Pour montrer que $1$ et $3$ sont équivalentes, il suffit de montrer l'égalité
\begin{equation}\label{modular.Poisson.tensor}
d\left(i_\pi\mu\right)=i_{x_\mu}\mu,
\end{equation}
et de remarquer que, pour toute forme volume $\mu$, $i_X\mu=0\Longleftrightarrow X=0$.\\
Dans une carte locale connexe $(U,x_1,...,x_n)$, la forme volume s'écrit $\mu=\Phi\,dx_1\wedge...\wedge dx_n$ avec $\Phi>0$, le tenseur de Poisson s'écrit $\pi=\sum_{i<j}\pi_{ij}\,\partial_{x_i}\wedge\partial_{x_j}$ et le champ hamiltonien s'écrit $X_f=\sum_{j=1}^n\left(\sum_{i=1}^n\pi_{ij}\frac{\partial f}{\partial x_i}\right)\partial_{x_j}$. On a
\begin{multline*}
X_\mu(f)=\Div\nolimits_\mu X_f\stackrel{\eqref{unimodular3}}{=}\Div\nolimits_{dx_1\wedge...\wedge dx_n}X_f+X_f\left(\log\Phi\right)\\\stackrel{\eqref{unimodular2}}{=}\sum_{j=1}^n\left(\sum_{j=1}^n\pi_{ij}\frac{\partial f}{\partial x_i}\right)\Div\nolimits_{dx_1\wedge...\wedge dx_n}\partial_{x_j}+\sum_{j=1}^n\sum_{i=1}^n\partial_{x_j}\left(\pi_{ij}\frac{\partial f}{\partial x_i}\right)\\
+\frac{1}{\Phi}\sum_{j=1}^n\sum_{i=1}^n\pi_{ij}\frac{\partial f}{\partial x_i}\frac{\partial\Phi}{\partial x_j}.
\end{multline*}
comme
\[\Div\nolimits_{dx_1\wedge...\wedge dx_n}\partial_{x_j}=0,\ j=1,...,n\ \text{et}\ \sum_{j=1}^n\sum_{i=1}^n\pi_{ij}\frac{\partial^2f}{\partial x_i\partial x_j}=0,\]
on a $X_\mu=\sum_{i=1}^n\left(\sum_{j=1}^n\frac{\partial\pi_{ij}}{\partial x_j}+\frac{1}{\Phi}\pi_{ij}\frac{\partial\Phi}{\partial x_j}\right)\partial x_i$, et donc
\begin{equation*}
i_{X_\mu}\mu=\sum_{i=1}^n(-1)^{i+1}\Big(\sum_{j=1}^n\partial x_j(\pi_{ij}\Phi)\Big)\,dx_1\wedge...\wedge \widehat{dx_i}\wedge...\wedge dx_n.
\end{equation*}
D'autre part l'égalité
\begin{multline*}
d\omega(X_1,...,X_{k+1})=\sum_{i=1}^{k+1}(-1)^{i+1}X_i\cdot\omega\left(X_1,...,\widehat{X_i},...,X_{k+1}\right)\\
+\sum_{1\leq i<j\leq k+1}(-1)^{i+j}\omega\left([X_i,X_j],X_1,...,\widehat{X_i},...,\widehat{X_j},...,X_{k+1}\right),
\end{multline*}
appliquée à $\omega=i_\pi\mu$ et aux champs $X_i=\partial x_i$, qui commutent deux à deux, donne
\begin{align*}
d\left(i_\pi\mu\right)(\widehat{\partial_{x_1}},\partial_{x_2},...,\partial_{x_n})=&\phantom{-}\partial_{x_2}(\pi_{12}\Phi)+...+\partial_{x_n}(\pi_{1n}\Phi)\\
d\left(i_\pi\mu\right)(\partial_{x_1},\widehat{\partial_{x_2}},...,\partial_{x_n})=&-\partial_{x_1}(\pi_{21}\Phi)-...-\partial_{x_n}(\pi_{2n}\Phi)\\
\dotfill=&\phantom{-}\ldots\\
d\left(i_\pi\mu\right)(\partial_{x_1},\partial_{x_2},...,\widehat{\partial_{x_n}})=&(-1)^{n+1}\partial_{x_1}(\pi_{n1}\Phi)+...+(-1)^{n+1}\partial_{x_{n-1}}(\pi_{n\,n-1}\Phi).
\end{align*}
On déduit que
\begin{equation*}
d\left(i_\pi\mu\right)=\sum_{i=1}^n(-1)^{i+1}\Big(\sum_{j=1}^n\partial x_j(\pi_{ij}\Phi)\Big)\,dx_1\wedge...\wedge \widehat{dx_i}\wedge...\wedge dx_n,
\end{equation*}
ce qui donne l'égalité \eqref{modular.Poisson.tensor}.
\end{proof}
Ayant caractérisé l'unimodularité des structures de Poisson, passons à l'étude de quelques exemples.
\paragraph{Variétés symplectiques} Toute variété symplectique $(M,\omega)$ est unimodulaire ; la forme volume $\mu=\wedge^{\frac{\dim M}{2}}\omega$ étant invariante par tout champ hamiltonien. En effet, la forme symplectique $\omega$ est invariante par tout champ hamiltonien, (donc aussi $\mu$) :
\[\mathscr{L}_{X_f}\omega=di_{X_f}\omega+i_{X_f}d\omega=di_{X_f}\omega=d\left(-df\right)=0.\]
\paragraph{Variétés de Poisson linéaires}
Soit $\G$ une algèbre de Lie. Rappelons que $\G$ est dite unimodulaire si pour tout $u\in\G$
\[\tr({\ad}_u)=0.\]
On a le résultat suivant qui justifie le terme d'unimodularité.
\begin{proposition} La structure de Poisson linéaire sur $\G^*$ est unimodulaire (en tant que variété de Poisson), si et seulement si, $\G$ est unimodulaire (en tant qu'algèbre de Lie).\end{proposition}
\begin{proof} On choisit un système de coordonnées linéaires $\{x_1,\ldots,x_n\}$ sur $\G^*$ associé à une base $(e_1,\ldots,e_n)$ et on considère une forme volume
$$\mu=\Phi\,dx_1\wedge\ldots\wedge dx_n.$$
On a, d'après la preuve de la Proposition 2,
$$X_\mu=\sum_{i=1}^n\left(\sum_{j=1}^n\frac{\partial\pi_{ij}}{\partial x_j}+\frac{1}{\Phi}\pi_{ij}\frac{\partial\Phi}{\partial x_j}\right)\partial x_i.$$
Or, d'après (\ref{lineaire1}),
$$\frac{\partial\pi_{ij}}{\partial x_j}=c_{ij}^j,$$
et $$\sum_{j=1}^nc_{ij}^j=\tr({\ad}_{e_i}).$$
On obtient alors que
$$X_\mu=\sum_{i=1}^n\left(\tr({\ad}_{e_i})+\frac{1}{\Phi}\pi_{ij}\frac{\partial\Phi}{\partial x_j}\right){\partial x_j}.$$
Si $\G$ est unimodulaire, on prend $\Phi=1$ et on obtient $X_\mu=0$. Inversement, si $X_\mu=0$, alors $X_\mu(0)=0$, ce qui entraîne que $\G$ est unimodulaire.
\end{proof}
\section{Groupes de Lie-Poisson}\label{liepoisson}
La notion de groupe de Lie-Poisson a été introduite par Drinfeld dans \cite{dr:bialgebra}, puis étudiée systématiquement par Semenov-Tian Shansky dans \cite{sts:dressing}. On peut se référer à l'excellent article de Lu \& Weinstein \cite{lu-we:poi} et à la thèse de Lu \cite{lu}, ou encore à \cite{duf-zun} et \cite{vai}.
\begin{definition}
Un groupe de Lie-Poisson\index{Lie-Poisson} est un groupe de Lie muni d'une structure de variété de
Poisson, tel que la multiplication $m : G\times G\rightarrow G$
soit un morphisme de Poisson. Si $\pi$ désigne le $2$-tenseur de Poisson sur $G$,
on a pour tout $g,h\in G$ :
\[\pi(gh)=L_{g*}\pi(h)+R_{h*}\pi(g).\]
\end{definition}
Plus généralement, un $k$-tenseur $K$ sur un groupe de Lie $G$ est multiplicatif\index{multiplicatif}
s'il vérifie la relation précédente, c'est-à-dire :
\begin{equation}\label{k-multiplicatif}
K(gh)=L_{g*}K(h)+R_{h*}K(g).
\end{equation}
Soit un $k$-tenseur $K$, et soient $K_\ell$ et $K_r :
G\rightarrow\wedge^k\G$ définis par les formules :
\[K_\ell(g)=L_{g^{-1}*}K(g),\quad K_r(g)=R_{g^{-1}*}K(g).\]
De l'égalité :
\begin{align*}
K_r(gh)-K_r(g)-\Ad(g)\cdot
K_r(h)=&R_{(gh)^{-1}*}K(gh)-R_{g^{-1}*}K(g)\\
&\hspace*{2cm}-L_{g*}R_{g^{-1}*}R_{h^{-1}*}K(h)\\
=&R_{(gh)^{-1}*}K(gh)-R_{g^{-1}*}K(g)\\
&\hspace*{2cm}-R_{(gh)^{-1}*}L_{g*}K(h)\\
&=R_{(gh)^{-1}*}\Big[K(gh)-R_{h*}K(g)-L_{g*}K(h)\Big],
\end{align*}
on déduit une première caractérisation :\\
\emph{
Un $k$-tenseur $K$ est multiplicatif, si et seulement si, son translaté à droite $K_r$ est un $1$-cocycle de $G$ dans $\wedge^k\G$ pour la représentation adjointe}
\[K_r(gh)=K_r(g)+\Ad(g)\cdot K_r(h).\]
(Voir l'annexe pour la cohomologie des groupes et des algèbres de Lie). On a une caractérisation analogue par $K_\ell$, c'est-à-dire
\[K_\ell(gh)=K_\ell(h)+\Ad(h^{-1})\cdot K_\ell(g).\] 
Si le groupe de Lie $G$ est connexe, alors on a une autre caractérisation importante :\\
\emph{Un $k$-tenseur $K$ est multiplicatif, si et seulement si, $K(e)=0$ et $\mathscr{L}_XK$ est invariant à gauche pour tout champ de vecteur $X$ invariant à gauche.} En effet, l'égalité \eqref{k-multiplicatif} appliquée à $g=h=e$ montre que $K(e)=0$, et :
\begin{align*}
\mathscr{L}_XK(g)=&\frac{d}{dt}\big|_{t=0}\left(\exp tX\right)^*K(\exp tX(g))\\
=&\frac{d}{dt}\big|_{t=0}R_{\exp tX}^*K(g\exp tX)\\
=&\frac{d}{dt}\big|_{t=0}R_{\exp tX}^*L_{g*}K(\exp tX)+\frac{d}{dt}\big|_{t=0}R_{\exp tX}^*R_{\exp tX*}K(g)\\
=&L_{g*}\frac{d}{dt}\big|_{t=0}R_{\exp tX}^*K(\exp tX)\\
=&L_{g*}\left(\mathscr{L}_XK\right)(e),
\end{align*}
ce qui veut dire que $\mathscr{L}_XK$ est invariant à gauche. Le raisonnement est le même pour les champs invariants à droite, en utilisant $K_\ell$.\\
Réciproquement, si $\mathscr{L}_XK$ est invariant à gauche pour tout champ invariant à gauche $X$, c'est-à-dire \[\mathscr{L}_XK_r(gh)=\Ad(g)\cdot\mathscr{L}_XK_r(h) \footnote{En fait un $k$-tenseur $T$ est invariant à gauche, si et seulement si, son translaté à droite $T_r(g)=R_{g^{-1}*}T(g)$ est équivariant, 
c'est-à-dire $T_r(gh)=\Ad(g)\cdot T_r(h)$},\]alors
\begin{align*}
\frac{d}{dt}\big[K_r(g\exp tX)&-K_r(g)-\Ad(g)\cdot K_r(\exp tX)\big]\\
=&\ R_{\exp tX}^*\left(\mathscr{L}_XK_r\right)(g\exp tX)-\Ad(g)\cdot R_{\exp tX}^*\left(\mathscr{L}_XK_r\right)(\exp tX)\\
=&\ R_{\exp tX}^*\big[\left(\mathscr{L}_XK_r\right)(g\exp tX)-\Ad(g)\cdot\left(\mathscr{L}_XK_r\right)(\exp tX)\big]\\
=&\ 0,
\end{align*}
car $\mathscr{L}_XK_r$ est équivariant. On a donc, pour tout $g\in G$ et tout $h$ dans un voisinage de $e$,
\[K_r(gh)=K_r(g)+\Ad(g)\cdot K_r(h).\]
Comme $G$ est connexe, il est engendré par l'image de l'application exponentielle
$\exp : \G\rightarrow G$, on déduit que $K_r$ est un $1$-cocycle pour la représentation adjointe de $G$ sur $\wedge^2\G$.
\paragraph{Dérivée intrinsèque}
Soit $K$ un $k$-tenseur s'annulant en $e$. La dérivée intrinsèque de
$K$ en $e$ est l'application linéaire :
\[\begin{array}{cccc}
d_eK : & \G & \longrightarrow & \wedge^k\G \\
 & v & \longmapsto & \mathscr{L}_XK(e)
  \end{array}\]
où $X$ est un champ de vecteurs quelconque, tel que $X(e)=v$.\\
Cette application est bien définie car elle ne dépend que de la valeur de $X$ en $e$. En effet, comme $K(e)=0$ alors $\mathscr{L}_XK(e)=0$, pour tout champ de vecteurs $X$ tel que $X(e)=0$.\footnote{Le crochet de Schouten de deux champs nuls en un point, est un champs nul nul en ce point \eqref{lichn}.}
\begin{proposition} Soit $G$ un groupe de Lie connexe Alors :
\begin{enumerate}
\item Un $k$-tenseur multiplicatif $K$ sur $G$ est nul, si et seulement si,
$d_eK=0$.
\item Soient $K$ un $k$-tenseur et $L$ un $\ell$-tenseur. Si $K$ et $L$ sont
multiplicatifs, leur crochet de Schouten 
$[K,L]$ est un $(k+\ell-1)$-tenseur multiplicatif.
\end{enumerate}
\end{proposition}
\begin{proof}~
\begin{enumerate}
\item Si $K$ est multiplicatif et $d_eK=0$, alors en particulier $\mathscr{L}_XK(e)=0$
pour tout champ $X$ invariant à gauche. 
Or $\mathscr{L}_XK$ est lui-même invariant à gauche, donc $\mathscr{L}_XK=0$. Mais $G$ étant connexe
et la dérivée de Lie de $K$ par rapport à tout 
champ de vecteurs invariant à gauche $X$ est nulle, $K$ est donc invariant à droite, et
donc nul puisque nul en $e$.
\item Soit $X$ (resp. $Y$) un champ de vecteurs invariant à gauche (respectivement, à droite). On a, grâce à la formule de Leibniz :
\[\mathscr{L}_Y\mathscr{L}_X[K,L]=[\mathscr{L}_Y\mathscr{L}_XK,L]+[\mathscr{L}_XK,\mathscr{L}_YL]+[\mathscr{L}_YK,\mathscr{L}_XL]+[K,\mathscr{L}_Y\mathscr{L}_XL],\]
or $\mathscr{L}_XK$ et $\mathscr{L}_XL$ sont invariants à gauche, tandis que $\mathscr{L}_YK$ et $\mathscr{L}_YL$ sont invariants à droite. Donc
\[\mathscr{L}_Y\mathscr{L}_XK=\mathscr{L}_Y\mathscr{L}_XL=[\mathscr{L}_XK,\mathscr{L}_YL]=[\mathscr{L}_YK,\mathscr{L}_XL]=0,\]
d'où 
\[\mathscr{L}_Y\mathscr{L}_X[K,L]=0.\]
Donc $\mathscr{L}_X[K,L]$ est invariant à gauche. Comme de plus $[K,L](e)=0$, le Théorème
précèdent montre que $[K,L]$ est multiplicatif.
\end{enumerate}
\end{proof}
Le crochet de Koszul de deux $1$-formes invariantes à gauche, sur un groupe de Lie-Poisson, est une $1$-forme invariante à gauche. Plus précisément
\begin{proposition}\label{kos.inv}
Si $(G,\pi)$ est un groupe de Lie-Poisson et si $\alpha,\beta\in\Omega^1(G)$ sont des $1$-formes 
invariantes à gauche, alors, pour tout champ de vecteurs invariant à gauche sur $G$,
\begin{equation}\label{koszul.inv}
[\alpha,\beta](X)=\left(\mathscr{L}_X\pi\right)(\alpha,\beta).
\end{equation}
\end{proposition}
\begin{proof}
Soient $\alpha,\beta$ deux $1$-formes invariantes à gauche sur un groupe de Lie-Poisson $(G,\pi)$. Soit $X$ un champ invariant à gauche. On a d'après \eqref{differential}
\begin{align*}
\left(\mathscr{L}_X\pi\right)(\alpha,\beta)=[X,\pi](\alpha,\beta)=&-\delta_\pi(X)(\alpha,\beta)\\
=&-\pi_\sharp(\alpha)\cdot X(\beta)+\pi_\sharp(\beta)\cdot X(\alpha)+X([\alpha,\beta])\\
=&X([\alpha,\beta])=[\alpha,\beta](X).
\end{align*}
Comme $\pi$ est multiplicatif et $X$ est invariant à gauche, $\mathscr{L}_X\pi$ est invariant à gauche. $\left(\mathscr{L}_X\pi\right)(\alpha,\beta)$ est constante et donc aussi $[\alpha,\beta](X)$ ; ce qui signifie que $[\alpha,\beta]$ est invariant à gauche. Pour une autre preuve, voir \cite{we:dressing}.
\end{proof}
\paragraph{Bigèbres de Lie}
Ayant caractérisé les tenseurs multiplicatifs, nous passerons maintenant à la structure infinitésimale des groupes de Lie-Poisson.
\begin{definition} Une bigèbre de Lie est la donnée d'un triplet $(\G,\xi,\G^*)$, où
$\G$ est une algèbre de Lie, $\G^*$ son espace vectoriel dual et $\xi : \G\rightarrow\wedge^2\G,$ est un $1$-cocycle par rapport à l'action adjointe de $\G$ sur $\wedge^2\G$ : 
\[\xi([x,y])=\ad\nolimits_x\xi(y)-\ad\nolimits_y\xi(x),\ \text{pour tout}\ x,y\in\G,\]
tels que l'adjoint de $\xi$, $\xi^t : \wedge^2\G^*\rightarrow\G^*$, définit un crochet de Lie sur $\G^*$.
\end{definition}
Le théorème suivant montre qu'à tout groupe de Lie-Poisson est canoniquement associée une structure de bigèbre de Lie et réciproquement :
\begin{theorem}\label{theomultiplicatif} Soit $G$ un groupe de Lie et $\G$ son algèbre de Lie.
\begin{enumerate}
\item Si $K$ est un $k$-multivecteur multiplicatif sur $G$, alors $d_eK$ est un $1$-cocycle pour la représentation de $\G$ sur $\wedge^k\G$. Réciproquement, si $G$ est connexe et simplement connexe, alors pour tout $1$-cocycle $\xi : \G\rightarrow\wedge^k\G$ il existe un unique $k$-multivecteur $K\in\wedge^kTG$ tel que $d_eK=\xi$.
\item Si $\pi$ est un tenseur de Poisson sur $G$ tel que $\pi(e)=0$ et si $\xi=d_e\pi$, alors l'application adjointe $^t\xi$ définit une structure de Lie canonique sur le cotangent $\G^*$ qui coïncide avec la restriction du crochet de Koszul des $1$-formes invariantes à gauche en $e$. 
\end{enumerate}
\end{theorem}
\begin{proof}(Voir \cite{lu-we:poi}, \cite{duf-zun}.)
\begin{enumerate}
\item Soit $K$ un $k$-multivecteur sur un groupe $G$. Soient $X,Y$ deux éléments de $\G$ considérée comme l'espace des champs invariants à gauche. On pose $\xi=d_eK$ et on montre que 
\[\xi\left([X,Y]\right)=\ad\nolimits_X\xi(Y)-\ad\nolimits_Y\xi(X).\]
On a 
\begin{align*}
\ad\nolimits_X\xi(Y)=&\ad\nolimits_X\left(\mathscr{L}_YK(e)\right)=\frac{d}{dt}\Big|_{t=0}\Ad\nolimits_{\exp(tX)}\cdot\mathscr{L}_YK(e)\\
=&\frac{d}{dt}\Big|_{t=0}R_{\exp(-tX)^*}L_{\exp(tX)^*}\left(\mathscr{L}_YK(e)\right)\\
=&\frac{d}{dt}\Big|_{t=0}R_{\exp(-tX)^*}\left(\mathscr{L}_YK\right)\left(\exp(tX)\right)\\
=&\left(\mathscr{L}_X\mathscr{L}_YK\right)(e),
\end{align*}
et donc
\begin{align*}
\ad\nolimits_X\xi(Y)-\ad\nolimits_Y\xi(X)=&\left(\mathscr{L}_X\mathscr{L}_YK\right)(e)-\left(\mathscr{L}_Y\mathscr{L}_XK\right)(e)\\
=&\left(\mathscr{L}_{[X,Y]}K\right)(e)=\xi\left([X,Y]\right).
\end{align*}
Réciproquement\footnote{On reprend ici la preuve donnée dans \cite{lu-we:poi}. Voir aussi \cite{ksma:poilie}, \cite{lu}.}, soit $G$ le groupe de Lie connexe et simplement connexe d'algèbre de Lie $\G$ et soit $\xi : \G\rightarrow\wedge^k\G$ un $1$-cocycle pour la représentation adjointe. Il existe un unique $1$-cocycle \(\Lambda : G\rightarrow\wedge^k\G,\) qui intègre $\xi$, c'est-à-dire $d\Lambda=\xi$ et, pour tout $g,h\in G$
\[\Lambda(gh)=\Lambda(g)+\Ad(g)\cdot\Lambda(h).\]
Le champ de $k$-multivecteurs $K(g)=R_{g*}\Lambda$ est multiplicatif et $d_eK=\xi$. 
\item Soient $\alpha,\beta\in\G^*$ et $x\in\G$. Notons par $\alpha^+,\beta^+$ les $1$-formee invariantes à gauche associées à $\alpha,\beta$ et $X$ le champ de vecteurs invariant à gauche associé à $x$. On a d'après l'égalité \eqref{koszul.inv}
\begin{align*}
[\alpha,\beta](x)=&[\alpha^+,\beta^+]_e(X(e))=\left(\mathscr{L}_X\pi\right)(e)(\alpha,\beta)\\
=&\xi(x)(\alpha,\beta)=\xi^t(\alpha,\beta)(x)
\end{align*}
ce qui signifie que $\xi^t$ coïncide avec le crochet de Koszul des $1$-formes invariantes à gauche.
\end{enumerate}
\end{proof}
\begin{remark}
Ce qui est remarquable dans les bigèbres de Lie est le fait symétrique suivant :
Si $(\G,\xi,\G^*)$ est une bigèbre de Lie, alors $(\G^*,\rho,\G)$ est aussi une bigèbre de Lie, appelée \emph{bigèbre de Lie duale}, où $\rho : \G^*\rightarrow\wedge^2\G^*$ est le $1$-cocycle relative à l'action adjointe de $\G^*$ sur $\wedge^2\G^*$, dont la transposée est le crochet de Lie de $\G$. (Voir \cite{duf-zun}, page 137).
\end{remark} 
\paragraph{Équation de Yang-Baxter}
\begin{definition}
Soit $\G$ une algèbre de Lie réelle de dimension finie. On appelle \emph{équation de Yang-Baxter généralisée}, l'équation dont la variable est $r\in\wedge^2\G$ :
\begin{equation}\label{y-b.généralisée}
\ad\nolimits_x[r,r]=0,\ \text{pour tout}\ x\in\G.
\end{equation}
On appelle \emph{équation de Yang-Baxter classique}, l'équation dont la variable est $r\in\wedge^2\G$ :
\begin{equation}\label{y-b.classique}
[r,r]=0.
\end{equation}
\end{definition}
Soient $G$ un groupe de Lie et $\G$ son algèbre de Lie. On considère la restriction du crochet de Schouten sur les champs de multi-vecteurs invariants à gauche qui étend le crochet
de Lie de $\G$.\\
Soit $r\in\wedge^2\G$ et soit l'application associée $r_\sharp : \G^*\rightarrow\G$ définie, pour tout $\alpha,\beta\in\G^*$, par :
$$\beta\left(r_\sharp(\alpha)\right)=r(\alpha,\beta).$$
On considère le $1$-cobord associé à $r$, c'est-à-dire le cocycle $\xi : \G\rightarrow\wedge^2\G$ défini par 
\[\xi(x)=\ad\nolimits_x r.\]
Il est naturel de se demander à quelles conditions sur $r$, le triplet $(\G,\xi,\G)$ correspondant est une bigèbre de Lie.
Le Lemme suivant donne la condition, nécessaire et suffisante pour que $\xi^t$ définit un crochet de Lie sur $\G^*$.
\begin{lemma}
\begin{enumerate}
\item Le crochet défini par $\xi^t$ sur $\G^*$ est donné par
$$[\alpha,\beta]_r=\ad\nolimits_{r_\sharp(\beta)}^*\alpha-\ad\nolimits_{r_\sharp(\alpha)}^*\beta.$$
\item Le crochet $[\,,\,]_r$ est de Lie, si et seulement si, $r$ est solution de l'équation de Yang-Baxter, c'est-à-dire si $\ad_x[r,r]=0$ pour tout $x\in\G$.
\end{enumerate}
\end{lemma}
\begin{proof}
En effet, on a
\begin{enumerate}
\item Pour tout $\alpha,\beta\in\G^*$ et $x\in\G$,
\begin{align*}
[\alpha,\beta]_r(x)=&\langle\alpha\wedge\beta,\ad\nolimits_xr\rangle=r\left(\ad\nolimits_x^*(\alpha\wedge\beta)\right)=
r\left(\ad\nolimits_x^*\alpha\wedge\beta\right)+r\left(\alpha\wedge\ad\nolimits_x^*\beta\right)\\
=&-\ad\nolimits_x^*\alpha\left(r_\sharp(\beta)\right)+\ad\nolimits_x^*\beta\left(r_\sharp(\alpha)\right)
=\ad\nolimits_{r_\sharp(\beta)}^*\alpha(x)-\ad\nolimits_{r_\sharp(\alpha)}^*\beta(x),
\end{align*}
où on a utilisé l'égalité 
\[\langle\ad\nolimits_x^*\alpha,y\rangle=\langle\alpha,\ad\nolimits_xy\rangle=-\langle\alpha,\ad\nolimits_yx\rangle=-\langle\ad\nolimits_y^*\alpha,x\rangle.\]
\item Soit $G$ un groupe de Lie connexe d'algèbre de Lie $\G$. Notons par $r^+$, respectivement $r^-$, le champ invariant à gauche, respectivement invariant à droite, associé à $r$. Soit le champ de bivecteurs $\pi=r^+-r^-$. On a $\pi(e)=0$ et, pour tout champ de vecteurs $X$ invariant à gauche, $\mathscr{L}_X\pi=\mathscr{L}_Xr^+$ est aussi invariant à gauche et donc $\pi$ est multiplicatif. On déduit aussi que $d_e\pi=\xi$, car $\mathscr{L}_X\pi(e)=\mathscr{L}_Xr^+(e)=\xi(X(e))$. D'autre part, comme
\[[r^+,r^-]=0,\ [r^+,r^+]=[r,r]^+\ \text{et}\ [r^-,r^-]=-[r,r]^-,\]
on déduit que
\begin{align*}
[\pi,\pi]=0\Longleftrightarrow&[r,r]^+=[r,r]^-\\
\Longleftrightarrow&\Ad(g)\cdot[r,r]=[r,r],\ \text{pour tout}\ g\in G.
\end{align*}
Comme $G$ est connexe, alors 
\[\Ad(g)\cdot[r,r]=[r,r],\ \text{pour tout}\ g\in G\Longleftrightarrow\ad\nolimits_x[r,r]=0,\ \text{pour tout}\ x\in\G.\] D'après le théorème \eqref{theomultiplicatif}, $\xi^t$ définit un crochet de Lie sur $\G^*$, si et seulement si, $[\pi,\pi]=0$. On déduit alors que
\[\xi^t\ \text{définit un crochet de Lie sur}\ \G^*\Longleftrightarrow\ad\nolimits_x[r,r]=0,\ \text{pour tout}\ x\in\G.\]
Pour une preuve purement algébrique voir \cite{ksma:poilie}.
\end{enumerate}
\end{proof}

On termine cette partie par un résultat de cohomologie.
\begin{theorem}
Soit $G$ un groupe de Lie connexe. Si $G$ est compact ou semi-simple, alors tout tenseur de Poisson multiplicatif sur $G$ est de la forme
$$\pi=r^+-r^-$$
pour un $r\in\wedge^2\G$ tel que $[r,r]$ est invariant, par l'action adjointe de $\G$ sur $\wedge^3\G$.
\end{theorem}
\begin{proof} Voir Théorème 1.10 de l'article de Lu \& Wienstein \cite{lu-we:poi}.\\
Soit $\pi_r$ le translaté à droite du tenseur de Poisson multiplicatif $\pi$. Comme $G$ est compact, tout $1$-cocycle pour l'action adjointe de $G$ sur $\wedge^2\G$ est un $1$-cobord. En particulier, comme $\pi_r$ est un $1$-cocycle, il existe alors $r\in\wedge^2\G$ tel que $\pi_r=\Ad(g)\cdot r-r$. On a donc
\[\pi=R_{g*}\pi_r=R_{g*}\left(\Ad(g)\cdot r-r\right)=L_{g*}r-R_{g*}r=r^+-r^-.\]
Si $G$ est semi-simple, tout $1$-cocycle pour l'action adjointe de $\G$ sur $\wedge^2\G$ est un $1$-cobord. En particulier, comme $d_e\pi=\xi$ est un $1$-cocycle, il existe alors $r\in\wedge^2\G$ tel que $\xi=\delta r$ et le tenseur de Poisson multiplicatif associé est $r^+-r^-$. Par unicité (voir Théorème~\ref{theomultiplicatif}, page~\pageref{theomultiplicatif}), ce tenseur $r^+-r^-$ doit être égal à $\pi$.
\end{proof}
\chapter{Déformations des structures riemanniennes}\label{chapter2}
\epigraph{Nous devons nous rappeler que ce que nous observons n'est pas la Nature elle-même, mais la Nature soumise à notre méthode de questionnement.}{Werner Heisenberg}
\lettrine[lines=3, nindent=0em]{É}{tant} donnée une variété de Poisson $(M,\pi)$ et $\nabla$ une connexion covariante sur $M$, il est naturel de considérer des conditions de compatibilités entre le tenseur de Poisson $\pi$ et la connexion $\nabla$.\\
Si l'on convient de dire que $\pi$ et $\nabla$ sont compatibles si $\pi$ est parallèle\footnote{Par analogie avec la compatibilité habituelle entre une métrique riemannienne et sa connexion de Levi-Civita associée.}, c'est-à-dire si
\[\nabla\pi=0,\]
alors dans ce cas le tenseur $\pi$ doit être de rang constant. Ceci est dû au fait que la dérivée covariante préserve le rang du tenseur. On va bien que cette notion ne convient que pour étudier les variétés de Poisson régulières ; ce qui exclue tout les exemples intéressants de : variétés de Poisson linéaires, groupes de Lie-Poisson...\\
La notion de dérivation contravariante était introduite par Vaisman dans \cite{vai}, puis la notion de connexion contravariante été étudié systématiquement par Fernandes dans \cite{fer}. 
Soit $(M,\pi)$ une variété de Poisson et soit $\pi_\sharp : T^*M\rightarrow TM$ l'ancrage défini par
\[\beta\left(\pi_\sharp(\alpha)\right)=\pi(\alpha,\beta).\]
Rappelons que l'ensemble des $1$-formes est une algèbres de Lie, pour le crochet de Koszul
\[[\alpha,\beta]=\mathscr{L}_{\pi_\sharp(\alpha)}\beta-\mathscr{L}_{\pi_\sharp(\beta)}\alpha-d\left(\pi(\alpha,\beta)\right).\]
Cette structure d'algébroïde de Lie (voir \cite{pr:68}) de $(T^*M,\pi_\sharp,TM)$ fournit tout les ingrédients pour définir la notion de connexion contravariante sur toute variété de Poisson.
\begin{definition}
Une connexion contravariante, sur une variété de Poisson $(M,\pi)$ est la donnée d'une application $\mathbb{R}$-bilinéaire \(\mathcal{D} : \Omega^1(M)\times\Omega^1(M)\rightarrow\Omega^1(M)\)
telle que, pour tout $f\in C^\infty(M)$ et tout $\alpha,\beta\in\Omega^1(M)$,
\begin{equation*}
\mathcal{D}_\alpha\left(f\beta\right)=f\mathcal{D}_\alpha\beta+\pi_\sharp(\alpha)(f)\beta\quad\text{et}\quad \mathcal{D}_{f\alpha}\beta=f\mathcal{D}_\alpha\beta.
\end{equation*}
\end{definition}

À noter que la définition d'une connexion contravariante est similaire à la définition d'une connexion covariante : les $1$-formes remplacent les champs de vecteurs, le crochet de Koszul remplace le crochet de Lie et la dérivée suivant une $1$-forme $\alpha$ est la dérivée de Lie suivant le champ de vecteurs $\pi_\sharp(\alpha)$.\\
On peut traduire beaucoup de notions, liées aux connexions covariantes (transport parallèle, géodésiques, courbures...), au cas des connexions contravariantes. Toutefois, il y a des résultats qui ne subsistent plus dans le cas contravariant (voir \cite{fer}).\\
On peut définir, comme dans le cas covariant, la torsion :
\[T(\alpha,\beta)=\mathcal{D}_\alpha\beta-\mathcal{D}_\beta\alpha-[\alpha,\beta],\]
c'est un tenseur de type $(2,1)$.
On peut montrer que si $\nabla$ est une connexion covariante alors $\mathcal{D}_\alpha=\nabla_{\pi_\sharp(\alpha)}$ défini une connexion contravariante, et si $M$ est symplectique, c'est-à-dire si $\pi_\sharp$ est un isomorphisme, alors toute connexion contravariante est de cette forme.\\

La différence majeure avec le cas covariant est que $\mathcal{D}_\alpha$ n'induit pas nécessairement une dérivation. En effet, si $\alpha\in\Omega^1(M)$ et si $\pi_\sharp(\alpha)=0$, alors 
\[\mathcal{D}_\alpha(f\beta)=f\mathcal{D}_\alpha\beta,\]
contrairement aux connexions covariantes.
\begin{example}
\begin{enumerate}
\item Si $\nabla$ est une connexion covariante, alors $\mathcal{D}_\alpha=\nabla_{\pi_\sharp(\alpha)}$ est une connexion contravariante spéciale, elle vérifie :
\begin{equation}\label{conn.ind.cov}
\pi_\sharp(\alpha)=0\Longrightarrow\mathcal{D}_\alpha=0.
\end{equation}
Si $M$ est symplectique alors l'ancrage $\pi_\sharp : T^*M\rightarrow TM$ est un isomorphisme et toute connexion contravariante $\mathcal{D}$ est induite de la connexion covariante :
\[\nabla_x=\mathcal{D}_{\alpha},\ \text{où}\ \alpha=\pi_\sharp^{-1}(x).\]
\item Il existe des connexions contravariantes qui ne sont pas induites d'une connexion covariante; c'est les connexions qui ne vérifient pas \ref{conn.ind.cov}. En voici un exemple :\\ 
Soit le groupe de Lie-Poisson $G=(\mathbb{R}^2,+,\pi)$, avec $\pi=y\,\partial_x\wedge\partial_y$. On considère la connexion contravariante sur $\G^*$, définie par 
\[\mathcal{D}_\alpha\beta=\frac{1}{2}[\alpha,\beta].\] 
On peut vérifier que $\mathcal{D}$ est sans torsion et $\mathcal{D}\pi=0$. Mais $\pi_\sharp(dx)=y\partial_y$ est nul sur tout point $(x,0)$ alors que $\mathcal{D}_{dx}dy=\frac{1}{2}[dx,dy]=\frac{1}{2}d\{x,y\}=\frac{1}{2}dy$ ne s'annule en aucun point.
\end{enumerate}
\end{example}
La notion de connexion contravariante est devenue une notion centrale en géométrie de Poisson. Pour ses application, on peut citer :
\begin{enumerate}
\item Les travaux de Boucetta sur la compatibilité des structures de Riemann et de Poisson. Voir \cite{bou:compatibility} et \cite{bou:riemann-lie-Poisson}.
\item La réalisation symplectique. Une réalisation d'une variété de Poisson $M$ est un morphisme de Poisson surjectif $J : S\rightarrow M$ où $S$ est une variété symplectique et $J$ est une submersion. L'existence d'une réalisation symplectique pour toute variété de Poisson est un résultat important dû à Karasev \cite{ka:anal} et Weinstein \cite{we:sympoid}. Il s'énonce comme suit :
\begin{theorem}
Toute variété de Poisson de dimension $n$, admet une réalisation symplectique de dimension $2n$.
\end{theorem}
Crainic \& Marcut ont donné, dans \cite{cra-mar}, une autre preuve de l'existence de la réalisation symplectique utilisant la notion de connexion contravariante.
\item Le problème de linéarisation. Si $\pi=\frac{1}{2}\sum_{i<j}\pi_{ij}\,\partial_{x_i}\wedge\partial_{x_j}$ est un tenseur de Poisson, sur une variété $M$, qui s'annule en un point $a$, on peut considérer le développement en série de Taylor de $\pi_{ij}$
\[\pi_{ij}(x)=\sum_{k=1}^nC_{ij}^k\,x_k+\mathrm{O}(x^2)\]
Les $(C_{ij}^k)$ forment des constantes de structure d'une algèbre de Lie $\G_a$, appelée \emph{algèbre d'isotropie}, qui n'est autre que le cotangent $T_a^*M$, muni du crochet de Koszul
\[[\alpha,\beta]=[\overline{\alpha},\overline{\beta}](a),\]
où $\alpha,\beta\in T_a^*M$ et $\overline{\alpha}$ et $\overline{\beta}$ sont des $1$-formes qui prolongent $\alpha$ et $\beta$ respectivement.
Le problème de linéarisation s'énonce comme suit :
\begin{center}
Existe-t-il un $\varphi : T_aM\rightarrow M$ tel que $\varphi^*\pi=\pi_\ell$ ?
\end{center}
où $\pi_\ell$ est le tenseur de Poisson linéaire canonique de $T_aM$. On dira que $\pi$ est formellement (resp. analytiquement, $C^\infty$)-linéarisable, au voisinage de $a$, suivant que $\varphi$ est formel, analytique ou $C^\infty$-difféomorphisme local. On a le résultat suivant (voir Conn~\cite{con1},~\cite{con2} et Fernandes et Monnier~\cite{fer-mon})
\begin{theorem}
Si l'algèbre d'isotropie $\G_a$ est semi-simple, alors $\pi$ est formellement et analytiquement linéarisable.
\end{theorem}
On définit alors le produit $A : T_a^*M\times T_a^* M\rightarrow T_a^* M$ par
\[2<A_\alpha\beta,\gamma>=<[\alpha,\beta],\gamma>+<[\gamma,\alpha],\beta> + <[\gamma,\beta],\alpha>.\]
On peut vérifier que 
\[A_\alpha\beta=(\mathcal{D}_\alpha\beta)(a).\]
Ainsi les courbures de $\mathcal{D}$ en $a$ coïncident avec les courbures de $A$. L'étude de la courbure des algèbres de Lie munies d'un produit scalaire a été faite dans \cite{mil}.
\item La théorie de déformation. Hawkins ajoute de l'importance à cette notion de connexion. Il a en fait, montré que le couple $\left(C^\infty(M),d\right)$, où $d$ est la différentielle extérieure, est déformable s'il existe un crochet de Poisson $\{\;,\;\}$ sur $M$ et ce crochet s'étend à l'ensemble $\Omega^1(M)$ des $1$-formes comme crochet de Lie gradué, s'il existe une connexion contravariante $\mathcal{D}$ définie par \[\{f,\alpha\}=\mathcal{D}_{d\!f}\alpha\]
qui est plate et sans torsion. Voir \cite{haw1} et \cite{haw2}.
\end{enumerate}
\section{Connexion de Levi-Civita contravariante}
Soit $(M,\pi,<,>)$ une variété de Poisson munie d'une métrique riemannienne. Soit $\Omega^1(M)$ l'espace de ses $1$-formes différentielles muni du crochet de Koszul.\\
On peut définir, sur $M$, une connexion appelée \emph{connexion de Levi-Civita contravariante} associée au couple $(\pi,<,>)$ :
\begin{equation*}
\begin{array}{cccc}
\mathcal{D} : & \Omega^1(M)\times\Omega^1(M) & \rightarrow & \Omega^1(M)\\
     & (\alpha,\beta) & \mapsto & \mathcal{D}_\alpha\beta
  \end{array}
\end{equation*}
telles que :
\begin{enumerate}
\item $\mathcal{D}$ est $\mathbb{R}$-bilinéaire,
\item $\mathcal{D}$ est $C^\infty(M)$-linéaire par rapport 
$\alpha$, c'est-à-dire, pour tout $f\in C^\infty(M)$
\[\mathcal{D}_{f\alpha}=f\mathcal{D}_{\alpha},\]
\item  $\mathcal{D}$ est une dérivation par rapport 
$\beta$, c'est-à-dire, pour tout $f\in C^\infty(M)$
\[\mathcal{D}\left(f\beta\right)=f\mathcal{D}_\alpha\beta+\pi_\sharp(\alpha)(f)\beta,\]
\item $\mathcal{D}_\alpha\beta-\mathcal{D}_\beta\alpha=[\alpha,\beta]$ (sans
torsion),
\item $\pi_\sharp(\alpha)\cdot<\beta,\gamma>=<\mathcal{D}_\alpha\beta,\gamma>+
<\beta,\mathcal{D}_\alpha\gamma>$ (métrique),
\end{enumerate}
où $[\;,\;]$ est le crochet de Koszul.
\begin{proposition}
La connexion de Levi-Civita contravariante est entièrement
déterminée par :
\begin{multline*}
2<\mathcal{D}_\alpha\beta,\gamma>=\pi_\sharp(\alpha)\cdot<\beta,\gamma>+
\pi_\sharp(\beta)\cdot<\alpha,\gamma>-\pi_\sharp(\gamma)\cdot<\alpha,\beta>\\
+<[\alpha,\beta],\gamma>+<[\gamma,\alpha],\beta>+<[\gamma,\beta],\alpha>.
\end{multline*}
\end{proposition}
\begin{proof} La preuve ne diffère guère de celle de la connexion de Levi-Civita covariante. En effet, comme la connexion $\mathcal{D}$ est métrique, on a 
\begin{align*}
\pi_\sharp(\alpha)\cdot<\beta,\gamma>=&<\mathcal{D}_\alpha\beta,\gamma>+<\beta,\mathcal{D}_\alpha\gamma>,\\
\pi_\sharp(\beta)\cdot<\alpha,\gamma>=&<\mathcal{D}_\beta\alpha,\gamma>+<\alpha,\mathcal{D}_\beta\gamma>,\\
\pi_\sharp(\gamma)\cdot<\alpha,\beta>=&<\mathcal{D}_\gamma\alpha,\beta>+<\alpha,\mathcal{D}_\gamma\beta>,
\end{align*}
\end{proof}
et comme la connexion $\mathcal{D}$ est sans torsion, alors
\begin{multline*}
\pi_\sharp(\alpha)\cdot<\beta,\gamma>+\pi_\sharp(\beta)\cdot<\alpha,\gamma>-\pi_\sharp(\gamma)\cdot<\alpha,\beta>=<\mathcal{D}_\alpha\beta,\gamma>\\+<\mathcal{D}_\alpha\beta-[\alpha,\beta],\gamma>
-<[\gamma,\alpha],\beta>-<[\gamma,\beta],\alpha>.
\end{multline*}
Ce qui donne la formule de Koszul.
\begin{definition} La courbure associée à la connexion $\mathcal{D}$ est
l'opérateur\\
$R : \Omega^1(M)\times\Omega^1(M)\times\Omega^1(M)\longrightarrow\Omega^1(M)$
défini par :
\begin{equation}\label{courbure}
R(\alpha,\beta)\gamma=\mathcal{D}_{[\alpha,\beta]}\gamma-\left(\mathcal{D}_\alpha \mathcal{D}_\beta\gamma-\mathcal{D}_\beta\mathcal{D}_\alpha\gamma\right).
\end{equation}
\end{definition}
\begin{proposition} L'opérateur de courbure $R$ est un tenseur de type $(1,3)$ qui vérifie :
\begin{enumerate} 
\item $R(\alpha,\beta)=-R(\beta,\alpha)$,
\item $R(f\alpha,\beta)\gamma=R(\alpha,f\beta)\gamma=R(\alpha,\beta)f\gamma=
fR(\alpha,\beta)\gamma$, pour tout $f\in C^\infty(M)$,
\item $R(\alpha,\beta)\gamma+R(\beta,\gamma)\alpha+R(\gamma,\alpha)\beta=0$. (Identité de
Bianchi),
\item $<R(\alpha,\beta)\gamma,\mu>=-<\gamma,R(\alpha,\beta)\mu>$,
\item $<R(\alpha,\beta)\gamma,\mu>=<R(\gamma,\mu)\alpha,\beta>$.
\end{enumerate}
\end{proposition}
\begin{proof} L'opérateur $R$ s'identifie à
\begin{equation*}
\begin{array}{cccc}
T : &\mathfrak{X}(M)\times\Omega^1(M)\times\Omega^1(M)\times\Omega^1(M)&\longrightarrow &C^\infty(M)\\
&(X,\alpha,\beta,\gamma)&\longmapsto &\langle R(\alpha,\beta)\gamma,X\rangle
\end{array}
\end{equation*}
qui est symétrique par rapport à $(\alpha,\beta)$ et $C^\infty(M)$-linéaire par rapport à chacune des composantes $(\alpha,\beta,\gamma)$. En effet,
\begin{enumerate}
\item $R(\beta,\alpha)\gamma=\mathcal{D}_{[\beta,\alpha]}\gamma-\mathcal{D}_{\beta}\mathcal{D}_{\alpha}\gamma+
\mathcal{D}_{\alpha}\mathcal{D}_{\beta}\gamma=-R(\alpha,\beta)\gamma$.
\item Pour tout $f\in C^\infty(M)$, on a
\begin{align*}
R(\alpha,f\beta)\gamma=&\mathcal{D}_{[\alpha,f\beta]}\gamma-\mathcal{D}_\alpha
\mathcal{D}_{f\beta}\gamma+\mathcal{D}_{f\beta}\mathcal{D}_\alpha\gamma\\
=&\mathcal{D}_{\left(f[\alpha,\beta]+\left(\pi_\sharp(\alpha)f\right)\beta\right)}\gamma-\mathcal{D}_\alpha
\left(f\mathcal{D}_\beta\gamma\right)+f\mathcal{D}_\alpha \mathcal{D}_\beta\gamma\\
=&f\mathcal{D}_{[\alpha,\beta]}\gamma+\left(\pi_\sharp(\alpha)f\right)\mathcal{D}_\beta\gamma-f\mathcal{D}_\alpha\mathcal{D}_\beta\gamma-\left(\pi_\sharp(\alpha)f\right)\mathcal{D}_\beta\gamma+
f\mathcal{D}_\alpha\mathcal{D}_\beta\gamma\\
=&fR(\alpha,\beta)\gamma,
\end{align*}
et
$R(f\alpha,\beta)\gamma=-R(\beta,f\alpha)\gamma=-fR(\beta,\alpha)\gamma=fR(\alpha,\beta)\gamma$, et
\begin{align*}
R(\alpha,\beta)f\gamma=&\mathcal{D}_{[\alpha,\beta]}f\gamma-\mathcal{D}_\alpha
\mathcal{D}_{\beta}f\gamma+\mathcal{D}_{\beta}\mathcal{D}_\alpha f\gamma\\
=&f\mathcal{D}_{[\alpha,\beta]}\gamma+\left(\pi_\sharp([\alpha,\beta])f\right)\gamma-\mathcal{D}_\alpha
\left(f\mathcal{D}_\beta\gamma+\left(\pi_\sharp(\beta)f\right)\gamma\right)\\
&+\mathcal{D}_\beta\left(f\mathcal{D}_\alpha\gamma+\left(\pi_\sharp(\alpha)f\right)\gamma\right)\\
=&f\mathcal{D}_{[\alpha,\beta]}\gamma+\left(\pi_\sharp([\alpha,\beta])f\right)\gamma-f\mathcal{D}_\alpha
\mathcal{D}_\beta\gamma-\left(\pi_\sharp(\alpha)f\right)\mathcal{D}_\beta\gamma\\
&-\left(\pi_\sharp(\beta)f\right)\mathcal{D}_\alpha\gamma-\left(\pi_\sharp(\alpha)(\pi_\sharp(\beta)f)\right)\gamma+ f\mathcal{D}_\beta\mathcal{D}_\alpha\gamma+\left(\pi_\sharp(\beta)f\right)\mathcal{D}_\alpha\gamma\\
&+\left(\pi_\sharp(\alpha)f\right)\mathcal{D}_\beta\gamma+\left(\pi_\sharp(\beta)(\pi_\sharp(\alpha)f)\right)\gamma\\
=&fR(\alpha,\beta)\gamma+\left(\pi_\sharp([\alpha,\beta])f\right)\gamma-
\left([\pi_\sharp(\alpha),\pi_\sharp(\beta)]f\right)\gamma\\
=&fR(\alpha,\beta)\gamma,
\end{align*}
car
$[\pi_\sharp(\alpha),\pi_\sharp(\beta)]=\pi_\sharp\left([\alpha,\beta]\right)$.
\item Comme la connexion $\mathcal{D}$ est sans torsion, alors
\begin{align*}
R(\alpha,\beta)\gamma=&\ \mathcal{D}_\gamma[\alpha,\beta]+[[\alpha,\beta],\gamma]-\mathcal{D}_\alpha\mathcal{D}_\beta\gamma+\mathcal{D}_\beta\mathcal{D}_\alpha\gamma\\
=&\ \mathcal{D}_\gamma\mathcal{D}_\alpha\beta-\mathcal{D}_\alpha\mathcal{D}_\gamma\beta+[[\alpha,\beta],\gamma]-\mathcal{D}_\alpha\mathcal{D}_\beta\gamma+\mathcal{D}_\beta\mathcal{D}_\alpha\gamma.
\end{align*}
On déduit que
\[R(\alpha,\beta)\gamma+R(\beta,\gamma)\alpha+R(\gamma,\alpha)\beta+R(\alpha,\beta)\gamma=[[\alpha,\beta],\gamma]+[[\beta,\gamma],\beta]+[[\gamma,\alpha],\beta]=0.\]
\item Il suffit de montrer que $<R(\alpha,\beta)\gamma,\gamma>=0$, pour tout $\gamma\in\Omega^1(M)$, puis de conclure en développant l'identité
\[<R(\alpha,\beta)(\gamma+\mu),\gamma+\mu>=0.\]
Comme la connexion $\mathcal{D}$ est métrique, on a 
\begin{align*}
\pi_\sharp(\alpha)\cdot\pi_\sharp(\beta)\cdot<\gamma,\gamma>=&\pi_\sharp(\alpha)\cdot2<\mathcal{D}_\beta\gamma,\gamma>\\
=&2<\mathcal{D}_\alpha\mathcal{D}_\beta\gamma,\gamma>+2<\mathcal{D}_\beta\gamma,\mathcal{D}_\alpha\gamma>.
\end{align*}
De même, $\pi_\sharp(\beta)\cdot\pi_\sharp(\alpha)\cdot<\gamma,\gamma>=2<\mathcal{D}_\beta\mathcal{D}_\alpha\gamma,\gamma>+2<\mathcal{D}_\alpha\gamma,\mathcal{D}_\beta\gamma>$.\\ On déduit que
\begin{equation}\label{anti1}
\pi_\sharp\left([\alpha,\beta]\right)\cdot<\gamma,\gamma>=[\pi_\sharp(\alpha),\pi_\sharp(\gamma)]\cdot<\gamma,\gamma>,
\end{equation}
et comme, d'autre part on a
\begin{equation}\label{anti2}
\pi_\sharp\left([\alpha,\beta]\right)\cdot<\gamma,\gamma>=2<\mathcal{D}_{[\alpha,\beta]}\gamma,\gamma>,
\end{equation}
on déduit alors, de \eqref{anti1} et \eqref{anti2}, que 
\[<R(\alpha,\beta)\gamma,\gamma>=0,\ \text{pour tout}\ \gamma\in\Omega^1(M).\]
\item La preuve de cette propriété est plus subtile. Il s'agit d'utiliser judicieusement les autres propriétés ci-dessus. La preuve est identique au cas covariant qu'on peut trouver, par exemple, dans le livre de Lee \cite{lee}, page 123.
\end{enumerate}
\end{proof}
\section{Crochet de Poisson généralisé et métacourbure}
Soit $M$ une variété différentielle et soit $(\Om^*(M),\wedge)$ l'espace de ses formes différentielles  muni du produit extérieur $\wedge$, qui est une algèbre commutative au sens gradué 
\[\alpha\wedge\beta=(-1)^{\deg\alpha\deg\beta}\beta\wedge\alpha.\]
Une déformation non commutative de $(\Om^*(M),\wedge)$ est la donnée d'une algèbre graduée $(\mathcal{A},\cdot)$, non nécessairement commutative, qui soit une extension de $(\Om^*(M),\wedge)$, c'est-à-dire, qu'il existe une suite exacte d'algèbres sur $\mathbb{R}$ :
\[0\rightarrow\mathcal{N}\rightarrow\mathcal{A}\stackrel{\mathcal{P}}{\longrightarrow}\Om^*(M)\rightarrow0,\]
où $\mathcal{N}$ s'identifie à une sous-algèbre centrale de $\mathcal{A}$ et $\mathcal{P}$ est surjectif et réalise un isomorphisme entre $\Om^*(M)$ et $\mathcal{A}/\mathcal{N}$.
Si $\alpha,\beta$ sont deux éléments de $\Om^*(M)$, la formule
\begin{equation}\label{crochetpoisson}
\left\{\alpha,\beta\right\}=\mathcal{P}\left([\wi\alpha,\wi\beta]\right),\end{equation}
où $\wi\al,\wi\be$ sont deux antécédents de $\alpha$ et $\beta$, définit sur
$\Om^*(M)$ un crochet, appelé par Hawkins {\it crochet de Poisson généralisé}. Hawkins montre alors que $(\Om^*(M),\wedge,d,\{\;,\;\})$ est une algèbre de Poisson différentielle graduée (voir chapitre \ref{chapter2}). Il montre aussi que le crochet de Poisson généralisé sur $\Om^*(M)$ est entièrement déterminé par les crochets
\begin{equation}\label{crochetpoisson1}
\left\{f,g\right\}\esp \left\{f,\alpha\right\}\quad f,g\in C^\infty(M),\ \alpha\in\Om^1(M).\end{equation}
Le premier crochet donne naissance à une {\it structure de Poisson} sur $M$ et donc un tenseur de Poisson $\pi\in\Gamma(\wedge^2TM)$ et, en posant
\begin{equation}\label{connexion}\mathcal{D}_{df}\alpha:=\left\{f,\al\right\},\end{equation}
on définit $$\mathcal{D}:\Om^1(M)\times\Om^1(M)\too\Om^1(M),$$ qui est en fait une {\it connexion contravariante}.\\
Inversement, étant donnée une structure de Poisson sur $M$ et une connexion contravariante sans courbure ni torsion $\mathcal{D}$, les formules \eqref{crochetpoisson1} et \eqref{connexion} se généralisent pour définir sur $\Om^*(M)$ un crochet $\{\;,\;\}$ compatible avec $\wedge$ et $d$ (voir chapitre \ref{chapter2}). En général, ce crochet ne vérifie pas l'identité de Jacobi graduée. Hawkins a mis en évidence un tenseur ${\mathcal M}$  de type $(2,3)$ appelé {\it métacourbure} et a montré que ce crochet $\{\;,\;\}$ vérifie l'identité de Jacobi graduée, si et seulement si, ${\mathcal M}$ est identiquement nul.\\
En conclusion, une structure d'algèbre de Poisson graduée sur $\Omega^\star(M)$ est déterminée par
des "crochets initiaux"
$$\{f,g\}=\pi(df,dg),\{f,\alpha\}=\mathcal{D}_{df}\alpha,\{\alpha,\beta\}=\Psi(\alpha,\beta),$$
où $\pi$ est un bivecteur Poisson, $\mathcal{D}$ est une connexion contravariante plate et sans torsion et
$\Psi$ est un opérateur bi-différentiel d'ordre $1$ qui satisfait des condition assurent l'identité de Jacobi. Le tenseur de métacourbure est un objet géométrique intéressant qui est l'ostruction à la déformation non commutative de l'algèbre différentielle de Poisson graduée $(\Om^*(M),\wedge,d,\{\;,\;\})$.
\paragraph{Métacourbure}
Soit $(M,\pi)$ une variété de Poisson et soit $\mathcal{D}$ une connexion contravariante sans torsion sur $M$. Il existe un unique crochet sur l'espace des formes différentielles $\Omega^*(M)$ tel que :
\begin{enumerate}
\item $\{\;,\;\}$ est $\mathbb{R}$-bilinéaire, antisymétrique et de degré $0$, c'est-à-dire
\[\{\alpha,\beta\}=-(-1)^{\deg\alpha\deg\beta}\{\beta,\alpha\},\quad\deg\{\alpha,\beta\}=\deg\alpha+\deg\beta.\]
\item La différentielle extérieure $d$ est une dérivation par rapport à $\{\;,\;\}$, c'est-à-dire
\[d\{\alpha,\beta\}=\{d\alpha,\beta\}+(-1)^{\deg\alpha}\{\alpha,d\beta\}.\]
\item Le crochet $\{\;,\;\}$ vérifie l'identité du produit
\[\{\alpha,\beta\wedge\gamma\}=\{\alpha,\beta\}\wedge\gamma+\left(-1\right)^{\deg\alpha\deg\beta}\beta\wedge\{\alpha,\gamma\}.\]
\item Pour tout $f,g\in C^\infty(M)$ et pour tout $\alpha\in\Omega^\star(M)$, le crochet $\{f,g\}$ coincide avec le crochet de Poisson initial sur $M$ et
\[\{f,\alpha\}=\mathcal{D}_{df}\alpha.\]
\end{enumerate}
Ce crochet est appelé \emph{crochet de Poisson généralisé}.\\
On associe à tout crochet de Poisson généralisé, un crochet $\mathbb{R}$-trilinéaire, appelé \emph{jacobiateur} $J : \Omega^\star(M)\times\Omega^\star(M)\times\Omega^\star(M)\rightarrow\Omega^\star(M)$, défini par
\begin{equation}
J(\alpha,\beta,\gamma)=\{\{\alpha,\beta\},\gamma\}-\{\alpha,\{\beta,\gamma\}\}+\left(-1\right)^{\deg\alpha\deg\beta}\{\beta,\{\alpha,\gamma\}\}.
\end{equation}
On peut vérifier que
\begin{equation}
J(\beta,\alpha,\gamma)=-\left(-1\right)^{\deg\alpha\deg\beta}J(\alpha,\beta,\gamma),\ \quad J(\alpha,\gamma,\beta)=-\left(-1\right)^{\deg\beta\deg\gamma}J(\alpha,\beta,\gamma),
\end{equation}
\begin{equation}
dJ(\alpha,\beta,\gamma)=J(d\alpha,\beta,\gamma)+\left(-1\right)^{\deg\alpha}J(\alpha,d\beta,\gamma)+\left(-1\right)^{\deg\alpha+\deg\beta}J(\alpha,\beta,d\gamma).
\end{equation}
À noter que
\begin{equation}
J(f,g,h)=\{\{f,g\},h\}+\{\{g,h\},f\}+\{\{h,f\},g\}=0,
\end{equation}
et
\begin{equation}
J(f,g,\alpha)=\mathcal{D}_{[df,dg]}\alpha-\left(\mathcal{D}_{df}\mathcal{D}_{dg}\alpha-\mathcal{D}_{dg}
\mathcal{D}_{df}\alpha\right)=R(df,dg)\alpha,
\end{equation}
où $R$ est la courbure de $\mathcal{D}$. 
On a 
\begin{align*}
J(\alpha\wedge\beta,\gamma,\mu)=&\phantom{-}\{\{\alpha\wedge\beta,\gamma\},\mu\}-\{\alpha\wedge\beta,\{\gamma,\mu\}\}+\left(-1\right)^{(\deg\alpha+\deg\beta)\deg\gamma}\{\gamma,\{\alpha\wedge\beta,\mu\}\}\\
=&\phantom{-}\{\alpha\wedge\{\beta,\gamma\},\mu\}+(-1)^{\deg\beta\deg\gamma}\{\{\alpha,\gamma\}\wedge\beta,\mu\}\\
&-\alpha\wedge\{\beta,\{\gamma,\mu\}\}-(-1)^{\deg\beta(\deg\gamma+\deg\mu)}\{\alpha,\{\gamma,\mu\}\}\wedge\beta\\
&+(-1)^{(\deg\alpha+\deg\beta)\deg\gamma}\{\gamma,\alpha\wedge\{\beta,\mu\}\}\\
&+(-1)^{(\deg\alpha+\deg\beta)\deg\gamma}(-1)^{\deg\beta\deg \mu}\{\gamma,\{\alpha,\mu\}\wedge\beta\}\\
=&\phantom{-}\alpha\wedge\{\{\beta,\gamma\},\mu\}+(-1)^{(\deg\beta+\deg\gamma)\deg\mu}\{\alpha,\mu\}\wedge\{\beta,\gamma\}\\
&+(-1)^{\deg\beta\deg\gamma}\{\alpha,\gamma\}\wedge\{\beta,\mu\}+(-1)^{\deg\beta(\deg\gamma+\deg\mu)}\{\{\alpha,\gamma\},\mu\}\wedge\beta\\
&-\alpha\wedge\{\beta,\{\gamma,\mu\}\}-(-1)^{\deg\beta(\deg\gamma+\deg\mu)}\{\alpha,\{\gamma,\mu\}\}\wedge\beta\\
&+(-1)^{(\deg\alpha+\deg\beta)\deg\gamma}\{\gamma,\alpha\}\wedge\{\beta,\mu\}\\
&+(-1)^{(\deg\alpha+\deg\beta)\deg\gamma}(-1)^{\deg\alpha\deg\gamma}\alpha\wedge\{\gamma,\{\beta,\mu\}\}\\
&+(-1)^{(\deg\alpha+\deg\beta)\deg\gamma}(-1)^{\deg\beta\deg\mu}\{\gamma,\{\alpha,\mu\}\}\wedge\beta\\
&+(-1)^{(\deg\alpha+\deg\beta)\deg\gamma}(-1)^{\deg\beta\deg\mu}(-1)^{(\deg\alpha+\deg\mu)\deg\gamma}\{\alpha,\mu\}\wedge\{\gamma,\beta\}\\
=&\phantom{-}\alpha\wedge\left(\{\{\beta,\gamma\},\mu\}-\{\beta,\{\gamma,\mu\}\}+(-1)^{\deg\beta\deg\gamma}\{\gamma,\{\beta,\mu\}\}\right)\\
&+(-1)^{\deg\beta(\deg\gamma+\deg\mu)}\Big(\{\{\alpha,\gamma\},\mu\}-\{\alpha,\{\gamma,\mu\}\}\\
&+(-1)^{\deg\alpha\deg\gamma}\{\gamma,\{\alpha,\mu\}\}\Big)\wedge\beta\\
=&\phantom{-}\alpha\wedge J(\beta,\gamma,\mu)+\left(-1\right)^{\deg\beta(\deg\gamma+\deg\mu)}J(\alpha,\gamma,\mu)\wedge\beta,
\end{align*}
et donc
\begin{equation}
J(\alpha\wedge\beta,\gamma,\mu)=\alpha\wedge J(\beta,\gamma,\mu)+\left(-1\right)^{\deg\beta(\deg\gamma+\deg\mu)}J(\alpha,\gamma,\mu)\wedge\beta.
\end{equation}
On déduit que, pour tout $f,g\in C^\infty(M)$ et tout $\alpha,\beta\in\Omega^1(M)$,  
\begin{align}\label{tensoriel}
J(fg,\alpha,\beta)=&fJ(g,\alpha,\beta)+gJ(f,\alpha,\beta),\\
J(f,g\alpha,\beta)=&gJ(f,\alpha,\beta)+\alpha\wedge R(df,dg)\beta.
\end{align}
Le jacobiateur défini un tenseur symétrique de type $(2,3)$ 
\begin{proposition}
Si la courbure $R$ de la connexion contravariante $\mathcal{D}$ est nulle, alors la quantité
\begin{equation}\label{metacourbure}
\mathcal{M}(df,\alpha,\beta)=J(f,\alpha,\beta),
\end{equation}
définit, pour tout $f\in C^\infty(M)$ et tout $\alpha,\beta\in\Omega^1(M)$, un tenseur symétrique $2$-fois covariant et $3$-fois contravariant.
\end{proposition}
\begin{proof}
De \ref{tensoriel}, on déduit que $\mathcal{M} : \Omega^1(M)\times\Omega^1(M)\times\Omega^1(M)\rightarrow\Omega^2(M)$, est telle que $\mathcal{M}(df,\alpha,\beta)$ est $C^\infty(M)$-linéaire par rapport à $\alpha$ (et par rapport à $\beta$). Comme $\mathcal{M}$ s'identifie à
\begin{equation*}
\begin{array}{cccc}
T : &\mathfrak{X}(M)\times\mathfrak{X}(M)\times\Omega^1(M)\times\Omega^1(M)\times\Omega^1(M)&\longrightarrow &C^\infty(M)\\
&(X,Y,\alpha,\beta,\gamma)&\longmapsto &\langle\mathcal{M}(\alpha,\beta,\gamma),X\wedge Y\rangle
\end{array}
\end{equation*}
alors $\mathcal{M}$ est un tenseur de type $(2,3)$.\\
D'autre part, on a $\mathcal{M}(df,\alpha,\beta)=\mathcal{M}(df,\beta,\alpha)$ et comme $R\equiv0$ alors
\[0=d\left(R(df,dg)dh\right)=dJ(f,g,dh)=J(df,g,dh)+J(f,dg,dh),\]
et donc
\[\mathcal{M}(df,dg,dh)=J(f,dg,dh)=-J(df,g,dh)=J(g,df,dh)=\mathcal{M}(dg,df,dh),\]
d'où $\mathcal{M}$ est symétrique.
\end{proof}
Le tenseur $\mathcal{M}$ est appelé \emph{métacourbure}\footnote{Voir l'article de Hawkins.} et c'est l'obstruction à la nullité du jacobiateur. En effet, Si $J\equiv0$ alors $\mathcal{M}\equiv0$ et réciproquement si $\mathcal{M}$ est nulle alors pour tout $f,g,h\in C^\infty(M)$ et tout $\alpha,\beta\in\Omega^1(M)$
\[J(f,g,h)=0,\ J(f,g,\alpha)=R(df,dg)\alpha=0,\ J(f,\alpha,\beta)=\mathcal{M}(df,\alpha,\beta)=0.\]
On va donner une formule explicite pour le crochet de Poisson généralisé de deux $1$-formes différentielles. Pour cela, on a besoin du résultat suivant
\begin{definition} Soi $(M,\pi)$ une variété de Poisson et soit $[\;,\;]$ le crochet de Koszul.
Il existe sur $\Omega^\star(M)$ un crochet, qu'on notera simplement $[\;,\;]$, appelé \emph{crochet de Koszul généralisé}, tel que :
\begin{enumerate}
\item le crochet est de degré $-1$, c'est-à-dire
\[\deg[\alpha,\beta]=\deg\alpha+\deg\beta-1,\]
\item le crochet vérifie une antisymétrie graduée
\[[\alpha,\beta]=-(-1)^{(\deg\alpha-1)(\deg\beta-1)}[\beta,\alpha],\]
\item le crochet vérifie l'identité de Leibniz
\begin{align}\label{leibniz.koszul}
[\alpha,\beta\wedge\gamma]=&[\alpha,\beta]\wedge\gamma+(-1)^{(\deg\alpha-1)\deg\beta}\beta\wedge[\alpha,\gamma]\\
[\alpha\wedge\beta,\gamma]=&\alpha\wedge[\beta,\gamma]+(-1)^{(\deg\gamma-1)\deg\beta}[\alpha,\gamma]\wedge\beta,
\end{align}
\item le crochet vérifie l'identité de Jacobi graduée
\begin{multline}\label{Jacobi.koszul}
(-1)^{(\deg\alpha-1)(\deg\gamma-1)}[\alpha,[\beta,\gamma]]+
(-1)^{(\deg\beta-1)(\deg\alpha-1)}[\beta,[\gamma,\alpha]]\\
+(-1)^{(\deg\gamma-1)(\deg\beta-1)}[\gamma,[\alpha,\beta]]=0,
\end{multline}
\item le crochet généralisé de deux $1$-formes coïncide avec leur crochet de Koszul et
\[[\alpha,f]=\pi_\sharp(\alpha)(f),\]
pour toute $f\in C^\infty(M)$ et toute $\alpha\in\Omega^1(M)$.
\end{enumerate}
À remarquer l'analogie de la définition du crochet de Koszul généralisé, pour les formes différentielles, avec la définition du crochet de Schouten, pour les champs de multivecteurs. L'existence se démontre donc de la même façon.
\end{definition}
\begin{proposition} Soit $(M,\pi)$ une variété de Poisson et $\mathcal{D}$ une connexion contravariante sans torsion. Pour toutes $\alpha,\beta\in\Omega^1(M)$ on a
\begin{equation}\label{crochet1formes}
\{\alpha,\beta\}=-\mathcal{D}_\alpha d\beta-\mathcal{D}_\beta d\alpha+d\mathcal{D}_\beta\alpha+[\alpha,d\beta].
\end{equation}
\end{proposition}
\begin{proof} Soient $\alpha,\beta\in\Omega^1(P)$, avec $\beta=fdg$ où $f,g\in C^\infty(M)$. On a
\begin{align*}
\{\alpha,fdg\}=&\{\alpha,f\}\wedge dg+f\{\alpha,dg\}\\
=&-\mathcal{D}_{df}\alpha\wedge dg+f\left(d\mathcal{D}_{dg}\alpha-\mathcal{D}_{dg}d\alpha\right)\\
=&-\mathcal{D}_{fdg}d\alpha+d\mathcal{D}_{fdg}\alpha-\mathcal{D}_{df}\alpha\wedge dg-df\wedge \mathcal{D}_{dg}\alpha\\
=&-\mathcal{D}_{fdg}d\alpha+d\mathcal{D}_{fdg}\alpha-\mathcal{D}_\alpha\left(df\wedge dg\right)-[df,\alpha]\wedge dg-df\wedge[dg,\alpha]\\
=&-\mathcal{D}_{fdg}d\alpha-\mathcal{D}_\alpha\left(d(fdg)\right)+d\mathcal{D}_{fdg}\alpha-[df,\alpha]\wedge dg-df\wedge[dg,\alpha]\\
=&-\mathcal{D}_{fdg}d\alpha-\mathcal{D}_\alpha\left(d(fdg)\right)+d\mathcal{D}_{fdg}\alpha+[\alpha,d(fdg)]\\
=&-\mathcal{D}_\alpha d\beta-\mathcal{D}_\beta d\alpha+d\mathcal{D}_\beta\alpha+[\alpha,d\beta].
\end{align*}
\end{proof}
Pour le calcul de métacourbure, on dispose du résultat suivant
\begin{proposition}~
\begin{enumerate}
\item Pour toute $1$-forme parallèle $\alpha$ et pour toute $1$-forme $\beta$,
\begin{equation}\label{1parallel}
\{\alpha,\beta\}=-\mathcal{D}_\beta d\alpha.
\end{equation}
\item Pour toutes $1$-formes parallèles $\beta$, $\gamma$ et pour toute $1$-forme $\alpha$,
\begin{equation}\label{2parallel}
\mathcal{M}(\alpha,\beta,\gamma)=-\mathcal{D}_\alpha\mathcal{D}_\beta d\gamma.
\end{equation}
\end{enumerate}
\end{proposition}
\begin{proof}
Si $\alpha$ est une $1$-forme parallèle, c'est-à-dire $\mathcal{D}\alpha=0$ et si $\beta=fdg$, alors
\begin{align*}
\{\alpha,\beta\}=\{\alpha,fdg\}&=\{\alpha,f\}\wedge dg+f\{\alpha,dg\}=-\mathcal{D}_{df}\alpha\wedge dg+f\{dg,\alpha\}\\
&=f\left(d\{g,\alpha\}-\{g,d\alpha\}\right)=-f\{g,d\alpha\}=-f\mathcal{D}_{dg}d\alpha\\
&=-\mathcal{D}_{fdg}d\alpha=-\mathcal{D}_\beta d\alpha.
\end{align*}
Si $\beta$ et $\gamma$ sont parallèles alors
\begin{equation*}
\mathcal{M}(df,\beta,\gamma)=\mathcal{D}_{df}\{\beta,\gamma\}=-\mathcal{D}_{df}\mathcal{D}_\beta d\gamma,
\end{equation*}
ce qui donne l'égalité \eqref{2parallel} pour tout $\alpha\in\Omega^1(M)$.
\end{proof}
\chapter{Problème algébrique de déformation}\label{chapter3}
\epigraph{La fantaisie et la liberté d'imagination ne s'acquièrent pas comme ça, qu'il y faut du temps, de l'obstination, de la sévérité, de la rigueur, des mathématiques, de la raison.}{Philippe Sollers.}
\lettrine[lines=3, nindent=0em]{D}{ans} ce chapitre, on va montrer que le problème géométrique de déformation, au sens de Hawkins, est équivalent à un problème algébrique. Pour cela on va considérer un groupe de Lie-Poisson $(G,\pi,\prs)$ muni d'une métrique riemannienne invariante à gauche et on va traduire, une à une, les conditions de déformation de Hawkins :
\begin{description}
\item[Première condition.] La connexion de Levi-Civita contravariante $\mathcal{D}$, associée au couple $(\pi,\prs)$, doit être plate.
\item[Deuxième condition.] La métacourbure $\mathcal{M}$ de la connexion $\mathcal{D}$ doit être nulle ($\mathcal{D}$ est \emph{métaplate}).
\item[Troisième condition.] Le tenseur de Poisson $\pi$ doit être compatible avec toute forme riemannien $\mu$, c'est-à-dire \[d\left(i_\pi\mu\right)=0.\]
\end{description}
On va montrer que 
\begin{enumerate}
\item La première condition équivaut à dire que $\G^*$ est de Milnor. Ceci signifie que $\G^*$ se décompose en somme orthogonale de $S=\{\alpha\in\G^*\mid \ad_\alpha+\ad_\alpha^t=0\}$ (qui est une sous-algèbre abélienne) et de son idéal dérivé $[\G^*,\G^*]$ (qui est abélien de dimension paire).
\item La deuxième condition équivaut à dire que $\ad_\alpha\ad_\beta\rho(\gamma)=0$ pour tout $\alpha,\beta,\gamma\in S$, où $\rho$ est le $1$-cocyle de la bigèbre de Lie duale de $(\G,\xi)$.
\item La troisième condition (qui est indépendante des des deux premières) est équivalente, dans le cas où $\G$ et $\G^*$ sont unimodulaire, à la condition\footnote{Cette condition est toujours nécessaire. Elle est suffisante si $\G$ et $\G^*$ sont unimodulaires. À noter que si $\G^*$ est de Milnor alors elle est unimodulaire et l'hypothèse d'unimodularité d'une algèbre de Lie n'est pas restrictive, puisque l'ensemble des algèbres de Lie unimodulaires, de dimension $n$, est dense dans l'ensemble de toutes les algèbres de Lie de dimension $n$.}
\[d\left(i_{\xi(x)}\mu\right)=0,\]
pour tout $x\in\G$, où $\xi$ est le $1$-cocycle de la bigèbre de Lie $(\G,\G^*)$.
\end{enumerate}
\section{Algèbres de Milnor}
Une algèbre de Lie réelle de dimension finie $\G$, est appelée \emph{algèbre de Milnor} si elle est munie d'un produit scalaire $\prs$ (défini positif) telles que :
\begin{enumerate}
\item La sous-algèbre de Lie $S=\{x\in\G\mid \ad_x+\ad_x^t=0\}$ est commutative, où $\ad_x^t$ est l'application adjointe de $\ad_x$ par rapport à $\prs$.
\item L'idéal dérivé $[\G,\G]$ est abélien et est l'orthogonal de $S$ par rapport à $\prs$, \[S^\perp=[\G,\G].\]
\end{enumerate}
La terminologie se justifie par un résultat de Milnor qui, dans \cite{mil}, a montré qu'un groupe de Lie muni d'une métrique riemannienne invariante à gauche est plat, si et seulement si, son algèbre de Lie est une somme semi-directe d'une sous-algèbre commutative $\mathfrak{b}$ et d'un idéal abélien $\mathfrak{u}$ tel que, pour tout $b\in\mathfrak{b}$, $\ad_b$ est antisymétrique. 
\begin{proposition}\label{prmilnor}
Soit $(G,\prs)$ un groupe de Lie muni d'une métrique riemannienne invariante à gauche. La connexion de Levi-Civita de $G$ est plate, si et seulement si, son algèbre de Lie $\G$, munie du produit scalaire $\prs_e$, est de Milnor.
\end{proposition}
\begin{proof} C'est essentiellement la preuve de Milnor, dans \cite{mil}, du Théorème 1.5 pages 298 et 326-327, mais avec une amélioration de la décomposition de l'algèbre $\G$ en somme orthogonale de $S$ et de l'idéal dérivé $[\G,\G]$.\\
La connexion de Levi-Civita de $\prs$ est entièrement déterminée par l'application bilinéaire $A : \G\times\G\rightarrow\G$ donnée par 
\[2\langle A_xy,z\rangle_e=\langle[x,y],z\rangle_e+\langle[z,x],y\rangle_e+\langle[z,y],x\rangle_e.\]
Si $\G=S\oplus[\G,\G]$ est une algèbre de Milnor alors
\[A_x=\left\{\begin{array}{ll}\ad\nolimits_x,&\ \text{si}\ x\in S\\
0,&\ \text{si}\ x\in[\G,\G]\end{array}\right.\]
et donc la courbure s'annule, puisque $\ad_{[x,y]}=\ad_x\ad_y-\ad_y\ad_x$.\\
Réciproquement, supposons que la courbure est nulle. Si $\mathfrak{u}=\{u\in\G\mid A_u=0\}$, alors
$\mathfrak{u}$ est un idéal car, pour tout $x\in\G$, $y\in\mathfrak{u}$, $A_{[x,y]}=[A_x,A_y]=0$.
L'idéal $\mathfrak{u}$ est abélien, puisque pour tout $x,y\in\mathfrak{u}$, $[x,y]=A_xy-A_yx=0$. Soit maintenant $\mathfrak{b}=\mathfrak{u}^\perp$. Pour tout $x,y\in\mathfrak{b}$ et tout $z\in\mathfrak{u}$ on a
\[\langle[x,y],z\rangle=2\langle A_zy,x\rangle-\big[\langle[z,y],x\rangle+\langle[x,z],y\rangle\big]=0,\]
puisque $z,[x,z],[y,z]\in\mathfrak{u}$. On en déduit que $\mathfrak{b}$ est une sous algèbre.\\
Pour tout $x,y\in\mathfrak{b}$ et tout $z\in\mathfrak{u}$ on a
\[2\langle A_xy,z\rangle=\langle[x,y],z\rangle+\langle[z,y],x\rangle+\langle[z,x],y\rangle\big]=\langle[x,y],z\rangle=0,\] de sorte que $A_xy\in\mathfrak{b}$, c'est-à-dire la connexion infinitésimale $A$ se restreint à $\mathfrak{b}$ et est donc plate. Considérons l'algèbre de Lie $\mathfrak{b}$ et remarquons que le groupe de Lie connexe et simplement connexe $B$, qui lui est associé, est isométrique à $\mathbb{R}^p$, en tant que variété riemannienne complète à connexion plate (C'est le théorème de Cartan-Hadamard, dont une preuve est donnée dans \cite{lee} page 196). D'autre part, $\mathfrak{b}$ possède un produit scalaire bi-invariant défini par
\[b(x,y)=-\tr(A_xA_y).\]
En effet, si $x\in\mathfrak{b}$, alors $A_x : \mathfrak{b}\rightarrow\mathfrak{b}$ est antisymétrique et donc à valeurs propres complexes pures $\{i\lambda_1,...,i\lambda_p\}$ et est diagonalisable :
\[A_x=P\left(\begin{smallmatrix}i\lambda_1&&\\
&\ddots&\\&&i\lambda_p\end{smallmatrix}\right)P^{-1}.\]
Par suite $b(x,x)=\sum_{i=1}^p\lambda_i^2\geq0$ et $b(x,x)=0$, si et seulement si, $A_x=0$ (car $A_x$ serais antisymétrique et nilpotente). Ce qui signifie que $x\in\mathfrak{u}$ et donc $x=0$. Comme le produit scalaire $b$ est bi-invariant, l'orthogonal par rapport à $b$ de tout idéal est un idéal et donc $\mathfrak{b}$ se décompose en somme directe d'idéaux simples $\mathfrak{b}=\mathfrak{b}_1\oplus...\oplus\mathfrak{b}_k$. Si l'un des $\mathfrak{b}_i$ n'est pas commutatif, il serais compact d'après le Théorème de Myers (voir \cite{lee} page 201), puisque la courbure de Ricci $r(x)>0$ pour tout $x$ non nul.\footnote{Dans ce cas, le centre de $\mathfrak{b}_i$ est réduit à $0$ et la courbure sectionnelle est donnée par \[K(x,y)=\frac{1}{4}\|[x,y]\|^2.\]} Mais ceci serais en contradiction avec le fait que $B$ est isométrique à $\mathbb{R}^p$ ($\mathbb{R}^p$ aurait un sous-groupe compact non trivial). Donc tout les $\mathfrak{b}_i$ sont commutatifs et $\mathfrak{b}$ aussi. Comme $\ad_x$ est antisymétrique, pour tout $x\in\mathfrak{b}$, $\mathfrak{b}\subset S$ et $[\G,\G]=[\mathfrak{b},\mathfrak{u}]$ et on a, pour tout $y\in\mathfrak{b}$
\[\langle\ad\nolimits_xy,z\rangle+\langle y,\ad\nolimits_xz\rangle=2\langle A_yz,x\rangle-\langle[y,z],x\rangle=-\langle[y,z],x\rangle.\]
On en déduit que $S=[\G,\G]^\perp$ et la sous-algèbre $S$ est commutative car, pour tout $x,y\in S$,
\[[x,y]\in S\cap[\G,\G]=\{0\}.\]
\end{proof}
\begin{remarks}
\begin{enumerate}
\item Une algèbre de Milnor est toujours résoluble, puisque son idéal dérivé est commutatif.
\item Une algèbre de Milnor est unimodulaire, car $\tr\ad_x=0$ pour tout $x\in[\G,\G]$ et aussi pour tout $x\in S$ (un endomorphisme antisymétrique est de trace nulle).
\item Une algèbre de Minlor n'est nilpotente que si elle est abélienne. En effet, pour tout $x\in S$, $\ad_x=0$ (un endomorphisme antisymétrique et nilpotent est nul) et donc $\G=S$ et $\G$ est abélienne.
\end{enumerate}
\end{remarks}
Le Théorème suivant caractérise complètement les algèbres de Milnor.
\begin{theorem}
Si $\G=S\oplus[\G,\G]$ est une algèbre de Milnor non abélienne, alors son idéal dérivé $[\G,\G]$ est de dimension paire et il existe une base $\{e_1,...,e_p\}$ de $S$ et une base orthonormée $\{f_1,...,f_{2r}\}$ de $[\G,\G]$ telles que
\begin{align*}
[e_i,f_{2j-1}]=\delta_{ij}f_{2j},&\quad [e_i,f_{2j}]=-\delta_{ij}f_{2j-1}\ (i,j=1,...,\ell)\ \text{et}\\
[e_i,f_{2j-1}]=\lambda_{ij}f_{2j},&\quad [e_i,f_{2j}]=-\lambda_{ij}f_{2j-1}\ (i=1,...,\ell,\ j=\ell+1,...,r),
\end{align*}
avec $\ell\leq\min(p,r)$ et le centre de $\G$ est engendré par $\{e_{\ell+1},...,e_p\}$.
\end{theorem}
\begin{proof}
Soit $\{s_1,...,s_p\}$ une base de $S$. La restriction de $ad_{s_1}$ à $[\G,\G]$ est antisymétrique, donc de noyau $K_1$ de codimension paire dans $[\G,\G]$. Comme $\ad_{s_2}$ commute avec $\ad_{s_1}$, il laisse invariant $K_1$, et le même argument ci-dessus montre que $K_1\cap\ker\ad_{s_2}$ est de codimension paire dans $K_1$ et donc aussi dans $[\G,\G]$. On en déduit que
\[K_p=[\G,\G]\cap\left(\bigcap_{i=1}^p\ker\ad\nolimits_{s_i}\right)\]
est de codimension paire dans $[\G,\G]$. De sa définition, $K_p$ est contenu dans le centre de $\G$, et donc dans $S$ ; ce qui signifie que $K_p=\{0\}$ et $[\G,\G]$ est de dimension paire.\\
Il existe une famille de formes linéaires $\lambda_1,\ldots,\lambda_r$ de $S$, toutes non nulles, et une base orthonormée $(f_1,\ldots,f_{2r})$ de $[\G,\G]$ telles que, pour tout $j=1,\ldots,r$ et pour tout $s\in S$,
\begin{equation}\label{eq1}[s,f_{2j-1}]=\lambda_j(s)f_{2j}\quad\mbox{et}\quad
[s,f_{2j}]=-\lambda_j(s)f_{2j-1}.\end{equation}
Notons $\ell$ la dimension de $Vect\{\lambda_1,\ldots,\lambda_r\}$. Quitte à permuter les $\lambda_i$, on peut supposer que $(\lambda_1,\ldots,\lambda_{\ell})$ est une base de $Vect\{\lambda_1,\ldots,\lambda_r\}$. Pour tout $j=\ell+1,\ldots,r$, on a
\begin{equation}\label{eq2}\lambda_j=\sum_{k=1}^{\ell}a_{kj}\lambda_k.\end{equation}
Si $\sharp:\G^*\rightarrow\G$ est l'isomorphisme associé à $\prs$ et si $S_0=Vect\{\sharp(\lambda_1),\ldots,\sharp(\lambda_\ell)\}$, alors 
$S=S_0\oplus\mathcal{Z}(\G)$. En effet, si $x\in S$ alors
\begin{align*}
x\in S_0^\perp\Longleftrightarrow\langle x,\sharp(\lambda_j)\rangle=0,\ \forall j=1,...,r&\Longleftrightarrow\lambda_j(x)=0,\ \forall j=1,...,r\\
&\Longleftrightarrow[x,f_{2j-1}]=[x,f_{2j}]=0,\ \forall j=1,...,r\\
&\Longleftrightarrow x\in\mathcal{Z}(\G).
\end{align*}
D'autre part, il existe une base $(e_1,\ldots,e_{\ell})$ de $S_0$ telle que
pour tout $i,j=1,...,\ell$, $\lambda_i(e_j)=\delta_{ij}$. En effet, si $e_i=\sum_{j=1}^nb_{ji}\,\sharp(\lambda_j)$, alors les conditions $\lambda_i(e_j)=\delta_{ij}$ s'écrivent 
\[\sum_{j=1}^nb_{ji}\,\langle\lambda_i,\lambda_j\rangle^*=\delta_{ij},\]
et ce système admet une unique solution, puisque la matrice $\left(\langle\lambda_i,\lambda_j\rangle^*\right)$ est inversible, du fait que la famille $\{\lambda_1,...,\lambda_\ell\}$ est libre.
Ainsi, pour tout $i,j=1,\ldots,\ell$ et pour tout $k=\ell+1,\ldots,r$,
\begin{eqnarray}\label{eq3}[e_i,f_{2j-1}]&=&\delta_{ij}f_{2j},\;
[e_i,f_{2j}]=-\delta_{ij}f_{2j-1},\\
\;[e_i,f_{2k-1}]&=&a_{ik}f_{2k},\;
[e_i,f_{2k}]=-a_{ik}f_{2k-1}.\nonumber\end{eqnarray}
\end{proof}
\section{Groupes de Lie-Poisson plats et métaplats}
Le résultat suivant est fondamental pour toute la suite. C'est une version contravariante du résultat de Milnor, pour un groupe de Lie-Poisson à métrique riemannienne invariante à gauche plate (métrique sur le cotangent), où une formule simple de la métacourbure sera établie.
\begin{theorem}
Soit $(G,\pi,\prs)$ un groupe de Lie-Poisson riemannien. alors
\begin{enumerate}
\item $(\pi,\prs)$ est plat, si et seulement si, l'algèbre de Lie duale $(\G^*,\prs^*)$ est de Milnor.
\item Si $(\pi,\prs)$ est plat, et si on identifie $\G^*$ avec l'espace des $1$-formes invariantes à gauche, alors la métacourbure $\mathcal{M}$ est donnée par
\[\mathcal{M}(\alpha,\beta,\gamma)=\left\{\begin{array}{ll}\ad\nolimits_\alpha\ad\nolimits_\beta\rho(\gamma),&\quad\text{pour tout}\ \alpha,\beta,\gamma\in S,\\0,&\quad\text{si}\ \alpha,\beta\ \text{ou}\ \gamma\in[\G^*,\G^*],\end{array}\right.\] 
où $S=\{\alpha\in\G^*\mid \ad_\alpha+\ad_\alpha^t=0\}$ et $\rho : \G^*\rightarrow\G^*\wedge\G^*$ est le $1$-cocycle dual.
\end{enumerate}
\end{theorem}
\begin{proof}
Tout d'abord, le crochet de Koszul de deux $1$-formes invariantes à gauche est une $1$-forme invariante à gauche qui coïncide avec le crochet de Lie de $\G^*$ (en tant que bigèbre de Lie), si on l'identifie avec l'espace des $1$-formes invariantes à gauche. (Voir Proposition~\ref{kos.inv}, page~\pageref{kos.inv} et Théorème~\ref{theomultiplicatif}, page~\pageref{theomultiplicatif}).\\
\begin{enumerate}
\item On dénote par $\prs^*$ la métrique invariante à gauche sur $T^*G$ associée à $\prs$ et on dénote par $\mathcal{D}$ la connexion de Levi-Civita contravariante associée à $(\pi,\prs^*)$. Comme la métrique riemannienne est invariante à gauche, pour tout $\alpha,\beta,\gamma\in\G^*$, on a 
\[2\langle\mathcal{D}_\alpha\beta,\gamma\rangle^*=\langle[\alpha,\beta],\gamma\rangle^*+\langle[\gamma,\alpha],\beta\rangle^*+\langle[\gamma,\beta],\alpha\rangle^*.\]
Donc la restriction de $\mathcal D$ à $\G^*\times\G^*$ définit un produit sur $\G^*$. La nullité de la courbure de $\mathcal D$ est équivalente à la nullité de la courbure de ce produit. Soit $G^*$ le groupe de Lie connexe et simplement connexe associé à $\G^*$. La connexion $\mathcal{D}$ s'identifie à la connexion de Levi-Civita (covariante) associée à la métrique riemannienne sur $\G^*$ (considérée comme algèbre de Lie des champs de vecteurs invariants à gauche sur $G^*$) ; cette connexion est plate et donc, d'après la Proposition~\ref{prmilnor}, page~\pageref{prmilnor}, $\G^*$ est de Milnor.
\item Supposons maintenant que la connexion associée à $(\pi,\prs)$ est plate et donc, d'après la première partie, $\G^*=S\stackrel{\perp}\oplus [\G^*,\G^*]$ où $S=\{\alpha\in\G^*\mid \ad_\alpha+\ad_\alpha^*=0\}$ et $S$ et $[\G^*,\G^*]$ sont commutatifs. À noter que pour tout $\alpha\in\G^*$, $d\alpha\left(X,Y\right)=-\alpha\left([X,Y]\right)$ et donc $\rho=-d$. 

De la formule de Koszul, pour la connexion $\mathcal{D}$, on déduit que, pour tout $\gamma\in\G^*$,
\begin{equation}\label{levi-civita}
\mathcal{D}_\alpha\gamma=\left\{\begin{array}{ll}
0 & \text{si}\ \alpha\in [\G^*,\G^*] \\
\;[\alpha,\gamma]=\ad\nolimits_\alpha\gamma & \text{si}\ \alpha\in S,
\end{array}\right.
\end{equation}
et, pour tout $\alpha\in S$, $\mathcal{D}\alpha=0$.
\begin{enumerate}\item Soient $\alpha,\beta,\gamma\in S$. Comme $\mathcal{D}\alpha=\mathcal{D}\beta=\mathcal{D}\gamma=0$, on déduit de \eqref{2parallel} que
$$\mathcal{M}(\alpha,\beta,\gamma)=-\mathcal{D}_\alpha\mathcal{D}_\beta d\gamma\stackrel{\eqref{levi-civita}}=ad_\alpha ad_\beta\rho(\gamma).$$
\item Soient $\alpha,\gamma\in S$ et soit $\beta\in [\G^*,\G^*]$, comme $\mathcal{D}\alpha=\mathcal{D}\gamma=0$, on déduit de \eqref{2parallel} que
$$\mathcal{M}(\alpha,\beta,\gamma)=-\mathcal{D}_\beta\mathcal{D}_\alpha d\gamma\stackrel{\eqref{levi-civita}}=0.$$
\item Soient $\alpha,\beta\in [\G^*,\G^*]$ et soit $\gamma\in S$. On a localement $\alpha=\sum f_idg_i$ et on déduit de \eqref{metacourbure} que
\begin{equation*}
\mathcal{M}(\alpha,\beta,\gamma)=\sum f_i\{g_i,\{\beta,\gamma\}\}-f_i\{\{g_i,\beta\},\gamma\}-f_i\{\{g_i,\gamma\},\beta\}.
\end{equation*}
De \eqref{2parallel} et \eqref{levi-civita}, on a $\{\beta,\gamma\}=-\mathcal{D}_\beta d\gamma=0$.
D'autre part, en utilisant \eqref{levi-civita}, $\{g_i,\gamma\}=\mathcal{D}_{dg_i}\gamma=0$ et donc
\begin{equation*}
\mathcal{M}(\alpha,\beta,\gamma)=\sum-f_i\{\{g_i,\beta\},\gamma\}=\sum f_i\mathcal{D}_{\mathcal{D}_{dg_i}\beta}d\gamma=
\mathcal{D}_{\mathcal{D}_\alpha\beta}d\gamma=0.
\end{equation*}
\item Pour $\alpha,\beta\in [\G^*,\G^*]$, le calcul de $\mathcal{M}(\alpha,\beta,\beta)$ est plus difficile. Tout d'abord, en comparant $\mathcal{M}(\alpha,\beta,\beta)$ et $[\beta,[\beta,d\alpha]]$, on montre qu'ils sont égaux à un signe près, puis, on montre que $[\beta,[\beta,d\alpha]]=0$ et on aura obtenu le résultat souhaité.

Soit $\alpha=\sum f_idg_i$. En utilisant \eqref{metacourbure}, on a
\begin{align*}
\mathcal{M}(\alpha,\beta,\beta)=&\sum f_i\{g_i,\{\beta,\beta\}\}-2f_i\{\{g_i,\beta\},\beta\}\\
=&\sum f_i\mathcal{D}_{dg_i}\{\beta,\beta\}-2\sum f_i\{\mathcal{D}_{dg_i}\beta,\beta\}\\
=&\mathcal{D}_\alpha\{\beta,\beta\}-2\sum f_i\{\mathcal{D}_{dg_i}\beta,\beta\}\\
\stackrel{(*)}=&-2\sum f_i\{\mathcal{D}_{dg_i}\beta,\beta\}\\
=&-2\sum \left(\{f_i\mathcal{D}_{dg_i}\beta,\beta\}+\mathcal{D}_{df_i}\beta\wedge\mathcal{D}_{dg_i}\beta\right)\\
=&-2\{\mathcal{D}_{\alpha}\beta,\beta\}-2\sum\mathcal{D}_{df_i}\beta\wedge\mathcal{D}_{dg_i}\beta\\
=&-2\sum\mathcal{D}_{df_i}\beta\wedge\mathcal{D}_{dg_i}\beta.
\end{align*}
Dans $(*)$ on a utilisé \eqref{levi-civita} et le fait que $\{\alpha,\beta\}\in\wedge^2\G^*$, qui peut être déduit de \eqref{crochet1formes}.
D'autre part,
\begin{align*}
\;[\beta,[\beta,d\alpha]]=&\sum[\beta,[\beta,df_i\wedge dg_i]]\\
=&\sum[\beta,[\beta,df_i]\wedge dg_i]+[\beta,df_i\wedge[\beta,dg_i]]\\
=&\sum[\beta,[\beta,df_i]]\wedge dg_i+[\beta,df_i]\wedge[\beta,dg_i]\\
&+[\beta,df_i]\wedge[\beta,dg_i]+df_i\wedge[\beta,[\beta,dg_i]].
\end{align*}
On choisit une base orthonormée $\{\alpha_1,...,\alpha_n\}$ de $\G^*$. pour toute $1$-forme $\gamma\in \Omega^1(G)$, on a
$\gamma=\sum\langle \gamma,\alpha_i\rangle^*\alpha_i$, et
\begin{align*}
[\beta,\gamma]=&\sum\pi_\sharp(\beta)\cdot\langle \gamma,\alpha_i\rangle^*\,\alpha_i+
\langle \gamma,\alpha_i\rangle^*[\beta,\alpha_i]\\
=&\mathcal{D}_\beta \gamma+\sum\langle \gamma,\alpha_i\rangle^*[\beta,\alpha_i].
\end{align*}
Donc
\begin{align*}
\;[\beta,[\beta,\gamma]]=&[\beta,\mathcal{D}_\beta \gamma]+\sum\pi_\sharp(\beta)\cdot\langle \gamma,
\alpha_i\rangle^*[\beta,\alpha_i]+\sum\langle \gamma,
\alpha_i\rangle^*[\beta,[\beta,\alpha_i]]\\
\stackrel{(*)}=&[\beta,\mathcal{D}_\beta \gamma]-\sum\langle\mathcal{D}_\beta \gamma,
\alpha_i\rangle^*\mathcal{D}_{\alpha_i}\beta\\
=&[\beta,\mathcal{D}_\beta \gamma]-\mathcal{D}_{\mathcal{D}_\beta\gamma}\beta\\
=&\mathcal{D}_\beta\mathcal{D}_\beta \gamma-2\mathcal{D}_{(\mathcal{D}_\beta \gamma)}\beta\\
=&\mathcal{D}_\beta\mathcal{D}_\beta \gamma-2\mathcal{D}_{[\beta,\gamma]}\beta+2\sum\langle \gamma,\alpha_i\rangle^*\mathcal{D}_{[\beta,\alpha_i]}\beta\\
\stackrel{(**)}=&\mathcal{D}_\beta\mathcal{D}_\beta \gamma-2\mathcal{D}_{[\beta,\gamma]}\beta\\
=&\mathcal{D}_\beta\mathcal{D}_\beta \gamma-2(K(\beta,\gamma)\beta+\mathcal{D}_\beta\mathcal{D}_{\gamma}\beta-\mathcal{D}_{\gamma}\mathcal{D}_\beta\beta)\\
=&\mathcal{D}_\beta\mathcal{D}_\beta \gamma-2\mathcal{D}_\beta\mathcal{D}_{\gamma}\beta.
\end{align*}
On a utilisé, dans les égalités $(*)$ et $(**)$, le fait que $[\G^*,\G^*]$ est abélien et donc $[\beta,[\beta,\alpha_i]]=-\mathcal{D}_{[\beta,\alpha_i]}\beta=0$.

En utilisant cette formule, on a
\begin{align*}
[\beta,[\beta,df_i]]\wedge dg_i=&\mathcal{D}_\beta\mathcal{D}_\beta df_i\wedge dg_i-2\mathcal{D}_\beta\mathcal{D}_{df_i}\beta\wedge dg_i\\
=&\mathcal{D}_\beta\left(\mathcal{D}_\beta df_i\wedge dg_i\right)-\mathcal{D}_\beta df_i\wedge\mathcal{D}_\beta dg_i\\ &-2\mathcal{D}_\beta\left(D_{df_i}\beta\wedge dg_i\right)+2D_{df_i}\beta\wedge\mathcal{D}_\beta dg_i,\\
df_i\wedge[\beta,[\beta,dg_i]]=&-\mathcal{D}_\beta\left(\mathcal{D}_\beta dg_i\wedge df_i\right)+\mathcal{D}_\beta dg_i\wedge\mathcal{D}_\beta df_i\\ &+2\mathcal{D}_\beta\left(D_{dg_i}\beta\wedge df_i\right)-2D_{dg_i}\beta\wedge\mathcal{D}_\beta df_i.\\
\end{align*} 
D'autre part,
\begin{align*}
2[\beta,df_i]\wedge[\beta,dg_i]=&2\mathcal{D}_\beta df_i\wedge\mathcal{D}_\beta dg_i-2\mathcal{D}_\beta df_i\wedge\mathcal{D}_{dg_i}\beta\\&-2\mathcal{D}_{df_i}\beta\wedge\mathcal{D}_\beta dg_i+2\mathcal{D}_{df_i}\beta\wedge\mathcal{D}_{dg_i}\beta.
\end{align*}
Donc
\begin{align*}
\;[\beta,[\beta,d\alpha]]=&\mathcal{D}_\beta\mathcal{D}_\beta d\alpha+2\sum\mathcal{D}_{df_i}\beta\wedge\mathcal{D}_{dg_i}\beta-
2\mathcal{D}_\beta\left(\mathcal{D}_{df_i}\beta\wedge dg_i\right)\\&-2\mathcal{D}_\beta\left(df_i\wedge\mathcal{D}_{dg_i}\beta\right)\\
=&\mathcal{D}_\beta\mathcal{D}_\beta d\alpha+2\sum\mathcal{D}_{df_i}\beta\wedge\mathcal{D}_{dg_i}\beta+
2\mathcal{D}_\beta\left([\beta,df_i]\wedge dg_i\right)\\&-2\mathcal{D}_\beta\left(\mathcal{D}_\beta df_i\wedge dg_i\right)
+2\mathcal{D}_\beta\left(df_i\wedge[\beta,dg_i]\right)-
2\mathcal{D}_\beta\left(df_i\wedge\mathcal{D}_\beta dg_i\right)\\
=&-\mathcal{D}_\beta\mathcal{D}_\beta d\alpha-\mathcal{M}(\alpha,\beta,\beta)+2\mathcal{D}_\beta[\beta,d\alpha]\\
=&-\mathcal{M}(\alpha,\beta,\beta).
\end{align*}
Montrons maintenant que $[\beta,[\beta,d\alpha]]=0$. On a $\alpha=\sum[\gamma_i,\gamma_j]$ et comme $d$ est un $1$-cocycle par rapport à l'action adjointe de $\G^*$, 
$$d\alpha=\sum\left([d\gamma_i,\gamma_j]+[\gamma_i,d\gamma_j]\right).$$
Cette relation implique que $d\alpha=\alpha_1+\alpha_2$, où $\alpha_1\in S\wedge[\G^*,\G^*]$ et $\alpha_2\in\wedge^2[\G^*,\G^*]$. De cette décomposition et du fait que $[\G^*,\G^*]$ est abélien, on peut déduire que $[\beta,[\beta,d\alpha]]=0$ et donc $\mathcal{M}(\alpha,\beta,\beta)=0$. Comme $\mathcal{M}$ est symétrique, on déduit que, pour tout $\alpha,\beta,\gamma\in[\G^*,\G^*]$, $\mathcal{M}(\alpha,\beta,\gamma)=0$.
\end{enumerate}
\end{enumerate}
\end{proof}
\section{Classe modulaire des groupes de Lie-Poisson}
Dans cette section on va considérer la question d'unimodularité des structures de Poisson. Dans le cas d'un groupe de Lie-Poisson, une condition nécessaire d'unimodularité est donnée par
\begin{proposition}\label{unimod.néc}
Soit $(G,\pi)$ un groupe de Lie-Poisson, muni d'une forme volume $\mu$. Si $(G,\pi)$ est unimodulaire alors
\begin{equation}\label{unimod.néc1}
\rho\left(i_{\xi(x)}\mu_e\right)=0,\ \text{pour tout}\ x\in\G.
\end{equation}
où $\xi$ est le $1$-cocycle de la bigèbre de Lie $(\G,\G^*)$ et $\rho$ est le $1$-cocycle dual.
\end{proposition}
\begin{proof}
Rappelons que $\rho : \G^*\rightarrow\wedge^2\G^*$ est égal à $-d$ (la restriction de la différentielle extérieure aux $1$-formes invariantes à gauche) et s'étend, comme la différentielle $d$, aux formes invariantes à gauche de tout degrés.\\
La formule de Koszul \cite{kos}, satisfaite pour tout champ de vecteurs $X$ et pour tout champ de multivecteurs $Q$, est donnée par :
\begin{equation}\label{koszul}
i_{[X,Q]}\mu=i_Xdi_Q\mu+(-1)^{\deg Q}di_Xi_Q\mu-i_Qdi_X\mu.
\end{equation}
On a, pour $Q=\pi$ et $X$ un champ de vecteurs invariant à gauche,
\begin{align*}
i_{[X,\pi]}\mu=i_Xdi_\pi\mu+di_Xi_\pi\mu-i_\pi di_X\mu.
\end{align*}
\end{proof}
Dans le cas d'un groupe de Lie unimodulaire on a le résultat suivant :
\begin{theorem}(\cite{bah-bou})
Soient $\pi$ et $\langle\,,\,\rangle$, respectivement, un tenseur de Poisson multiplicatif et une métrique riemannienne invariante à gauche sur un groupe de Lie $G$ unimodulaire, et soit $\mu$ la forme volume associée à $\langle\,,\,\rangle$. Alors $d\left(i_\pi\mu\right)=0$, si et seulement si, $\G^*$ est une algèbre de Lie unimodulaire et si, pour tout $x\in\G$, $\delta\left(i_{\xi(x)}\mu\right)=0$, où $\delta$ est la restriction de la différentielle $d$ aux formes invariantes à gauche.
\end{theorem}
\begin{proof}
Soit le champ modulaire $X_\mu$ donné par $d\left(i_\pi\mu\right)=i_{X_\mu}\mu$, et soit la forme modulaire $\kappa : \G^*\to\mathbb{R}$ donnée par $\kappa(\alpha)=\tr\ad_\alpha$, où $\ad_\alpha\beta=[\alpha,\beta]^*$ (l'action adjointe de $(\G^*,[\,,\,]^*)$). La forme modulaire définit un vecteur dans $\G$, qu'on continu de noter $\kappa$, et on note $\kappa^+$ le champ invariant à gauche associée à $\kappa$. Montrons que
\begin{equation}\label{kappa}
X_\mu(e)=\kappa.
\end{equation}
Si $X_1,...,X_n$ et $\alpha,...,\alpha_n$ sont des bases orthonormées respectivement de $(\G,\langle\,,\,\rangle_e)$ et de $(\G^*,\langle\,,\,\rangle_e^*)$, on notera $X_i^+$, respectivement $\alpha_i^+$, le champ invariant à gauche associé  $X_i$, respectivement la forme invariante à gauche associée à $\alpha_i$. On a alors $\mu=\alpha_1^+\wedge...\wedge\alpha_n^+$, $\pi=\sum_{i<j}\pi_{ij}\,X_i^+\wedge X_j^+$ et le champ hamiltonien est donné par $X_f=\sum_{j=1}^n\left(\sum_{i=1}^n\pi_{ij}<df,X_i>\right)X_j^+$. On a
\begin{align*}
X_\mu(f)=&\Div\nolimits_{\alpha_1^+\wedge...\wedge\alpha_n^+}\sum_{j=1}^n\left(\sum_{i=1}^n\pi_{ij}<df,X_i>\right)X_j^+\\
=&\sum_{j=1}^n\left(\sum_{i=1}^n\pi_{ij}X_i(f)\right)\Div\nolimits_{\alpha_1^+\wedge...\wedge\alpha_n^+}X_j^++\sum_{j=1}^n
\sum_{i=1}^nX_j^+\left(\pi_{ij}<df,X_i>\right).
\end{align*}
Comme, pour tout $i,j=1,...,n$, 
\[\Div\nolimits_{\alpha_1^+\wedge...\wedge\alpha_n^+}X_j^+=0,\ \text{et}\ \pi_{ij}(e)=0,\]
on a
\[X_\mu(e)=\sum_{i=1}^n\left(\sum_{j=1}^n\mathscr{L}_{X_j}\pi_{ij}(e)\right)X_i,\]
en plus
\begin{align*}
<\alpha_i,X_\mu(e)>=&\sum_{j=1}^n\mathscr{L}_{X_j^+}\pi_{ij}(e)=\sum_{j=1}^n<\alpha_i\wedge\alpha_j,\xi(X_j)>\\
=&\sum_{j=1}^n[\alpha_i,\alpha_j]^*(X_j)=\sum_{j=1}^n\left\langle\sum_{k=1}^n[\alpha_i,\alpha_j]^*(X_k)\alpha_k,\alpha_j\right\rangle^*\\
=&\sum_{j=1}^n\langle[\alpha_i,\alpha_j]^*,\alpha_j\rangle^*=\tr\ad\nolimits_{\alpha_i}=\kappa(\alpha_i).
\end{align*}
Comme ceci est vrai pour tout $i=1,...,n$, on en déduit que $X_\mu(e)=\kappa$.\\
La formule de Koszul \eqref{koszul}
\begin{equation}
i_{[X,Q]}\mu=i_Xdi_Q\mu+(-1)^{\deg Q}di_Xi_Q\mu-i_Qdi_X\mu.
\end{equation}
Soit $X$ un champ de vecteurs invariant  gauche. Comme $G$ est unimodulaire,  
\[\mathscr{L}_X\mu=0\ \text{et}\ di_X\mu=\mathscr{L}_X\mu-i_Xd\mu=0.\]
En appliquant la formule de Koszul à $[X,X_\mu]$ et  $[X,\pi]$ on a
\begin{align*}
i_{[X,X_\mu]}\mu=&i_Xdi_{X_\mu}\mu-di_Xi_{X_\mu}\mu-i_{X_\mu}di_X\mu\\
=&-di_Xi_{X_\mu}\mu,
\end{align*}
et d'autre part
\begin{align*}
i_{[X,\pi]}\mu=&i_Xdi_\pi\mu-di_Xi_\pi\mu-i_\pi di_X\mu\\
=&i_Xi_{X_\mu}\mu-di_Xi_\pi\mu.
\end{align*}
Donc
\begin{equation}\label{formula}
d\left(i_{[X,\pi]}\mu\right)=-i_{[X,X_\mu]}\mu.
\end{equation}
Comme $[X,\pi]$ et $\mu$ sont invariants à gauche, on déduit de \eqref{formula} que $[X,X_\mu]$ est invariant à gauche. Si $X_m=X_\mu-\kappa^+$ alors $[X,X_m]=[X,X_\mu]+[X,\kappa^+]$ est invariant à gauche et $X_m(e)=0$, c'est-à-dire $X_m$ est un champ de vecteurs multiplicatif et $X_\mu=0$, si seulement si, $\kappa=0$ et $X_m=0$, c'est-à-dire $(\G^*,[\,,\,]^*)$ est unimodulaire. Mais $X_m=0$, si et seulement si, $[X,X_m](e)=0$, pour tout champ invariant à gauche $X$, puisque $X_m$ est multiplicatif et $G$ est connexe. Dans ce cas la formule \eqref{formula} en $e$ donne
\[\delta\left(i_{\xi(X)}\mu\right)=0.\]
\end{proof}
\chapter{Groupes de Lie-Poisson riemanniens déformables}\label{chapter4}
\epigraph{C'est avec la logique que nous prouvons et avec l'intuition que nous trouvons.}{Henri Poincaré.}
\lettrine[lines=3, nindent=0em]{D}{ans} ce chapitre on déterminera, en petites dimensions (jusqu'à la dimension cinq) et à isomorphisme près, tous les groupes de Lie-Poisson à métriques riemanniennes invariantes à gauche vérifiant les trois conditions de déformations de Hawkins.\\
Soit $(G,\pi)$ un groupe de Lie-Poisson, $(\G,[\,,\,])$ son algèbre de Lie et
$\xi : \G\rightarrow\G\times\G$ le $1$-cocycle associé. On notera $(\G^*,[\,,\,]^*)$ l'algèbre de Lie duale et $G^*$ le groupe de Lie connexe et simplement connexe associé. Soit l'application
linéaire $\rho: \G^*\rightarrow\G^*\times\G^*$ le dual du crochet de Lie de $\G$. Il est connu que $\rho$ est un $1$-cocycle pour la représentation adjointe de $(\G^*,[\,,\,]^*)$ (Voir \cite{duf-zun}), et définit un tenseur de Lie Poisson $\pi^*$ sur $G^*$. D'un autre côté, tout produit scalaire $\prs$, sur $\G^*$ (ou $\G$ ) muni à la fois $G$ et $G^*$ d'une métrique invariante à gauche.

On commencera par déterminer les bigèbres de Lie $(\G^*,\rho)$ telles que :
\begin{enumerate}
\item $\G^*=S\oplus[\G^*,\G^*]$ est de Milnor,
\item $\ad_\alpha\ad_\beta\rho(\gamma)=0$, pour tout $\alpha,\beta,\gamma\in S$,
\item la condition $d\left(i_{\xi(x)}\mu\right)=0$ est satisfaite pour tout $x\in\G$, dans le cas où $\G$ est unimodulaire. ($\mu$ est une forme volume et $\xi$ est le cocycle correspondant de la bigèbre de Lie duale.)
\end{enumerate}
Puis, on déterminera les groupes de Lie connexes et simplement connexes\footnote{Parfois on se contentera de trouver un goupe de Lie $G$ dont l'algèbre de Lie est $\G$.} correspondants aux algèbres de Lie trouvées $\G$, et les tenseurs de Poisson multiplicatifs en intégrant les cocycles $\xi$ trouvés.\\
Tout d'abord, on examinera quelques structures de Riemann-Poisson qui sont déformables.
\section{Quelques familles génériques}
\subsection{Structures de Poisson linéaires}
Soit $\G=S\oplus[\G,\G]$ une algèbre de Milnor. Comme $S$ est une sous-algèbre commutative qui agit sur l'idéal dérivé $[\G,\G]$ par endomorphismes anti-symétriques, il existe une famille de vecteurs non nuls $u_1,...,u_r$ de $S$ et une base orthonormale $\{f_1,...,f_{2r}\}$ de $[\G,\G]$ telle que, pour tout $j=1,...,r$ et tout $s\in S$
\begin{equation}
[s,f_{2j-1}]=<s,u_j>f_{2j},\quad [s,f_{2j}]=-<s,u_j>f_{2j-1}
\end{equation}
D'après le Corollaire, le triplet $(\G^*,\pi_\ell,\langle\,,\,\rangle)$ satisfait les conditions de Hawkins. Il existe une famille de constantes $(a_{ij})_{1\leq i,j\leq q}$ telles que $(\G^*,\pi_\ell,\langle\,,\,\rangle)$ est isomorphe à $(\mathbb{R}^{q+2r},\pi_0,\langle\,,\,\rangle_0)$, où $\langle\,,\,\rangle_0$ est la métrique euclidienne et 
\[\pi_0=\sum_{i=1}^r(a_{1i}\partial_{x_1}+\ldots+a_{qi}\partial_{x_q})\wedge(y_{2i}\partial_{y_{2i-1}}-y_{2i-1}\partial_{y_{2i}}).\]
\subsection{Structures triangulaires d'un groupe compact semi-simple}
\begin{proof} 
Soit $(G,\pi,\prs)$ un groupe de Lie-Poisson riemannien tel que $G$ est compact semi-simple, $\prs$ est bi-invariante est $\pi$ est exact, c'est-à-dire il existe $r\in\wedge^2\G$ tel que $\pi=r^--r^+$, où $r^+$ (resp. $r^-$) est le champ de bivecteurs invariant à gauche (resp. à droite) associé à $r$. Il est connu que $[r,r]$ est $ad$-invariant et la structure d'algèbre de Lie duale sur $\G^*$ est donnée par
$$[\al,\be]^*=ad_{r_{\#}(\al)}^*\be-ad_{r_{\#}(\be)}^*\al,$$
où $r_{\#}:\G^*\too\G$ est la contraction associée à $r$. Comme $\prs$ est bi-invariante, la connexion de Levi-Civita contravariante $\mathcal{D}$ associée à $(\pi,\prs)$ est donnée par
\begin{equation}\label{exact1}\mathcal{D}_\al\be=ad_{r_{\#}(\al)}^*\be,\quad\al,\be\in\G^*, \end{equation} et sa courbure est donnée par (see \cite{haw1})
\begin{equation}K(\al,\be)\ga=ad_{[r,r](\al,\be,.)}^*\ga.\end{equation}
Comme $G$ est semi-simple, $K$ s'annulle, si et seulement si, $[r,r]=0$. Supposons que $[r,r]=0$ et montrons que $(\pi,\prs)$ est métaplat et vérifie la troisième condition de Hawkins. \\ Comme $(\pi,\prs)$ est plat, en vertu du Théorème \ref{main1}, $\G^*=S\oplus[\G^*,\G^*]$, où la sous-algèbre $S=\{\al\in\G^*,\; ad_{\al}+ad_{\al}^t=0\}$ est abélienne et l'idéal dérive $[\G^*,\G^*]$ est abélien. En utilisant la preuve du Lemme \ref{killing2} et (\ref{exact1}), il est facile de montrer que \begin{equation}\label{exact2}S=\{\al\in\G^*,ad_{r_{\#}(\be)}^*\al=0\;\mbox{for all}\;\be\in\G^*\}.\end{equation} D'autre part, on peut vérifier que $\ker r_{\#}\subset S$. De plus, $\G=\mbox{Im}r_{\#}\oplus\mbox{Im}r_{\#}^\perp$ and $\mbox{Im}r_{\#}$ est unimodulaire et symplectique et donc résoluble (voir\cite{li-med}).\\
Comme $\mbox{Im}r_{\#}$ admet un produit scalaire  bi-invariant, il est nécessairement abélien (voir \cite{mi}).Montrons que (\ref{flat}) est vérifiée. Soit $\B_1=\{e_1,\ldots,e_{2p}\}$ une base orthonormale de $\mbox{Im}r_{\#}$, soit $\B_2=\{f_1,\ldots,f_{n-2p}\}$ une base orthonormale de $\mbox{Im}r_{\#}^\perp$ et soit $\{\al_1,\ldots,\al_{2p},\be_1,\ldots,\be_{n-2p}\}$ la base duale de $\B_1\cup\B_2$. Soient $\al,\ga\in S$. Pour tout $r_{\#}(\mu_1),r_{\#}(\mu_2)\in\mbox{Im}r_{\#}$ et tout
 $u\in\mbox{Im}r_{\#}^\perp$, on a 
 \begin{eqnarray*}
 d\ga(r_{\#}(\mu_1),r_{\#}(\mu_2))&=&-\ga([r_{\#}(\mu_1),r_{\#}(\mu_2)])\\
 &=&0,\\
d\ga(r_{\#}(\mu_1),u)&=&-ad_{r_{\#}(\mu_1)}^*\ga(u)\\
&\stackrel{(\ref{exact2})}=&0,\end{eqnarray*} et donc
\begin{equation*}
d\ga=\sum_{i,j}a_{i,j}\be_i\wedge\be_j,\end{equation*} où $a_{i,j}\in\R$. Ona 
$$ad_{\al}d\ga=\sum_{i,j}a_{i,j}\left(ad_{\al}\be_i\wedge\be_j+\be_i\wedge ad_{\al}\be_j\right).$$
Or
\begin{eqnarray*}
ad_{\al}\be_i&=&[\al,\be_i]^*\\
&=&ad_{r_{\#}(\al)}^*\be_i-ad_{r_{\#}(\be_i)}^*\al\\
&\stackrel{(\ref{exact2})}=&ad_{r_{\#}(\al)}^*\be_i,\end{eqnarray*} et $\be_i$, l'annulateur de $\mbox{Im}r_{\#}$, est égal à $\ker r_{\#}$. On a montré que
$\ker r_{\#}\subset S$ et, suivant (\ref{exact2}), $ad_{r_{\#}(\al)}^*\be_i=0$, alors
$ad_{\al}d\ga=0$ et (\ref{flat}) est vérifiée. Pour conclure, on va montrer que (\ref{unimodular}) est vérifiée et on obtient le résultat, suivant le Théorème \ref{main2} (puisque $\G$ est unimodulaire). À noter que dans notre cas $G$ is $\xi(u)=[u,r]$ et, en utilisant (\ref{koszul.formula}), on obtient
\begin{equation}
i_{[u,r]}\mu=i_udi_r\mu+di_ui_r\mu-i_rdi_u\mu.
\end{equation} Mais comme $G$ est unimodulaire, $di_u\mu=0$. D'autre part, $r=\sum_{i,j}b_{ij}e_i\wedge e_j$ et donc
\begin{eqnarray*}
d(i_{r}\mu)&=&d\left(\sum_{i,j}b_{ij}i_{e_i\wedge e_j}\mu\right)\\
&=&\sum_{i,j}b_{ij}\left(i_{[e_i,e_j]}\mu-
i_{e_i}\Li_{e_j}\mu-i_{e_j)}\Li_{e_i}\mu\right)\\
&=&0,\end{eqnarray*} ce qui achève la démonstration.
\end{proof}
\subsection{Groupes de Heisenberg}
On reprend ici le cas des groupes de Heisenberg de toutes dimensions, étudié dans \cite{bah-bou:heisenberg}.
Soit $H_n$ le groupe de Heisenberg de dimension $2n+1$, c'est-à-dire le groupe des matrices
\[\begin{pmatrix}1 & X^t & C\\
0 & \id_n & Y\\
0 & 0 &1\end{pmatrix},\]
où $c\in\mathbb{R}$ et $X,Y\in\mathbb{R}^n$. Soit $\mathcal{H}_n$ son algèbre de Lie et $z$ un élément non-nul du centre de $\mathcal{H}_n$. Il existe une $2$-forme $\omega$ sur $\mathcal{H}_n$, telle que $i_z\omega=0$. La projection de $\omega$ sur $\mathcal{H}_n/(\mathbb{R}z)$ est non-dégénérée et, pour
tous $u,v\in\mathcal{H}_n,\ [u,v]=\omega(u,v)z$. On a alors
\begin{theorem}Soient $\pi$ et $\prs$, respectivement, un tenseur de Poisson multiplicatif et une métrique riemannienne invariante à gauche sur $H_n$. Alors $\pi$ et $\prs$ sont compatibles dans le sens de Hawkins, si et seulement si :
\begin{enumerate}
\item il existe un endomorphisme $J:\mathcal{H}_n\too\mathcal{H}_n$ antisymétrique par rapport à $\prs_e$ tel que $J(z)=0$ et, pour tout $u\in\mathcal{H}_n$, $\xi(u)=z\wedge Ju,$
\item pour tout $u,v\in\mathcal{H}_n$, $\omega(J^2u,v)+\omega(u,J^2v)+2\omega(Ju,Jv)=0$.
\end{enumerate}
\end{theorem}
\begin{proof}
\begin{description}
\item [La nullité de la courbure.] Soit $H_n^*$ le groupe de Lie connexe et simplement connexe associé à l'algèbre de Lie $(\mathcal{H}_n^*,[\,,\,]^*)$ et soit $g$ une métrique riemannienne invariante à gauche sur $H_n^*$ de valeur $\prs_e$ en l'élément neutre. La nullité de la courbure de $g$ est équivalente à la nullité de la courbure de la connexion de Levi-Civita contravariante $\mathcal{D}$, associée à $(\pi,\prs)$. Si $\mathcal{D}$ est plate, alors d'après le Théorème $\mathcal{H}_n^*=S\oplus[\mathcal{H}_n^*,\mathcal{H}_n^*]$. De plus, il existent $\lambda_1,...,\lambda_r\in S\setminus\{0\}$ et une base orthonormale $(\alpha_1,\beta_1,...,\alpha_r,\beta_r)$ tels que, pour tout $s\in S$ et tout $i=1,...,r$,
\[[s,\alpha_i]^*=\langle s,\lambda_i\rangle^*\beta_i\quad\text{et}\quad[s,\beta_i]^*=-\langle s,\lambda_i\rangle^*\alpha_i\]
Soit $\sharp : \mathcal{H}_n^*\rightarrow\mathcal{H}_n$ l'isomorphisme induit par la métrique et soient $v_i=\sharp(\lambda_i)$, $e_j=\sharp(\alpha_j)$ et $f_j=\sharp(\beta_j)$. On a, pour tout $i=1,...,r$, $\xi(e_i)=-v_i\wedge f_i$, $\xi(f_i)=v_i\wedge e_i$ et $\sharp(S)=\ker\xi$. En effet
\begin{align*}
\alpha\in S\Leftrightarrow&\langle\alpha,[\beta,\gamma]^*\rangle^*=0, \forall\beta,\gamma\in\mathcal{H}_n^*\\
\Leftrightarrow&\langle\sharp(\alpha),\sharp\left([\beta,\gamma]^*\right)\rangle=0, \forall\beta,\gamma\in\mathcal{H}_n^*\\\Leftrightarrow&
[\beta,\gamma]^*\left(\sharp(\alpha)\right)=0, \forall\beta,\gamma\in\mathcal{H}_n^*\\
\Leftrightarrow&\langle\xi\left(\sharp(\alpha)\right),\beta\wedge\gamma\rangle=0, \forall\beta,\gamma\in\mathcal{H}_n^*\\
\Leftrightarrow&\sharp(\alpha)\in\ker\xi.
\end{align*}
Il reste à montrer que $\xi(z)=0$ et que $v_1,...,v_r$ sont dans le centre de $\mathcal{H}_n$. Si $z\not\in\ker\xi$ alors $\ker\xi$ est une sous algèbre abélienne, car pour tout $x,y\in\ker\xi$
\[0=\xi([x,y])=\xi\left(\omega(x,y)z\right)=\omega(x,y)\xi(z),\] 
c'est-à-dire $\omega(x,y)=0$ et donc $[x,y]=\omega(x,y)z=0$.
Il existent $u_0\in\ker\xi$ et $v_0\in\ker\xi^\perp$ tels que $[u_0,v_0]=z$ et on a \(\xi(z)=\xi([u_0,v_0])=\ad\nolimits_{u_0}\xi(v_0)=X\wedge z\), où 
\[X=\sum_{i=1}^r\left(-\langle v_0,e_i\rangle\omega(u_0,f_i)+\langle v_0,f_i\rangle\omega(u_0,e_i)\right)v_i\in\ker\xi\setminus\{0\}.\]
On a, pour tout $u\in\ker\xi$ et tout $i=1,...,r$,
\begin{align*}
\omega(u,e_i)X\wedge z=\xi\left([u,e_i]\right)=\ad\nolimits_u\xi(e_i)=&-\omega(u,f_i)v_i\wedge z\ \text{et}\\
\omega(u,f_i)X\wedge z=\xi\left([u,f_i]\right)=\ad\nolimits_u\xi(f_i)=&\phantom{-}\omega(u,e_i)v_i\wedge z.
\end{align*}
Comme $z\not\in\ker\xi$, $\omega(u,e_i)X+\omega(u,f_i)v_i=0$ et $\omega(u,f_i)X-\omega(u,e_i)v_i=0$. Il en résulte que $\omega(u,e_i)^2+\omega(u,f_i)^2=0$ et donc $\omega(u,e_i)\omega(u,f_i)=0$, ceci signifie que $u$ est un élément central ce qui est en contradiction avec $z\not\in\ker\xi$. Par suite, $z\in\ker\xi$.\\
La condition de cocycle, appliquée à $(e_i,f_j)$ et à $(u_i,e_i)$ donne
\begin{align*}
0=\ad\nolimits_{e_i}\xi(f_j)-\ad\nolimits_{f_j}\xi(e_i)=&\omega(e_i,v_j)z\wedge e_j+
\omega(e_i,e_j)v_j\wedge z+\omega(f_j,v_i)z\wedge f_i\\
&\hfill +\omega(f_j,f_i)v_i\wedge z\\
0=\ad\nolimits_u\xi(e_i)=&-\omega(u,v_i)z\wedge f_i-\omega(u,f_i)v_i\wedge z.
\end{align*}
On déduit des relations ci-dessus que, $\omega(e_i,v_j)=\omega(f_j,v_i)=\omega(u,v_i)=0$ et donc $v_1,...,v_r$ sont dans le centre de $\mathcal{H}_n$.
\item [La nullité de la métacourbure.] Si la courbure est nulle alors $\mathcal{H}_n^*=S\oplus[\mathcal{H}_n^*,\mathcal{H}_n^*]$ est de Milnor et d'après le Théorème, la métacourbure serait nulle, si et seulement si, $[\alpha,[\beta,d\gamma]]=0$, pour tous $\alpha,\beta,\gamma\in S$. Comme, pour tout $u\in\mathcal{H}_n$, $\xi(u)=z\wedge J(u)$ avec $J$ antisymétrique et $J(z)=0$, il existe une base orthonormée $(z_1,...,z_{2n-2r},e_1,f_1,...,e_r,f_r)$ de $(\mathbb{R}z)^\perp$ et $a_1,...,a_r\in\mathbb{R}^*$ tels que, pour tout $j=1...,2n-2r$, $J(z_j)=0$, 
$J(e_i)=a_if_i$ et $J(f_i)=-a_ie_i$, pour tout $i=1,...,r$. Soit $(\lambda,\lambda_1,...,\lambda_{2n-2r},\alpha_1,\beta_1,...,\alpha_r,\beta_r)$ la base duale de $(z,z_1,...,z_{2n-2r},e_1,f_1,...,e_r,f_r)$. On peut vérifier que $S$ est engendrée par $\lambda,\lambda_1,...,\lambda_{2n-2r}$, que $\lambda_1,...,\lambda_{2n-2r}$ sont dans le centre de $\mathcal{H}_n^*$ et $d\lambda_1=...=d\lambda_{2n-2r}=0$. La métacourbure est donc nulle, si et seulement si, $[\lambda,[\lambda,d\lambda]]=0$. Pour conclure, il suffit de remarquer que $d\lambda=-\omega$ et que, pour toute $2$-forme $\rho$ invariante à gauche, $[\lambda,\rho](u,v)=\rho(Ju,v)+\rho(u,Jv)$.
\item [La nullité de la classe modulaire.] La forme volume est donnée par
\[\mu=\lambda\wedge\lambda_1\wedge...\wedge\lambda_{2n-2r}\wedge\alpha_1\wedge\beta_1\wedge...\wedge\alpha_r\wedge\beta_r.\]
Toutes les formes $\lambda_1,\lambda_{2n-2r},\alpha_1,\beta_1,...,\alpha_r,\beta_r$ sont fermées et comme \[i_{\xi(u)}\mu=i_{z\wedge J(u)}\mu=i_{J(u)}\left(\lambda_1\wedge...\wedge\lambda_{2n-2r}\wedge\alpha_1\wedge\beta_1\wedge...\wedge\alpha_r\wedge\beta_r\right),\]
pour tout $u\in\mathcal{H}_n$, $d(i_{\xi(u)}\mu)=0$. On déduit la nullité de la classe modulaire, puisque $\mathcal{H}_n$ et $\mathcal{H}_n^*$ sont unimodulaires. 
\end{description}
\end{proof}

\section{Groupes de Lie-Poisson riemanniens déformables en petites dimensions}
Dans cette partie on déterminera, à isomorphisme près, tous les groupes de Lie-Poisson de petites dimensions vérifiant les conditions de déformation de Hawkins.
\subsection{Dimension deux}
Comme toute algèbre de Milnor de dimension $2$ est commutative, un groupe de Lie-Poisson riemannien $(G,\pi,\langle\,,\,\rangle)$ de dimension $2$ vérifie les conditions de Hawkins, si et seulement si, son tenseur de Poisson est trivial.
\subsection{Dimension trois}
Si $\G^*$ est une algèbre de Milnor de dimension $3$, alors il existe un réel non nul $\lambda$ et une base orthonormale $\{e_1^*,e_2^*,e_3^*\}$ de $\G^*$ telle que :
\[[e_1^*,e_2^*]=\lambda e_3^*,\quad[e_1^*,e_3^*]=-\lambda e_2^*,\quad[e_2^*,e_3^*]=0.\]
Un $1$-cocycle $\rho : \G^*\rightarrow\G^*\wedge\G^*$ est nécessairement de la forme :
\begin{equation}
\rho(e_1^*)=a\,e_2^*\wedge e_3^*,\quad\rho(e_2^*)=b\,e_1^*\wedge e_2^*,\quad\rho(e_3^*)=b\,e_1^*\wedge e_3^*.
\end{equation}
On considère maintenant $\G$ muni du crochet associé à $\rho$, et du produit scalaire dual. Le $1$-cocycle associé $\xi : \G\rightarrow\G\wedge\G$ est donné par :
\begin{equation}
\xi(e_1)=0,\quad\xi(e_2)=-\lambda\,e_1\wedge e_3,\quad\xi(e_3)=\lambda\,e_1\wedge e_2,
\end{equation}
où $\{e_1,e_2,e_3\}$ est la base duale de $\{e_1^*,e_2^*,e_3^*\}$. Le crochet de Lie de $\G$ est donné par :
\begin{equation}
[e_1,e_2]=be_2,\quad[e_1,e_3]=be_3,\quad[e_2,e_3]=ae_1.
\end{equation}
L'identité de Jacobi est donnée par :
\[[e_1,[e_2,e_3]]+[e_2,[e_3,e_1]]+[e_3,[e_1,e_2]]=-2abe_1\]
Pour $\mu=e_1^*\wedge e_2^*\wedge e_3^*$ on a 
\[\rho\left(i_{\xi(e_2)}\mu\right)=\lambda\rho(e_2^*)=\lambda b\,e_1^*\wedge e_2^*,\quad\rho\left(i_{\xi(e_3)}\mu\right)=\lambda\rho(e_3^*)=\lambda b\,e_1^*\wedge e_3^*.\]
En conclusion, $\rho$ définit une structure de bigèbre de Lie sur $\G^*$, qui vérifie, si et seulement si :
\begin{equation}
\rho(e_1^*)=a\,e_2^*\wedge e_3^*,\quad\rho(e_2^*)=\rho(e_3^*)=0.
\end{equation} 
À noter que dans ce cas l'algèbre de Lie $\G$ est unimodulaire. La Proposition qui suit résume tout ce qui précède.
\begin{proposition}
Soit $(G,\pi,\langle\,,\,\rangle)$ un groupe de Lie-Poisson riemannien de dimension $3$, connexe et simplement connexe. Soit $(\G,\xi,\langle\,,\,\rangle_e)$ son algèbre de Lie, munie du $1$-cocycle $\xi$, associé à $\pi$, et du produit scalaire $\langle\,,\,\rangle_e$, la valeur de la métrique riemannienne à l'élément neutre. Alors $(G,\pi,\langle\,,\,\rangle)$ vérifie les conditions de Hawkins, si et seulement si, le triplet $(\G,\xi,\langle\,,\,\rangle_e)$ est isomorphe à l'un des triplets :
\begin{enumerate}
\item $(\mathbb{R}^3,\xi_0,\langle\,,\,\rangle_0)$, où $\mathbb{R}^3$ est muni de sa structure d'algèbre de Lie abélienne, $\xi_0$ est donné par
\[\xi_0(e_1)=0,\quad\xi_0(e_2)=-\lambda\,e_1\wedge e_3,\quad\xi_0(e_3)=\lambda\,e_1\wedge e_2\quad(\lambda\neq0)\]
et $\langle\,,\,\rangle_0$ est le produit scalaire euclidien de $\mathbb{R}^3$.
\item $(\mathcal{H}_3,\xi_0,\langle\,,\,\rangle_0)$, où $\mathcal{H}_3$ est l'algèbre de Heisenberg
$\big\{\left(\begin{smallmatrix}
0&x&z\\
0&0&y\\
0&0&0\end{smallmatrix}\right),\ x,y,z\in\mathbb{R}\big\}$, $\xi_0$ est donné par :
\[\xi_0(e_3)=0,\quad\xi_0(e_1)=-\lambda\,e_3\wedge e_2,\quad\xi_0(e_2)=\lambda\,e_3\wedge e_1\quad(\lambda\neq0)\]
et $\langle\,,\,\rangle_0$ est le produit scalaire sur $\mathcal{H}_3$ dont la matrice dans la base $(e_1,e_2,e_3)$ est donnée par $\left(\begin{smallmatrix}1&0&0\\0&1&0\\0&0&a\end{smallmatrix}\right),\ a>0$.
\end{enumerate}
\end{proposition}
On peut intégrer ces bigèbres de Lie et on a alors le Théorème
\begin{theorem}
Soit $(G,\pi,\langle\,,\,\rangle)$ un groupe de Lie-Poisson riemannien de dimension $3$, connexe et simplement connexe. Alors le triplet $(G,\pi,\langle\,,\,\rangle)$ vérifie les conditions de Hawkins, si et seulement si, il est isomorphe à l'un des triplets :
\begin{enumerate}
\item $(\mathbb{R}^3,\pi,\langle\,,\,\rangle)$, où $\mathbb{R}^3$ est muni de sa structure de groupe de Lie abélien, $\langle\,,\,\rangle$ étant le produit scalaire euclidien, et le tenseur de Poisson est donné par
\[\pi=\lambda\,\partial_x\wedge\left(z\partial_y-y\partial_z\right),\quad\lambda\in\mathbb{R}.\]
\item $(H_3,\pi,\langle\,,\,\rangle)$, où $H_3=\Big\{\left(\begin{smallmatrix}1&x&z\\0&1&y\\0&0&1\end{smallmatrix}\right),\ x,y,z\in\mathbb{R}\Big\}$ (groupe de Heisenberg) et
\[\pi=\lambda\, \left(x\partial_y-y\partial_x\right)\wedge\partial_z,\ \langle\,,\,\rangle=dx^2+dy^2+a\left(dz-xdy\right)^2,\ (\lambda\in\mathbb{R},\ a>0).\]
\end{enumerate}
\end{theorem}
\subsection{Dimension quatre}
Si $\G^*$ est une algèbre de Milnor de dimension $4$, elle admet une base orthogonale $\{e_1^*,e_2^*,e_3^*,e_4^*\}$ telle que
\[[e_1^*,e_3^*]=e_4^*,\quad[e_1^*,e_4^*]=-e_3^*,\]
avec $\|e_2^*\|=\|e_3^*\|=\|e_4^*\|=1$.\\
Un $1$-cocycle $\rho : \G^*\rightarrow\G^*\wedge\G^*$, pour l'action adjointe, qui vérifie est nécessairement de la forme
\begin{equation*}
\begin{cases}\rho(e_1^*)=&\phantom{-}a\,e_3^*\wedge e_4^*\\
\rho(e_2^*)=&\phantom{-}b\,e_3^*\wedge e_4^*\\
\rho(e_3^*)=&\phantom{-}c\,e_1^*\wedge e_3^*+d\,e_2^*\wedge e_3^*+e\,e_2^*\wedge e_4^*\\
\rho(e_4^*)=&\phantom{-}c\,e_1^*\wedge e_4^*-e\,e_2^*\wedge e_3^*+d\,e_2^*\wedge e_4^*
\end{cases}
\end{equation*}
On considère maintenant $\G$ muni du crochet associé à $\rho$, et du produit scalaire dual. Le $1$-cocycle associé $\xi : \G\rightarrow\G\wedge\G$ est donné par :
\begin{equation}
\xi(e_1)=\xi(e_2)=0,\quad\xi(e_3)=-e_1\wedge e_4,\quad\xi(e_4)=e_1\wedge e_3,
\end{equation}
où $\{e_1,e_2,e_3,e_4\}$ est la base duale de $\{e_1^*,e_2^*,e_3^*,e_4^*\}$. Le crochet de Lie de $\G$ est donné par :
\begin{alignat*}{3}
[e_1,e_2]=0,& [e_1,e_3]=ce_3,& [e_1,e_4]=ce_4\\
[e_2,e_3]=de_3-ee_4,& [e_2,e_4]=ee_3+de_4,& [e_3,e_4]=ae_1+be_2.
\end{alignat*}
On a
\begin{align*}
[e_1,[e_2,e_3]]+[e_2,[e_3,e_1]]+[e_3,[e_1,e_2]]=&\phantom{-2a}0\\
[e_1,[e_2,e_4]]+[e_2,[e_4,e_1]]+[e_4,[e_1,e_2]]=&\phantom{-2a}0\\
[e_1,[e_3,e_4]]+[e_3,[e_4,e_1]]+[e_4,[e_1,e_3]]=&-2c[e_3,e_4]\\
[e_2,[e_3,e_4]]+[e_3,[e_4,e_2]]+[e_4,[e_2,e_3]]=&-2d[e_3,e_4].
\end{align*}
L'identité de Jacobi est donc satisfaite, si et seulement si, $a=b=0$ ou $c=d=0$. Comme $\tr\ad_{e_1}=2c$, $\tr\ad_{e_2}=2d$ et $\tr\ad_{e_3}=\tr\ad_{e_4}=0$, $\G$ est unimodulaire, si et seulement si, $c=d=0$.\\ D'autre part, pour $\mu=e_1^*\wedge e_2^*\wedge e_3^*\wedge e_4^*$ on a
\[\rho\left(i_{\xi(e_3)}\mu\right)=c\,e_1^*\wedge e_2^*\wedge e_3^*,\quad\rho\left(i_{\xi(e_4)}\mu\right)=c\,e_1^*\wedge e_2^*\wedge e_4^*.\]
En remplaçant $e_1$ par $\frac{e_1}{\|e_1\|}$ et en renommant les constantes de structures, on obtient
\begin{proposition}
Soit $(G,\pi,\langle\,,\,\rangle)$ un groupe de Lie-Poisson riemannien de dimension $4$, connexe et simplement connexe. Soit $(\G,\xi,\langle\,,\,\rangle_e)$ son algèbre de Lie, munie du $1$-cocycle $\xi$, associé à $\pi$, et du produit scalaire $\langle\,,\,\rangle_e$, la valeur de la métrique riemannienne à l'élément neutre. Alors $(G,\pi,\langle\,,\,\rangle)$ vérifie les conditions de Hawkins, si et seulement si, le triplet $(\G,\xi,\langle\,,\,\rangle_e)$ est isomorphe à $(\mathbb{R}^4,\xi_0,\langle\,,\,\rangle_0)$ où :
\begin{enumerate}
\item dans la base canonique $(e_1,e_2,e_3,e_4)$ le crochet de Lie est donné par
\[[e_1,e_i]=0,\ [e_2,e_3]=ae_3-be_4,\ [e_2,e_4]=be_3+ae_4,\ [e_3,e_4]=ce_1+de_2\]
avec $a=0$ ou $c=d=0$,
\item $\xi_0(e_1)=\xi_0(e_2)=0,\ \xi_0(e_3)=-\lambda\, e_1\wedge e_4,\ \xi_0(e_4)=\lambda\, e_1\wedge e_3,\ (\lambda\neq0)$ et
\item $\langle\,,\,\rangle_0$ est le produit scalaire euclidien de $\mathbb{R}^4$.
\end{enumerate}
\end{proposition}
\begin{remark}
Lorsque $a=0$ les algèbres de Lie ci-dessus sont unimodulaires, qu'on peut intégrer en triplets $(G,\pi,\langle\,,\,\rangle_0)$ vérifiant les conditions de Hawkins. Cependant, même dans le cas non unimodulaire, c'est-à-dire $a\neq0$, nous allons voir que le triplet $(G,\pi,\langle\,,\,\rangle_0)$ associé vérifie les conditions de Hawkins. Ce qui montre que la condition d'unimodularité du groupe, n'est pas nécessaire, en général, pour que la troisième condition de Hawkins soit satisfaite.
\end{remark}
Maintenant, il va falloir identifier les algèbres de Lie, dans la Proposition ci-dessus pour pouvoir les  intégrer. Comme ces algèbres de Lie sont des produits directs, ou semi-directs, d'algèbres de Lie classiques, la détermination des groupes de Lie ne pose pas de difficultés particulières. Ensuite, pour déterminer les tenseurs de Poisson multiplicatifs, à partir des $1$-cocycles, on va utiliser la méthode décrite dans Zung \& Dufour (on peut aussi se référer à l'annexe).
\paragraph{Cas unimodulaire $a=0$.}
\begin{enumerate}
\item Si $b=c=d=0$, alors $(G,\pi,\prs)$ est isomorphe à $(\mathbb{R}^4,\pi_0,\prs_0)$, où $\mathbb{R}^4$ est muni de sa structure de groupe abélien et
\[\pi_0=\lambda\partial_x\wedge\left(t\partial_z-z\partial_t\right),\ (\lambda\in\mathbb{R})\quad\prs_0=dx^2+dy^2+dz^2+dt^2.\]
\item Si $a=0$ et $c\neq0$, alors $(G,\pi,\prs)$ est isomorphe au groupe
\[G=\bigg\{\left(\begin{smallmatrix}x&0&0&0\\0&1&y&t\\0&0&1&z\\0&0&0&1\end{smallmatrix}\right),\ x>0,\ y,z,t\in\mathbb{R}\bigg\},\]
avec le tenseur de Poisson multiplicatif
\[\pi=\left(\partial_t-\alpha x\partial_x\right)\wedge\left(y\partial_z-z\partial_y\right)+\frac{1}{2}\alpha\left(z^2-y^2\right)x\,\partial_x\wedge\partial_t\]
et la métrique riemannienne invariante à gauche
\[\langle\,,\,\rangle=\left(\frac{1}{x}dx+\alpha dt-\alpha ydz\right)^2+dy^2+dz^2+\alpha\left(dt-ydz\right)^2.\]
\item Si $a=0$ et $\beta\neq0$, alors $(G,\pi,\prs)$ est isomorphe au groupe
\[G=\bigg\{\left(\begin{smallmatrix}x&0&0&0\\0&1&y&t\\0&0&1&z\\0&0&0&1\end{smallmatrix}\right),\ x>0,\ y,z,t\in\mathbb{R}\bigg\},\]
avec le tenseur de Poisson multiplicatif
\[\pi=x\partial_x\wedge\left(y\partial_z-z\partial_y\right)+\frac{1}{2}\alpha\left(y^2-z^2\right)x\,\partial_x\wedge\partial_t\]
et la métrique riemannienne invariante à gauche
\[\langle\,,\,\rangle=\frac{1}{x^2}dx^2+dy^2+dz^2+\alpha\left(dt-ydz\right)^2.\]
\item Si $a\neq0$ et $\alpha=\beta=0$, alors $(G,\pi,\prs)$ est isomorphe à $(\mathbb{R}^4,\pi_0,\prs_0)$, où $\mathbb{R}^4$ est muni du produit
\[X\cdot X'=\left(x+x',y+y',z+z'\cos y+t'\sin y,t-z'\sin y+t'\cos y\right),\]
où $X=(x,y,z,t)$ et $X'=(x',y',z',t')$, $\pi_0=\partial_x\wedge\left(z\partial_t-t\partial_z\right)$ et \[\langle\,,\,\rangle=dx^2+a\,dy^2+dz^2+dt^2,\ (a>0).\]
\item Si $b\neq0$, $\alpha\neq0$ et $\beta=0$, alors $(G,\pi,\prs)$ est isomorphe à $(\mathbb{R}^2\times\mathbb{C},\pi_0,\prs_0)$, où $\mathbb{R}^2\times\mathbb{C}$ est muni de la structure de groupe oscillateurs 
\[(t,s,z)\cdot(t',s',z')=\left(t+t',s+s'+\frac{1}{2}\im(\overline{z}\exp(it)z'),z+\exp(it)z'\right),\]
$\pi_0=\partial_s\wedge\left(x\partial_y-y\partial_x\right)$ et 
\[\prs_0=a\,dt^2+b\,ds^2+ds(ydx-xdy)+\frac{1}{4}(ydx-xdy)^2,\ \text{avec}\ a,b>0.\]
\item Si $b\neq0$ et $\beta\neq0$, alors $(G,\pi,\prs)$ est isomorphe à $(\mathbb{R}\times G_0,\pi_0,\prs_0)$, où $\mathbb{R}\times G_0$ est le produit du groupe abélien $\mathbb{R}$ avec $G_0$ qui est soit $SU(2)$ ou $\widetilde{SL(2,\mathbb{R})}$ (le revêtement universel de $SL(2,\mathbb{R})$. Si $\{E_1,E_2,E_3\}$ est une base de l'algèbre de Lie de $G_0$ telle que
\[[E_1,E_2]=E_3,\quad[E_3,E_1]=E_2\quad[E_2,E_3]=\pm E_1,\]
alors $\pi=\partial_t\wedge\left(E_1^+-E_1^-\right)$,
où $E_1^+$ (respectivement $E_1^-$) est le champ invariant à gauche (respectivement à droite) associé à $E_1$. La métrique invariante à gauche $\prs_0$ est donnée par sa valeur en l'identité, dans la base $\{E_0,E_1,E_2,E_3\}$, par la matrice
\[\begin{pmatrix}a&b&0&0\\b&c&0&0\\0&0&d&0\\0&0&0&d\end{pmatrix}.\]
\paragraph{Cas non unimodulaire, $a\neq0$.} Dans ce cas $(G,\pi,\prs)$ est isomorphe à $(\mathbb{R}^
4,\pi_0,\prs_0)$, où $\mathbb{R}^4$ est muni du produit
\[X\cdot X'=\left(x+x',y+y',z+e^{bx}\left(z'\cos(cx)+t'\sin(cx)\right),t+e^{bx}\left(-z'\sin(cx)+t'\cos(cx)\right)\right)\]
et $X=(x,y,z,t)$, $X'=(x',y',z',t')$, $\pi_0=\partial_y\wedge\left(z\partial_t-t\partial_z\right)$ et \[\prs_0=dx^2+dy^2+\exp(-2bx)\left(dz^2+dt^2\right).\]
Le volume riemannien est donné par 
\[\mu=\exp(-2bx)\,dx\wedge dy\wedge dz\wedge dt,\]
et
\[i_\pi\mu=-\exp(-2bx)\left(zdx\wedge dz+tdx\wedge dt\right).\]
Donc $d\left(i_\pi\mu\right)=0$, et la troisième condition de Hawkins est satisfaite.
\end{enumerate}
\subsection{Dimension cinq}
Si $\G^*$ est une algèbre de Milnor non commutative, de dimension $5$, alors son idéal dérivé est de dimension $2$ ou $4$. 
\paragraph{Cas où $\dim[\G^*,\G^*]=2$ :}
Il existe une base orthogonale $\{e_1^*,e_2^*,e_3^*,e_4^*,e_5^*\}$ de $\G^*$ telle que $\|e_i^*\|=1$ pour $i=2,3,4,5$. La structure d'algèbre de Lie est donnée par
\begin{equation}
[e_1^*,e_4^*]=e_5^*,\ [e_1^*,e_5^*]=-e_4^*
\end{equation}
et les autres crochets sont nuls, ou obtenus par antisymétrie.\\
Le cocycle $\xi$ sur $\G$, dual du crochet de Lie $[\,,\,] : \G^*\times\G^*\rightarrow\G^*$ est donné, dans la base duale, par :
\begin{equation*}
\xi(e_1)=\xi(e_2)=\xi(e_3)=0,\ \xi(e_4)=-e_1\wedge e_5,\ \xi(e_5)=e_1\wedge e_4.
\end{equation*}
Soit $\rho : \G^*\rightarrow\wedge^2\G^*$ le $1$-cocycle dual de la bigèbre de Lie $(\G,\xi)$. La condition 
\[\ad\nolimits_{e_1^*}\ad\nolimits_{e_1^*}\rho(e_i^*)=0,\quad i=1,2,3\]
donne, pour $i=1,2,3$
\[\rho(e_i^*)=\alpha_{12}^ie_1^*\wedge e_2^*+\alpha_{13}^ie_1^*\wedge e_3^*+\alpha_{23}^ie_2^*\wedge e_3^*+\alpha_{45}^ie_4^*\wedge e_5^*.\]
La condition de cocycle, $\rho\left([e_i,e_j]\right)=\ad_{e_i}\rho(e_j)-\ad_{e_j}\rho(e_i)$ donne
\begin{equation*}
\left\{\begin{array}{lllll}
\rho(e_i^*)=&a_ie_2^*\wedge e_3^*+b_ie_4^*\wedge e_5^*\ \text{pour tout}\ i=1,2,3\\
\rho(e_4^*)=&Ae_1^*\wedge e_4^*+Be_2^*\wedge e_4^*+Ce_2^*\wedge e_5^*+De_3^*\wedge e_4^*+Ee_3^*\wedge e_5^*\\
\rho(e_5^*)=&Ae_1^*\wedge e_5^*-Ce_2^*\wedge e_4^*+Be_2^*\wedge e_5^*-Ee_3^*\wedge e_4^*+De_3^*\wedge e_5^*.
\end{array}\right.
\end{equation*}
La condition nécessaire d'unimodularité, $d\left(i_{\xi(e_i)}\mu\right)=0,\quad i=1,...,5$, où $\mu$ est la forme volume invariante à gauche $e_1^*\wedge e_2^*\wedge e_3^*\wedge e_4^*\wedge e_5^*$, donne
\begin{equation*}\left\{\begin{array}{cc}
d\left(i_{\xi(e_4)}\mu\right)&=d\left(e_2^*\wedge e_3^*\wedge e_4^*\right)=-\rho(e_2^*)\wedge e_3^*\wedge e_4^*+e_2^*\wedge\rho(e_3^*)\wedge e_4^*-e_2^*\wedge e_3^*\wedge\rho(e_4^*)\\
&=-e_2^*\wedge e_3^*\wedge\rho(e_4^*)=-Ae_1^*\wedge e_2^*\wedge e_3^*\wedge e_4^*\\
d\left(i_{\xi(e_5)}\mu\right)&=d\left(e_2^*\wedge e_3^*\wedge e_5^*\right)=-\rho(e_2^*)\wedge e_3^*\wedge e_5^*+e_2^*\wedge\rho(e_3^*)\wedge e_5^*-e_2^*\wedge e_3^*\wedge\rho(e_5^*)\\
&=-e_2^*\wedge e_3^*\wedge\rho(e_5^*)=-Ae_1^*\wedge e_2^*\wedge e_3^*\wedge e_5*
\end{array}\right.
\end{equation*}
et donc $A=0$ et $\rho$ est donné par
\begin{equation*}
\left\{\begin{array}{lllll}
\rho(e_i^*)=&a_ie_2^*\wedge e_3^*+b_ie_4^*\wedge e_5^*\ \text{pour tout}\ i=1,2,3\\
\rho(e_4^*)=&Ae_2^*\wedge e_4^*+Be_2^*\wedge e_5^*+Ce_3^*\wedge e_4^*+De_3^*\wedge e_5^*\\
\rho(e_5^*)=&-Be_2^*\wedge e_4^*+Ae_2^*\wedge e_5^*-De_3^*\wedge e_4^*+Ce_3^*\wedge e_5^*.
\end{array}\right.
\end{equation*}
La structure d'algèbre de Lie de $\G$ est donnée, dans la base duale, par
\begin{equation*}\begin{array}{lll}
\;[e_2,e_3]=a_1e_1+a_2e_2+a_3e_3,&
\;[e_2,e_4]=Ae_4-Be_5,&
\;[e_2,e_5]=Be_4+Ae_5,\\
\;[e_3,e_4]=Ce_4-De_5,&
\;[e_3,e_5]=De_4+Ce_5,&\;[e_4,e_5]=b_1e_1+b_2e_2+b_3e_3.
\end{array}
\end{equation*}
Examinons la condition de Jacobi :
\begin{equation*}
\left\{\begin{array}{lllll}
J(e_2,e_3,e_4)=&-(a_2A+a_3C)e_4+(a_2B+a_3D)e_5\\
J(e_2,e_3,e_5)=&-(a_2B+a_3D)e_4-(a_2A+a_3C)e_5\\
J(e_2,e_4,e_5)=&\phantom{-}(a_1b_3-2b_1A)e_1+(a_2b_3-2b_2A)e_2+b_3(a_3-2A)e_3\\
J(e_3,e_4,e_5)=&-(a_1b_2+2b_1C)e_1-b_2(a_2+2C)e_2-(a_3b_2+2b_3C)e_3.
\end{array}\right.
\end{equation*}
Les vecteurs, $v_1=\left(\begin{smallmatrix}a_2\\a_3\end{smallmatrix}\right)$, $v_2=\left(\begin{smallmatrix}-C\\\phantom{-}A\end{smallmatrix}\right)$ et $v_3=\left(\begin{smallmatrix}-D\\\phantom{-}B\end{smallmatrix}\right)$ sont liés et l'identité de Jacobi est vérifiée, si et seulement si
\begin{enumerate}
\item $a_1=a_2=a_3=0$, $A=C=0$, c'est-à-dire
\begin{equation*}
[e_4,e_5]=b_1e_1+b_2e_2+b_3e_3.
\end{equation*}
\item $a_2=a_3=0$, $b_1=b_2=b_3=0$, c'est-à-dire
\begin{equation*}\begin{array}{lll}
\;[e_2,e_3]=a_1e_1,&
\;[e_2,e_4]=Ae_4-Be_5,&
\;[e_2,e_5]=Be_4+Ae_5,\\
\;[e_3,e_4]=Ce_4-De_5,&
\;[e_3,e_5]=De_4+Ce_5,&\;[e_4,e_5]=0.
\end{array}
\end{equation*}
\item $a_2=a_3=0$, $b_2=b_3=0$, $A=C=0$, c'est-à-dire
\begin{equation*}\begin{array}{lll}
\;[e_2,e_3]=a_1e_1,&
\;[e_2,e_4]=-Be_5,&
\;[e_2,e_5]=Be_4,\\
\;[e_3,e_4]=-De_5,&
\;[e_3,e_5]=De_4,&\;[e_4,e_5]=b_1e_1.
\end{array}
\end{equation*}
\end{enumerate}

Si $(a_2,a_3)\neq(0,0)$, alors il existe $\alpha,\beta\in\mathbb{R}$ tels que
\[A=\alpha a_3,\ B=\beta a_3,\ C=-\alpha a_2,\ D=-\beta a_2\]
et on a le système
\begin{equation}\label{system}
\left\{\begin{array}{cccc}
(1-2\alpha)a_2b_2=&0,&\ (1-2\alpha)a_3b_3=&0\\
a_2b_3-2\alpha a_3b_2=&0,&\ 2\alpha a_2b_3-a_3b_2=&0\\
a_1b_3-2\alpha a_3b_1=&0,&\ a_1b_2-2\alpha a_2b_1=&0.
\end{array}\right.
\end{equation}
Par suite :
\begin{enumerate}
\item $b_1=b_2=b_3=0$, $A=\alpha a_3,\ B=\beta a_3,\ C=-\alpha a_2,\ D=-\beta a_2$, c'est-à-dire
\begin{equation*}\begin{array}{lll}
\;[e_2,e_3]=a_1e_1+a_2e_2+a_3e_3,&
\;[e_2,e_4]=a_3\alpha e_4-a_3\beta e_5,&
\;[e_2,e_5]=a_3\beta e_4-a_3\alpha e_5,\\
\;[e_3,e_4]=-a_2\alpha e_4+a_2\beta e_5,&
\;[e_3,e_5]=-a_2\beta e_4-a_2\alpha e_5,&\;[e_4,e_5]=0.
\end{array}
\end{equation*}
\item $b_2=b_3=0$, $A=C=0$, $B=\beta a_3,\ D=-\beta a_2$, c'est-à-dire
\begin{equation*}\begin{array}{lll}
\;[e_2,e_3]=a_1e_1+a_2e_2+a_3e_3,&
\;[e_2,e_4]=-a_3\beta e_5,&
\;[e_2,e_5]=a_3\alpha e_5,\\
\;[e_3,e_4]=a_2\beta e_5,&
\;[e_3,e_5]=-a_2\beta e_4,&\;[e_4,e_5]=b_1e_1.
\end{array}
\end{equation*}
\item $A=\frac{1}{2} a_3,\ B=\beta a_3,\ C=-\frac{1}{2} a_2,\ D=-\beta a_2$, $b_i=\gamma a_i$ $(i=1,2,3)$. En effet, si $\alpha=\frac{1}{2}$, le système \eqref{system} devient
\begin{equation*}
\left\{\begin{array}{lll}
a_2b_3-a_3b_2=&0,\\
a_1b_3-a_3b_1=&0,\\
a_1b_2-a_2b_1=&0,
\end{array}\right.
\end{equation*}
c'est-à-dire $\left(\begin{smallmatrix}a_1\\a_2\\a_3\end{smallmatrix}\right)\wedge\left(\begin{smallmatrix}b_1\\b_2\\b_3\end{smallmatrix}\right)=0$
et donc, il existe $\gamma\in\mathbb{R}$ tel que $b_i=\gamma a_i$, $(i=1,2,3)$. Dans ce cas on a
\begin{equation*}\begin{array}{lll}
\;[e_2,e_3]=a_1e_1+a_2e_2+a_3e_3,&
\;[e_2,e_4]=\frac{1}{2}a_3 e_4-a_3\beta e_5,&
\;[e_2,e_5]=a_3\beta e_4-\frac{1}{2}a_3e_5,\\
\;[e_3,e_4]=-\frac{1}{2}a_2e_4+a_2\beta e_5,&
\;[e_3,e_5]=-a_2\beta e_4\frac{1}{2}a_2e_5,&\;[e_4,e_5]=\gamma\left(a_1e_1+a_2e_2+a_3e_3\right).
\end{array}
\end{equation*}
\end{enumerate}
\paragraph{Cas où $\dim[\G^*,\G^*]=4$ :}
Il existe une base orthogonale $\{e_1^*,e_2^*,e_3^*,e_4^*,e_5^*\}$ de $\G^*$ telle que $\|e_i^*\|=1$ pour $i=2,3,4,5$. La structure d'algèbre de Lie est donnée par
\begin{equation}
[e_1^*,e_2^*]=e_3^*,\ [e_1^*,e_3^*]=-e_2^*,\ [e_1^*,e_4^*]=\lambda e_5^*,\ [e_1^*,e_5^*]=-\lambda e_4^*\ (\lambda\neq0)
\end{equation}
et les autres crochets sont nuls, ou obtenus par antisymétrie.\\
Le cocycle $\xi$ sur $\G$, dual du crochet de Lie $[\,,\,] : \G^*\times\G^*\rightarrow\G^*$, est donné, dans la base duale par :
\begin{equation*}
\xi(e_1)=0,\ \xi(e_2)=-e_1\wedge e_3,\ \xi(e_3)=e_1\wedge e_2,\ \xi(e_4)=-\lambda\,e_1\wedge e_5,\ \xi(e_5)=\lambda\,e_1\wedge e_4.
\end{equation*}
Soit $\rho : \G^*\rightarrow\wedge^2\G^*$ le $1$-cocycle dual de la bigèbre de Lie $(\G,\xi)$. La condition 
\[\ad\nolimits_{e_1^*}\ad\nolimits_{e_1^*}\rho(e_1^*)=0,\]
donne
\begin{equation}
\rho(e_1^*)=\left\{\begin{array}{ll}Ae_2^*\wedge e_3^*+Be_2^*\wedge e_4^*+Ce_2^*\wedge e_5^*\mp Ce_3^*\wedge e_4^*\pm Be_3^*\wedge e_5^*+De_4^*\wedge e_5^*,\ \text{si}\ \lambda=\pm1\\
Ae_2^*\wedge e_3^*+Be_4^*\wedge e_5^*,\ \text{sinon}.
\end{array}\right.
\end{equation}
Examinons maintenant la condition nécessaire d'unimodularité :
\[d\left(i_{\xi(e_i)}\mu\right)=0,\quad i=1,...,5,\]
où $\mu$ est la forme volume invariante à gauche $e_1^*\wedge e_2^*\wedge e_3^*\wedge e_4^*\wedge e_5^*$. On a 
\begin{equation*}\left\{\begin{array}{llll}
i_{\xi(e_2)}\mu=&e_2^*\wedge e_4^*\wedge e_5^*\\
i_{\xi(e_3)}\mu=&e_3^*\wedge e_4^*\wedge e_5^*\\
i_{\xi(e_4)}\mu=&\lambda\,e_2^*\wedge e_3^*\wedge e_4^*\\
i_{\xi(e_5)}\mu=&\lambda\,e_2^*\wedge e_3^*\wedge e_5^*.
\end{array}\right.
\end{equation*}
Pour toute $1$-forme invariante à gauche $\alpha$, et pour tout couple de champs invariants à gauche $(X,Y)$, on a $d\alpha\,(X,Y)=\mathscr{L}_X\alpha(Y)-\mathscr{L}_Y\alpha(X)-\alpha([X,Y])=-\alpha([X,Y])$, de sorte que la restriction de la différentielle extérieur aux $1$-formes invariantes à gauche, coïncide avec $-\rho$ ; donc
\begin{equation*}
\left\{\begin{array}{llll}d\left(i_{\xi(e_2)}\mu\right)&=-\rho(e_2^*)\wedge e_4^*\wedge e_5^*+e_2^*\wedge\rho(e_4^*)\wedge e_5^*-e_2^*\wedge e_4^*\wedge\rho(e_5^*)\\
d\left(i_{\xi(e_3)}\mu\right)&=-\rho(e_3^*)\wedge e_4^*\wedge e_5^*+e_3^*\wedge\rho(e_4^*)\wedge e_5^*-e_3^*\wedge e_4^*\wedge\rho(e_5^*)\\
d\left(i_{\xi(e_4)}\mu\right)&=-\lambda\left(\rho(e_2^*)\wedge e_3^*\wedge e_4^*-e_2^*\wedge\rho(e_3^*)\wedge e_4^*+e_2^*\wedge e_3^*\wedge\rho(e_4^*)\right)\\
d\left(i_{\xi(e_5)}\mu\right)&=-\lambda\left(\rho(e_2^*)\wedge e_3^*\wedge e_5^*-e_2^*\wedge\rho(e_3^*)\wedge e_5^*+e_2^*\wedge e_3^*\wedge\rho(e_5^*)\right).
\end{array}\right.
\end{equation*}
On en déduit que
\begin{equation*}
\left\{\begin{array}{lllll}
\rho(e_i^*)(e_1,e_j)=0\ \text{pour tout}\ i,j=2,...,5&\\
\rho(e_2^*)(e_2,e_3)-\rho(e_4^*)(e_3,e_4)-\rho(e_5^*)(e_3,e_5)=&0\\
\rho(e_2^*)(e_2,e_4)+\rho(e_3^*)(e_3,e_4)-\rho(e_5^*)(e_4,e_5)=&0\\
\rho(e_2^*)(e_2,e_5)+\rho(e_3^*)(e_3,e_5)+\rho(e_4^*)(e_4,e_5)=&0\\
\rho(e_3^*)(e_2,e_3)+\rho(e_4^*)(e_2,e_4)+\rho(e_5^*)(e_2,e_5)=&0.
\end{array}\right.
\end{equation*}
D'autre part, la condition de cocycle pour $\rho$ donne $\rho(e_i^*)(e_2,e_3)=\rho(e_i^*)(e_4,e_5)=0$, pour tout $i=2,...,5$, et
\begin{equation*}
\left\{\begin{array}{cc}
\rho(e_2^*)=&-\ad\nolimits_{e_1^*}\rho(e_3^*)\\
\rho(e_3^*)=&\phantom{-}\ad\nolimits_{e_1^*}\rho(e_2^*)\\
\rho(e_4^*)=&-\frac{1}{\lambda}\ad\nolimits_{e_1^*}\rho(e_5^*)\\
\rho(e_5^*)=&\phantom{-}\frac{1}{\lambda}\ad\nolimits_{e_1^*}\rho(e_4^*).
\end{array}\right.
\end{equation*}
On en déduit alors, grâce à l'identité de Jacobi du crochet de Lie de $\G$, que
$\rho(e_i^*)=0$ pour tout $i=2,3,4,5$.\\
\paragraph{Cas $\lambda\not\in\{-1,1\}$}
Dans ce cas $\rho(e_1^*)=a\,e_2^*\wedge e_3^*+b\,e_4^*\wedge e_5^*$ et la structure d'algèbre de Lie de $\G$ est donnée par
\[[e_2,e_3]=ae_1,\quad[e_4,e_5]=be_1.\]
Si $a=b=0$, alors $\G$ est l'algèbre commutative $\mathbb{R}^5$.\\
Si $a=0$ et $b\neq0$ (resp. $a\neq0$ et $b=0$), alors $\G$ est la somme directe de l'algèbre commutative $\mathbb{R}^2$ et l'algèbre de Heisenberg $\mathfrak{h}_3$
\[\G=<e_2,e_3>\oplus<e_1,e_4,e_5>\quad\text{resp.}\ <e_4,e_5>\oplus<e_1,e_2,e_3>\]
Si $(a,b)\neq(0,0)$ alors $\G$ est l'algèbre de Heisenberg $\mathfrak{h}_5$.
\paragraph{Cas $\lambda=-1$}
$\rho(e_1^*)=a\,e_2^*\wedge e_3^*+b\,e_2^*\wedge e_4^*+c\,e_2^*\wedge e_5^*+c\,e_3^*\wedge e_4^*-b\,e_3^*\wedge e_5^*+d\,e_4^*\wedge e_5^*$
et la structure d'algèbre de Lie de $\G$ est donnée par
\begin{alignat*}{3}
[e_2,e_3]=ae_1,&[e_2,e_4]=be_1,&[e_2,e_5]=ce_1\\
[e_3,e_4]=ce_1,&[e_3,e_5]=-be_1,&[e_4,e_5]=de_1.
\end{alignat*}
\paragraph{Cas $\lambda=1$}
$\rho(e_1^*)=a\,e_2^*\wedge e_3^*+b\,e_2^*\wedge e_4^*+c\,e_2^*\wedge e_5^*-c\,e_3^*\wedge e_4^*+b\,e_3^*\wedge e_5^*+d\,e_4^*\wedge e_5^*$
et la structure d'algèbre de Lie de $\G$ est donnée par
\begin{alignat*}{3}
[e_2,e_3]=ae_1,&[e_2,e_4]=be_1,&[e_2,e_5]=ce_1\\
[e_3,e_4]=-ce_1,&[e_3,e_5]=be_1,&[e_4,e_5]=de_1.
\end{alignat*}
\chapter{Annexe}\label{chapter5}
\lettrine[lines=3, nindent=0em]{L}{e} but de cette annexe est d'introduire les objets mathématiques utilisés dans cette thèse. Nous donnons les définitions et les résultats fondamentaux, illustrés de quelques exemples.
\section{Théorie de déformations et cohomologies}
\paragraph{Pourquoi la déformation non commutative ?}~\\
Pour comprendre les origines de la géométrie non commutative, Il faut faire un saut vers la fin du $19^{\text{ème}}$ siècle. 

La physique classique est constituée de la mécanique classique (théorie de la matière, mouvement des particules...), de l'électromagnétisme (théorie des champs) et des équations les gouvernants (équation de Newton, équation de Maxwell). Cette théorie était expérimentalement très bien vérifiée. Toutefois, quelques problèmes surgissaient, en début du $20^{\text{ème}}$ siècle avec l'évolution de la physique expérimentale, notamment les deux problèmes majeurs suivants :
\begin{enumerate}
\item Le spectre des atomes n'est pas continu, comme prédisait la théorie, mais plutôt discret.
\item La vitesse de la lumière est constante, indépendamment du mouvement de l'observateur.
\end{enumerate}
\paragraph{La mécanique quantique} La théorie était initiée et développé par Heisenberg, Bohr, De Broglie, Schrödinger, Dirac et d'autres en $1924$-$1925$. Elle portait des réponses au premier problème en décrivant ``l'état d'un point" par un vecteur unitaire dans un espace de Hilbert. Les coordonnées de cet espace ne commutent plus entre eux, d'après le \emph{principe d'incertitude de Heisenberg} (on ne peut pas mesurer avec précision une coordonnée sans perturber la mesure d'une autre).
Il fallait surmonter la difficulté, de définir une notion de point de l'espace ``quantique''. Dirac a eu l'idée géniale d'interpréter la mécanique quantique dans un formalisme géométrique non commutatif ! Cette idée heuristique était appuyée par le Théorème de Gelfand-Naïmark, qui établissait un pont entre la topologie et l'algèbre. Plus précisément, le Théorème de Gelfand-Naimark reconstitue parfaitement la notion de point parce que l'algèbre $C_0(X)$ (des fonction continues, nulles à l'infini sur un espace topologique
localement compact $X$) est une $C^\star$-algèbre commutative. On a ainsi une bijection
naturelle entre les espaces topologiques localement compacts et les $C^\star$-algèbres
commutatives, de sorte que le spectre de $C_0(X)$, avec une topologie naturelle, est
alors homéomorphe à l'espace $X$. On reconstitue donc non seulement la notion de
point, mais toute la topologie de l'espace.
\paragraph{La théorie de la relativité} Suite aux idées de Poincaré et de Minkowski, Einstein suggéra dans sa théorie de la relativité restreinte en $1905$, un modèle dans lequel l'espace et le temps forment un espace de dimension $4$ et où la signification du temps et de l'espace dépend de l'observateur. Cet espace est muni d'une forme bilinéaire, et les transformations qui laissent cette forme invariante sont \emph{les transformations de Lorentz} et toutes les équations, sur cet espace, doivent être invariantes par ces transformations de Lorentz. Plus tard, en $1912$, Einstein publia sa théorie de relativité générale, où il interpréta l'espace-temps comme variété pseudo-riemannienne et la gravitation comme la courbure de l'espace-temps.

La physique théorique s'est vite confrontée à un problème de taille qui continue, jusqu'à nos jours, à défier l'ensemble des mathématiciens et physiciens : comment fusionner dans un même cadre conceptuel, l'aspect opératoriel de la mécanique quantique (théorie des opérateurs, $C^\star$-algèbre, Théorème de Gelfand-Naimark...) et l'aspect géométrique de la relativité générale (géométrie pseudo-riemannienne, courbure de l'espace-temps, théorie de jauges ou les connexions sur les fibrés principaux...). Deux théories récentes, sont prometteuses :
\begin{description}
\item[Théorie des cordes]Cette théorie qui existe depuis $25$ ans n'est pas capable de faire des prédictions qui peuvent être vérifiées expérimentalement ; ce qui la laisse au rang d'une métathéorie ! Toutefois, Kontsevich s'est inspiré de cette théorie dans son Théorème de formalité, pour montrer que toute variété de Poisson admet une déformation (formelle) par quantification.
\item[Géométrie non commutative] La géométrie non commutative est devenue en quelques années un sujet de recherche très actif, aussi bien en physique théorique qu'en mathématiques.
\end{description}
La géométrie non commutative trouve son origine dans la mécanique quantique qui, dans sa première formulation, apparaît sous forme comparable à la mécanique hamiltonienne où les coordonnées de l'espaces de phases sont remplacées par des opérateurs sur un espace de Hilbert, qui ne commutent pas entre eux. Le principe d'incertitude de Heisenberg affirme qu'on ne peut pas localiser une particule, avec une précision infinie, mais on peut seulement parler de la probabilité de sa présence à un point donné. Ce sont les fonctions d'ondes (les solutions de l'équation de Schrödinger) qui permettent de calculer la probabilité de présence de la particule. Dans le but de fusionner l'aspect opératoriel de la mécanique quantique avec l'aspect géométrique de la théorie de relativité, on a cherché des équations de type Schrödinger qui sont d'ordre $1$ et invariantes par les transformations de Lorentz ; d'où l'introduction de l'opérateur de Dirac. La géométrie non commutative... (à compléter)
La théorie mathématique de la déformation est née avec la physique quantique. L'ensemble des observables de l'espace n'est plus l'algèbre commutative des fonctions numériques ; mais plutôt 

Si $A_0$ est une algèbre, une extension infinitésimale de $A_0$ est une algèbre associative sur l'espace vectoriel $A_0 \oplus A_0\hbar$, où $\hbar$ est un élément central et $\hbar^2=0$ qui coincide avec la structure initiale sur $A_0$ modulo $\hbar$. L'exemple trivial est donné par

$$(a+b\hbar)(c+d\hbar)=ac+(ad+bc)\hbar$$
Ceci revient à multiplier, comme polynômes en $\hbar$, et prendre $\hbar^2=0$.
Les exemples non triviaux sont de la forme 
$$(a+b\hbar)\star_f(c+d\hbar)=ac+(ad+bc+f(a,c))\hbar$$
où $f:A\times A\to A$ est bilinéaire, et vérifie quelques conditions (qui résultent de l'associativité), dites \emph{la condition de $2$-cocycle de Hochschild}.

En géométrie non commutative un \emph{triplet spectral} \index{triplet spectral} est un ensemble de données qui encode la géométrie de manière analytique. La définition comprend généralement un espace de Hilbert, une algèbre d'opérateurs sur cet espace et un opérateur non borné autoadjoint muni de structures supplémentaires. Ces notions ont été introduite par Alain Connes pour généraliser le Théorème de l'indice d'Atiyah-Singer aux espaces non commutatifs. On trouve ces notions avec d'autres appellations comme $K$-cycles non bornés ou \emph{modules de Fredholm}.\\
Un triplet spectral est un triplet $(A,H,D)$, où $H$ est un espace de de Hilbert, $A$ est une algèbre d'opérateurs de $H$ (stable par passage à l'adjoint) et $D$ est un opérateur non borné autoadjoint, à domaine dense tel que
\begin{equation}
\|[a,D]\|<\infty,\quad\text{pour tout}\ a\in A
\end{equation}
\paragraph{Déformation} La théorie de déformation constitue un volet important de l'algèbre homologique.
Soit $\mathcal{A}$ une algèbre (un module sur un anneau $\mathbb{K}$, muni d'un produit bilinéaire, associatif). Une extension de $\mathcal{A}$ est une courte suite exacte
$$0\rightarrow N\rightarrow E\rightarrow\mathcal{A}\to 0$$
d'algèbres (sur le même anneau). Cela veut dire que les flèches représentent des homomorphismes d'algèbres et que l'image de l'une est égale au noyau de la suivante. En particulier, l'homomorphisme de $N$ dans $E$ est injectif et celui de $E$ dans $\mathcal{A}$ est surjectif. 

On identifie alors $N$ à son image qui, étant le noyau d'un homomorphisme, est un idéal de $E$. De plus, $\mathcal{A}$ s'identifie au quotient $E/N$.\\
Une \emph{déformation} d'une algèbre $\mathcal{A}_0$ est une extension de la forme
$$0\rightarrow\hbar\mathbb{A}\rightarrow\mathbb{A}\stackrel{\mathcal{P}}\rightarrow\mathcal{A}_0\to 0$$
où $\hbar$ est un élément central de $\mathcal{A}_0$ et pour tout $a\in\mathbb{A}$,
\[\hbar^2a=0\Longrightarrow a\in\hbar\mathbb{A}.\]
\section{Cohomologie des groupes et des algèbres de Lie}
Soit $G$ un groupe qui agit sur un espace vectoriel $V$,
\[\begin{array}{ccc}
G\times V & \longrightarrow & V\\
(g,x) & \longmapsto & g\cdot x
\end{array}\]
Soit les espaces :
\begin{itemize}
\item [$\ $] $C^0(G,V)=V$,
\item [$\ $] $C^k(G,V)=\{\text{les applications}\ \omega : G\times...\times
G\longrightarrow V\}$.
\end{itemize}
On définit un opérateur $\delta : C^k(G,V)\longrightarrow C^{k+1}(G,V)$, par :
\begin{multline*}
(\delta\omega)(g_1,...,g_{k+1})=g_1\cdot\omega(g_2,...,g_{k+1})+\sum_{i=1}^k
(-1)^k\omega(g_1,...,g_ig_{i+1},...,g_{k+1})\\
+(-1)^{k+1}\omega(g_1,...,g_k).
\end{multline*}
\begin{itemize}
\item[$\bullet$] $(\delta\omega)(g)=g\cdot\omega-\omega,\ $ (pour $k=0$).
\item[$\bullet$] $(\delta\omega)(g,h)=g\cdot\omega(h)-\omega(gh)+\omega(g),\ $
(pour $k=1$).
\end{itemize}
On peut vérifier que $\delta$ est un opérateur de cobord, c'est-à-dire
$\delta\circ\delta=0$, ceci permet de définir les espaces de cohomologie :
\[H^k(G,V)=\frac{\ker\{\delta : C^k(G,V)\to C^{k+1}(G,V)\}}{\im\{\delta :
C^{k-1}(G,V)\to C^k(G,V)\}}\]
On dit que $\omega\in V$ est un $0$-cocycle, pour l'action du groupe $G$ sur
$V$, si
\[g\cdot\omega=\omega,\quad\text{pour tout}\ g\in G.\]
On dit qu'une application $\omega : G\rightarrow V$ est un $1$-cocycle, pour
l'action du groupe $G$ sur $V$, si pour tout $g,h\in G$
\begin{equation}\label{1cocycle}
\omega(gh)=\omega(g)+g\cdot\omega(h)
\end{equation}
et $\omega\in C^k(G,V)$ est un $k$-cocycle, pour l'action du groupe $G$ sur $V$,
si $\delta\omega=0$.
\paragraph{Cohomologie des algèbres de Lie}
Soit $(\G,[,])$ une algèbre de Lie réelle de dimension
finie $n$, et soit $V$ un $\mathbb{R}$-espace vectoriel de dimension
finie. Une représentation de $\G$ est un morphisme
d'algèbres de Lie
\[\rho : \G\rightarrow\gl(V)\]
tel que pour tout $X,Y\in\G$\\
\[\rho\left([X,Y]\right)=\rho(X)\circ\rho(Y)-\rho(Y)\circ\rho(X).\]
On dit aussi que $V$ est un $\G$-module.\\
Une $k$-forme sur $\G$ à valeurs dans $V$ est une
application : $\omega :
\G\times...\times\G\rightarrow V$, $k$-linéaire
alternée. On notera l'ensemble de ces $k$-formes par
$C^k(\G,V)$, avec la convention
$C^0(\G,V)=V$ et
$C^1(\G,V)=\mathscr{L}(\G,V)$.\\
On définit l'opérateur de cobord $\delta :C^k(\G,V)
\rightarrow C^{k+1}(\G,V)$ par
\begin{align*}
(\delta\omega)(X_1,...,X_{k+1})=\sum_{i=1}^{k+1}&(-1)^{i+1}\rho(X_i)\cdot
\omega\left(X_1,\ldots,\hat{X}_i,\ldots,X_{k+1}\right)\\
&+\sum_{i<j}(-1)^{i+j}\omega\left([X_i,X_j],X_1,\ldots,\hat{X}_i,\ldots,
\hat{X}_j,\ldots,X_{k+1}\right)
\end{align*}
\begin{itemize}
\item[$\bullet$] Pour $k=0$, $(\delta\omega)(X)=\rho(X)\cdot\omega$, pour tout
$X\in\G$, $\omega\in V$.
\item[$\bullet$] Pour $k=1$, $(\delta\omega)(X,Y)=\rho(X)\cdot\omega(Y)-\rho(Y)
\cdot\omega(X)-\omega\left([X,Y]\right)$.
\end{itemize}
\begin{proposition}
$\delta$ est un opérateur de cobord, c'est-à-dire $\delta\circ\delta=0$.
\end{proposition}
\begin{proof}
 compléter...
\end{proof}
Un $k$-cocycle est une $k$-forme fermée, c'est-à-dire
$\omega\in C^k(\G,V)$ et
$\delta\omega=0$. On note leur ensemble par $B^k(\G,V)$.\\
Un $k$-cobord est une $k$-forme exacte, c'est-à-dire
$\omega\in C^k(\G,V)$ et il existe
$\eta\in C^{k-1}(\G,V)$ telle que $\omega=\delta\eta$.
On note leur ensemble par $Z^k(\G,V)$.
\begin{lemma}
$B^k(\G,V)$ et $Z^k(\G,V)$ sont des
$\mathbb{R}$-espaces vectoriels, et $Z^k(\G,V)$ est un
sous-espace de $B^k(\G,V)$.
\end{lemma}
\begin{proof}
C'est immédiat. Le deuxième point résulte de la Proposition ci-dessus.
\end{proof}
On appel $k^{\text{ème}}$ groupe de cohomologie de $\G$, à coefficients dans $V$ (ou par rapport à la représentation $\rho$), l'espace
vectoriel quotient
\[H^k(\G,V)=B^k(\G,V)/Z^k(\G,V).\]
Si $G$ est un groupe de Lie connexe d'algèbre de Lie $\G$, alors
l'espace $\Omega_L^*(G)$ des $1$-formes différentielles invariantes à gauche sur $G$
est un sous-complexe du complexe de de Rham de $G$ qui est naturellement
isomorphe au complexe de Chevalley-Eilenberg $C^*(\G,\mathbb{R})$
pour l'action triviale de $\G$ sur $\mathbb{R}$, qui implique que
leurs cohomologies
sont isomorphes :
\begin{equation}\label{action.triviale}
H_L^*(G) \cong H^*(\G,\mathbb{R}).
\end{equation}
(L'isomorphisme de $\Omega_L^*(G)$ sur $C^*(\G,\mathbb{R})$
associe toute $1$-forme invariante à gauche à sa valeur en l'élément neutre $e$
de $G$, après identification de $\G^*$ avec $T_ eG$). En particulier,
lorsque $G$ est compact le processus de la moyenne $\alpha\mapsto\int_G L_g^\alpha\,dg$
($\alpha$ dénote une forme différentielle sur $G$, et $L_g$ dénote la
translation à gauche par $g\in G$) induit un isomorphisme de $H_{dR}^*G$ dans $H_L^*(G)$, et on a $H_{dR}(G)\cong H^*(\G,\mathbb{R})$.
\begin{theorem}[Whitehead1] Si $\G$ est semi-simple et $V$ est un
$\G$-module de dimension finie, alors  $H^1(\G,V)=0$ et
$H^2(\G,V)=0$.
\end{theorem}
\begin{theorem}[Whitehead2] Si $\G$ est semi-simple et $V$ est un
$\G$-module de dimension finie, et $V^{\G}=0$ où
$V^{\G}=\{v\in V,\ \rho(x)\cdot v=0,\ \text{pour tout}\
x\in\G\}$ est l'ensemble des éléments de $V$ qui sont invariants par
l'action de $\G$, alors
\[H^k(\G,V)=0,\quad\text{pour tout}\ k\geq 0.\]
\end{theorem}
\paragraph{Représentation adjointe}
Le cas important, qui nous intéresse, est la cohomologie relative à
la représentation adjointe, c'est-à-dire au morphisme d'algèbres de Lie
\[\begin{array}{cccc}
\ad : & \G & \rightarrow & \gl(\G)\\
     & x & \mapsto & \ad_x,
\end{array}\]
avec
\[\begin{array}{cccc}
\ad_x : & \G & \rightarrow & \G \\
 & y & \mapsto & \ad_x(y)=[x,y]
\end{array}\]
 $V=\wedge^{2}\G$ et $\rho=\ad^{(2)} :
\G\rightarrow\gl(\wedge^2\G)$,
$x\mapsto\ad_x^{(2)}$ o
\[\begin{array}{cccc}
\ad_x^{(2)} : & \wedge^2\G & \rightarrow & \wedge^2\G \\
 & y & \mapsto & \ad_x(y)\wedge z+y\wedge\ad_x(z)
  \end{array}\]
Autrement dit, puisque $\ad_x(y)=[x,y]$ on a
\[\rho(x)(y\wedge z)=[x,y]\wedge z+y\wedge[x,z].\]
\section{Intégration des algèbres et des bigèbres de Lie}
Soit $\G$ une algèbre de Lie réelle de dimension finie et soit $G$ son groupe de Lie connexe et simplement connexe. Soit $\xi : \G\rightarrow\wedge^2\G$ un $1$-cocycle pour la représentation adjointe de $\G$ sur $\wedge^2\G$. Pour tout champ de vecteurs $X\in\mathfrak{X}(G)$ on associé la $1$-forme différentielle $P$ à valeurs dans $\wedge^2\G$, définie par
\begin{equation}
<P,X>(g)=\Ad(g)\cdot\xi\left(L_g*X\right).
\end{equation}
On a 
\begin{align*}
<L_g^*P,X>(h)=&<P,L_{g*}X>(gh)\\
=&\Ad(gh)\cdot\xi\left(L_{gh}^*L_{g*}X\right)\\
=&\Ad(g)\Ad(h)\cdot\xi\left(L_h^*X\right)\\
=&\Ad(g)\cdot<P,X>(h),
\end{align*}
donc $P$ est équivariante, c'est-à-dire
\begin{equation}
L_g^*P=\Ad(g)\cdot P.
\end{equation}
D'autre part, $dP=0$ puisque $\xi$ est un cocycle. En effet,\\
Comme $G$ est simplement connexe, $P$ est exacte et donc il existe une unique fonction $\Lambda$ à valeurs dans $\wedge^2\G$, $\Lambda : G\mapsto\wedge^2\G$, telle que $d\Lambda=P$ et $\Lambda(e)=0$. On définit un champ de bivecteurs par
\begin{equation}
\pi(g)=R_{g*}\Lambda(g)
\end{equation}
D'après le Théorème, $\pi$ est multiplicatif, si et seulement si, $\Lambda$ est un $1$-cocycle pour la représentation adjointe de $G$ sur $\wedge^2\G$, c'est-à-dire
\[\Lambda(gh)=\Lambda(g)+\Ad(g)\cdot\Lambda(h)\]
et ceci est vérifié, puisque $P$ est équivariante.
\chapter{Conclusion}\label{chapter6}
La ``simplicité'' de l'approche algébrique nous a permis de construire des exemples de structures déformables au sens de Hawkins, où l'approche géométrique est avare d'exemples.\\
Toutefois, l'approche algébrique n'est pas entièrement résolue, donc elle peut faire l'objet d'un travail futur. Cette approche peut se résumer comme suit :\\
Soit $G=\mathbb{R}^n$, muni d'une structure de Lie résoluble (non abélienne), d'un tenseur de Poisson multiplicatif $\pi$ et d'une métrique riemannienne $\prs$ invariante à gauche plate. L'algèbre de Lie $\G$ de $G$ est donc de Milnor. On suppose que le $1$-cocycle $\xi=\G\rightarrow\wedge^2\G$ associé à $\pi$ vérifie :
\[\ad\nolimits_x\ad\nolimits_x\xi(z)=0,\ \text{pour tout}\ x,y,z\in S,\]
où $S$ est la sous-algèbre abélienne $S=\{x\in\G\mid \ad_x+\ad_x^t=0\}$. Alors,
\begin{enumerate}
\item Comment caractériser la bigèbre de Lie duale $(\G^*,\rho)$ ?
\item Si de plus, pour une base $\{X_1,...,X_n\}$ de $\G$ on a 
\[\xi\left(i_{\rho(\alpha)}(X_1\wedge...\wedge X_n)\right)=0,\ \text{pour tout}\ \alpha\in\G^*,\]
alors comment caractériser la bigèbre duale $(\G^*,\rho)$ ?
\end{enumerate}
Sur un autre registre, on a vu l'importance du tenseur de métacourbure dans le problème géométrique de déformation et qu'il n'existe pas de formule locale de ce tenseur. D'autre part, comme toute structure d'algèbre de Poisson graduée sur $\Omega^\star(M)$ est déterminée par
des "crochets initiaux"
$$\{f,g\}=\pi(df,dg),\{f,\alpha\}=D_{df}\alpha,\{\alpha,\beta\}=\Psi(\alpha,\beta),$$
où $\pi$ est un bivecteur Poisson, $D$ est une connexion contravariante plate et sans torsion et
$\Psi$ est un opérateur bi-différentiel d'ordre $1$ qui satisfait des condition
assurent l'identité de Jacobi. Il serais donc intéressant, d'étudier la déformation de l'algèbre des formes complètement symétriques (en plus des formes différentielles), comme dans l'article de Mitiric et Vaisman \cite{mi-vai}, pour mieux comprendre le tenseur de métacourbure avec de possibles résultats nouveaux.\\

Notons enfin qu'il sera très intéressant de considérer la déformation non commutative dans le cas des structures de Poisson affines, c'est-à-dire le cas d'un groupe de Lie muni d'un tenseur de Poisson $\pi$ vérifiant
\[\pi(gh)=L_{g*}\pi(h)+R_{h*}\pi(g)+L_{g*}R_{h*}\pi(e)\]
Ces structures ont été introduites est étudiées dans \cite{daso:affine}. Ce sont des structures qui contiennent les groupes de Lie-Poisson et les structures de Poisson invariantes à gauche. Elles possèdent des propriétés simples : leurs feuilles symplectiques sont les orbites des actions d'habillage, leur cohomologie de Poisson se réalise comme la cohomologie d'algèbres de Lie, de leur algèbres de Lie duales. Je suis particulièrement curieux de savoir si les résultats que j'ai obtenu pour les groupes de Lie-Poisson, subsistent dans le cas des structures de Poisson affines munies de métriques riemanniennes invariantes à gauche. Un sujet qui peut être, une suite logique de cette thèse. 
\newpage
\printindex

\end{document}